\documentclass[reqno,10pt]{amsart}
\usepackage{amsmath,amsfonts,amsthm,amssymb}

\usepackage{amsmath}
\usepackage{graphicx}
\usepackage{amsfonts}
\usepackage{amssymb}
\usepackage{amsthm}
\usepackage{dsfont}    
\usepackage{mathrsfs}
\usepackage{amsthm}
\usepackage{verbatim}
\usepackage{esint}
\usepackage{stmaryrd}

\theoremstyle{plain}

\numberwithin{equation}{section}

\newtheorem{thm}{Theorem}[section]

\newtheorem{cor}[thm]{Corollary}
\newtheorem{dfn}[thm]{Definition}
\newtheorem{lemma}[thm]{Lemma}
\newtheorem{prop}[thm]{Proposition}
\newtheorem{rmq}[thm]{Remark}

\newcommand{\R}{\mathbb{R}}

\newcommand{\N}{\mathbb{N}}
\newcommand{\C}{\mathbb{C}}

\newcommand{\ZZZ}{\mathds{Z}}

\newcommand{\del}{\partial}

\newcommand{\ii}{{\rm i}}

\def\hat{\widehat}

\def\bar{\overline}

\begin{document}

\title[Dispersion estimates for the wave equation outside of a cylinder]{Dispersive estimates for the wave equation outside a cylinder in $\mathbb{R}^3$}

\date{}

\author{Felice Iandoli, Oana Ivanovici}
 \address{${}^{*}$Sorbonne Universit\'e, CNRS, LJLL, F-75005 Paris, France}
\email{felice.iandoli@sorbonne-universite.fr, oana.ivanovici@sorbonne-universite.fr}

    \thanks{ The authors were supported by ERC grant ANADEL 757 996.  
    }

\begin{abstract}
We consider the wave equation with Dirichlet boundary conditions in the exterior of a cylinder in $\mathbb{R}^3$ and we construct a global in time parametrix to derive sharp dispersion estimates for all frequencies (low and high) and, as a corollary, Strichartz estimates, all matching the $\mathbb{R}^3$ case. 
\end{abstract}

\maketitle
%

\section{General setting}
We consider the linear wave equation on an exterior domain $\Omega\subset \mathbb{R}^{3} $ with smooth boundary; let $\Delta_D$ be the Laplacian with constant coefficients and Dirichlet boundary conditions,
\begin{equation} \label{WE} 
\left\{ \begin{array}{l}
   (\partial^2_t-
 \Delta_D) u=0,  \;\; \text{ in } \Omega, \\ 
 u|_{t=0} = u_0, \; \partial_t u|_{t=0}=u_1,\quad
 u|_{x=0}=0.
 \end{array} \right.
 \end{equation}
A basic homogeneous (local) estimate says that 
on any smooth Riemannian manifold $(\Omega,g)$ {\it without} boundary, a solution $u$ to the wave equation satisfies (for $T<\infty$)
\begin{equation}\label{SE}
\|u\|_{L^q(0,T) L^r(\Omega)}\leq
C_T \bigl(\,||u_0||_{\dot{H}^{\beta}(\Omega)} +
||u_1||_{\dot{H}^{\beta-1}(\Omega)} \bigr)\,,
\end{equation}
where $\beta=d(\frac 12 -\frac 1r)-\frac 1q$ is dictated by scaling and the pair $(q,r)$ is wave-admissible, i.e such that $\frac{2}{q}+\frac{d-1}{r}\leq\frac{d-1}{2}$ and $(q,r,d)\neq(2,\infty,3)$. 
Here $\dot H^{\beta}(\Omega)$ denotes the homogeneous $L^2$ Sobolev space over $\Omega$. If \eqref{SE} holds for $T = \infty$, Strichartz estimates are said to be global. Such inequalities were established long ago for
Minkowski space (flat metrics) and can be generalized to any smooth Riemannian manifold $(\Omega,g)$  because of their local character (finite propagation speed). They are sharp on every Riemannian manifold $(\Omega,g)$ with $\partial\Omega=\emptyset$.
\vskip1mm

The aforementioned results for ${\mathbb R}^d$ and manifolds without
boundary are now well understood. Euclidean results go back
to R.Strichartz's pioneering work \cite{stri77}, where he proved the particular case $q=r$ for the wave and Schr\"{o}dinger equations. This was later generalized to mixed $L^{q}_{t}L^{r}_{x}$ norms by J.Ginibre and G.Velo \cite{give85} for Schr\"{o}dinger equations, where $(q,r)$ is sharp admissible and $q>2$; wave estimates were obtained by J.Ginibre and G.Velo \cite{GV85,give95}, H.Lindblad and C.Sogge \cite{ls95}, as well as L.Kapitanski for a smooth variable coefficients metric,\cite{lev90}. Endpoint cases for both equations  were finally settled by M.Keel and T.Tao \cite{keta98}. On manifolds without boundary, by finite speed of propagation it suffices to work in coordinate charts and to establish estimates for variable coefficients operators in $\mathbb{R}^d$. For operators with $C^{1,1}$ coefficients, Strichartz estimates were shown by H.Smith \cite{sm98} (see also D.Tataru \cite{tat02} for metrics with $C^{\alpha}$ coefficients).
\vskip1mm

The canonical path leading to such Strichartz estimates is to obtain a
stronger, fixed time, dispersion estimate, which is then combined with
energy conservation, interpolation and a duality argument to obtain
\eqref{SE}. 
If $e^{\pm it\sqrt{-\Delta_{\mathbb{R}^d}}}$ are the
half-wave propagators in $(\mathbb{R}^d,(\delta_{i,j}))$, $\chi\in C_{0}^\infty
(]0,\infty[)$ then the following holds:
\begin{equation}\label{disprdWE}
\|\chi(hD_t)e^{\pm it\sqrt{-\Delta_{\mathbb{R}^d}}}\|_{L^1(\mathbb{R}^d)\rightarrow L^{\infty}(\mathbb{R}^d)}\leq C(d)h^{-d}\min\{1,(h/|t|)^{\frac{d-1}{2}}\}.
\end{equation}
Our aim in the present paper is to prove dispersion for \eqref{WE} when $\partial\Omega$ is a cylinder in $\mathbb{R}^3$ :  a parametrix near diffractive points may be explicitly obtained in a similar way as in \cite{IL} (where the case of the wave and Schr\"odinger equations outside a ball of $\mathbb{R}^3$ was dealt with by the second author and G.Lebeau)  and the diffractive effects in the shadow region are much weaker; however, dealing with the case when both the source and the observation points are located very close to the boundary at a long distance is a real hurdle. In fact, this situation corresponds to rays that remain close to the boundary for a large time interval and propagate near points where the curvature vanishes : to our knowledge, a parametrix near such points, allowing for sharp amplitude estimates, was only constructed in \cite{Meas} inside a cylindrical domain of $\mathbb{R}^3$. However, while in \cite{Meas} the time is bounded (as at the time we did not know to handle the reflections in very large time in the interior case), the parametrix we construct here is global in time, depending on the angle of the initial directions of propagation and on the initial distance of the data to the boundary: different values of these parameters completely modify its construction; dealing with points where the curvature vanishes requires handling separately different situations (involving Hankel and Bessel functions). We expect that in order to deal with general boundaries with no convexity or concavity assumption, and allowing for possibly vanishing curvatures along lower dimensional submanifolds, we need to understand a variety of simple models and the exterior of the cylinder is the first of them after the exterior of a sphere.

Let us provide some details : introducing cylindrical coordinates in $\mathbb{R}^3$, our domain becomes $\Omega=\{(r,\theta,z), r\geq 1, \theta\in [0,2\pi), z\in \mathbb{R}\}$ and $\Delta_D=\frac{\partial^2}{\partial r^2}+\frac{1}{r}\frac{\partial}{\partial_r}+\frac{1}{r^2}\frac{\partial^2}{\partial \theta^2}+\frac{\partial^2}{\partial z^2}$. With $h$ a small parameter and $\tau=h\partial_t/i$, $\eta=h\partial_y/i$, $\xi=h\partial_x/i$, $\vartheta=h\partial_z/i$, the characteristic set of $\partial^2_t-\Delta_D$ is $\tau^2=\xi^2+\frac{1}{r^2}\eta^2+\vartheta^2$ and the boundary is $\{r=1\}$. In \cite{IL}, G.Lebeau and the second author constructed a global in time parametrix for the wave equation outside a ball in $\mathbb{R}^3$, which allowed them to obtain sharp dispersion bounds. In the particular case of \cite{IL} the model domain was $\{(r,\theta,\omega), r\geq 1, \theta\in [0,\pi), \omega \in[0,2\pi)\}$ and the Laplace operator was given by $\Delta_F:
 =\frac{\partial^2}{\partial r^2}+\frac{d-1}{r}\frac{\partial}{\partial_r}+\frac{1}{r^2}(\frac{\partial^2}{\partial \theta^2}+\frac{1}{\sin ^2\theta}\frac{\partial^2}{\partial\omega^2})$. The main difficulty came from rays that hit the boundary without being deviated (corresponding to $\xi=0$, $\eta=1$ and $r$ near $1$; in fact, due to the rotational symmetry, in the exterior of the ball the characteristic equation is $\xi^2+\frac{1}{r^2}\eta^2=\tau^2$) : for this regime, the most efficient tool is  the Melrose-Taylor parametrix (see \cite{zw90}), as it provides us with the form of the solution to \eqref{WE} near diffractive points $\xi=0$, $r=1$ (recall that this parametrix was first used by H.Smith and Ch.Sogge in \cite{smso95} to obtain, in a direct way, local in time sharp Strichartz bounds for waves). In the case of the exterior of a cylinder, the ``diffractive regime" would correspond to $(\eta/\tau)^2+(\vartheta/\tau)^2=1$, $\xi=0$, $r=1$ (instead of $(\eta/\tau)^2=1$, $\xi=0$, $r=1$ of \cite{IL}) : it turns out that when $\vartheta/\tau$ is very close to $1$ the Melrose-Taylor parametrix fails to apply (essentially because one cannot perform any kind of stationary phase arguments anymore in the oscillatory integrals that allow to obtain the form of the solution near the boundary in terms of Airy functions). In particular, the situation $\vartheta/\tau=1$ correspond to rays that (start and) remain close to the boundary for all time and at our knowledge has been encountered only in \cite{Meas} where the author studied dispersive bounds for \eqref{WE} in the interior of a cylindrical domain $\{ (r,\theta,z),r\leq 1, \theta\in [0,2\pi), z\in \mathbb{R}\}\subset \mathbb{R}^3$ with Dirichlet Laplacian $\Delta_D=\partial^2_r+(2-r)\partial^2_{\theta}+\partial^2_z$ (and obtained a ``sharp loss" of $1/4$ due to swallowtail type singularities in the wave front set) ; notice however that in \cite{Meas} the time is bounded 
 so when $\vartheta/\tau$ is close to $1$ the estimates follow easy by Sobolev embedding (and a parametrix is naturally obtained in terms of a spectral sum). In the exterior of a cylinder, our aim is to construct the parametrix globally in time, which makes this situation more difficult (and the case $1-\vartheta/\tau\sim 2^{-j}$ already very delicate when compared to the exterior of a ball). 

 Throughout the rest of the paper $A\lesssim B$ means that there exists a constant $C$ such that $A\leq CB$, such a constant may change from line to line and it is independent of all parameters, and $A\sim B$ means that $B\lesssim A\lesssim B$. We may now state our main results.

\begin{thm}\label{thmdisp3D}
Let $\Theta\subset\mathbb{R}^3$ be the cylinder in $\mathbb{R}^3$ and set $\Omega=\mathbb{R}^3\setminus \Theta$. Let $\Delta_D$ denote the Laplace operator in $\Omega$ with Dirichlet boundary condition and let $\chi\in C^{\infty}_0(0,\infty)$.  The following estimate holds for all $t>0$
 \begin{equation}\label{dispomega3}
\|\chi(hD_t)e^{\pm it\sqrt{-\Delta_D}}\|_{L^1(\Omega)\rightarrow L^{\infty}(\Omega)}\lesssim h^{-3}\min\{1,\frac{h}{t}\}.
\end{equation}
Moreover, let $\chi_0\in C^{\infty}_0(-2,2)$, equal to $1$ on $[0,3/2]$. Then
 $\|\chi_0(D_t)e^{\pm it\sqrt{-\Delta_D}}\|_{L^1(\Omega)\rightarrow L^{\infty}(\Omega)}\lesssim 1/(1+t)$.
 \end{thm}

\begin{thm}\label{thmStri3D}
Under the assumptions of Theorem \ref{thmdisp3D}, Strichartz estimates for the wave flow outside a cylinder in $\mathbb{R}^3$ hold as in the flat case, globally in time.
\end{thm}
Theorem \ref{thmStri3D} follows from \eqref{dispomega3} using the usual $TT^*$ argument and the conservation of energy. In the remaining of this work we focus on the proof of Theorem \ref{thmdisp3D}, first in the high-frequency situation which is by far the most difficult one. The small frequency case will be sketched in the last part.\\

 We recall a classical notion of asymptotic expansion: a function $f(w)$ admits an asymptotic expansion for $w\rightarrow 0$ when there exists a (unique) sequence $(c_{n})_{n}$ such that, for any $n$, $\lim_{w\rightarrow 0} w^{-(n+1)}(f(w)-\sum_{0}^{n} c_{j} w^{j})=c_{n+1}$. We denote $ f(w)\sim_{w} \sum_{n} c_{n} w^{n}$.

\subsubsection{The incoming wave }\label{sectWEOutSC}
Let $\mathbb{D}$ denote the unit disk in $\mathbb{R}^2$ and let $\Theta:=\mathbb{D}\times
\mathbb{R}\subset \mathbb{R}^3$. We set $\Omega:=\mathbb{R}^3\setminus\Theta$, then $\partial\Omega=\mathbb{S}^1\times \mathbb{R}$ is the infinite cylinder. 
We introduce cylindrical coordinates as follows: a point of $Q$ of $\Omega$ with coordinates $(x_1,x_2,x_3)\in \mathbb{R}^3$ is defined by $(r,\theta,z)$ where $r>1$, $\theta\in[0,2\pi)$ and $z\in\mathbb{R}$ and where $x_1=r\cos({\theta}), x_2=r\sin({\theta}), x_3=z$. We also set
$r=1+x$, $x\geq 0$, $y:=\pi/2-\theta$, $\theta\in [0,2\pi)$, $z\in\R$.
In these coordinates, the Laplacian becomes 
\begin{equation}\label{delta}
\Delta=\frac{\partial^2}{\partial x^2}+\frac{1}{(1+x)}\frac{\partial}{\partial_x}+\frac{1}{(1+x)^2}\frac{\partial^2}{\partial y^2}+\frac{\partial^2}{\partial z^2}.
\end{equation}
In the new coordinate system, $x\rightarrow (x,y,z)$ is the ray orthogonal to $\partial\Omega$ at $(0,y,z)\in\partial\Omega$. Any point in $Q\in\Omega$ can be written under the form $Q=(0,y,z)+x\vec{\nu}_{(y,z)}$, where $(y,z)$ is the orthogonal projection of $Q$ on $\partial\Omega$ and $\vec{\nu}_{(y,z)}$ the outward unit normal to $\partial\Omega$ pointing towards $\Omega$. The dual variable to $(x,y,z)$ is denoted $(\xi,\eta,\vartheta)$. The principal symbol of $\partial^2_t-\Delta$ associated to \eqref{delta} is $p(x,\xi,\eta,\vartheta,\tau)=-\tau^2+\xi^2+(1+x)^{-2}\eta^2+\vartheta^2$. The time variable and its dual are $t$ and $\tau$. We let $\mathcal{Q}=\{(x,y,z,t,\xi,\eta,\vartheta,\tau),x=0\}$, $\mathcal{P}=\{(x,y,z,t,\xi,\eta,\vartheta,\tau), p=0\}$.
The cotangent bundle of $\partial\Omega\times\mathbb{R}$ is the quotient of $\mathcal{Q}$ by the action of translation in $\xi$, and we take as coordinates $(y,z,t,\eta,\vartheta,\tau)$. A point $(y,z,t,\eta,\vartheta,\tau)\in T^{*}(\partial\Omega\times\mathbb{R})$ is classified as one of three distinct types:
it is said to be \emph{hyperbolic} if there are two distinct nonzero real solutions $\xi$ to $p|_{x=0}=0$. These two solutions yield two distinct bicharacteristics, one of which enters $\Omega$ as $t$ increases (the \emph{incoming ray}) and one which exits $\Omega$ as $t$ increases (the \emph{outgoing ray}). The point is \emph{elliptic} if there are no real solutions $\xi$ to $p|_{x=0}=0$. In the remaining case $\tau^2=\eta^2+\vartheta^2$, there is an unique solution $\xi=0$ to $p|_{x=0}=0$ which yields a glancing ray, and the point is said to be a \emph{glancing point}. A glancing ray has exactly second order contact with the boundary if we have in addition $\eta^2\frac{d}{dx}(1+x)^{-2}|_{x=0}=-\eta^2/2<0$, which means if $\eta\neq 0$. We set $\alpha=\eta/\tau$, $\gamma=\vartheta/\tau$ : the glancing condition becomes $\alpha^2+\gamma^2=1$, while the hyperbolic (or elliptic) regime satisfy $1-\alpha^2-\gamma^2>0$ (or $1-\alpha^2-\gamma^2<0$). A point in $T^*(\partial\Omega\times \mathbb{R})$ such that $1\geq \alpha^2>0$ may be a glancing point of order exactly two. When $\alpha=0$, it is a glancing point of order $\infty$ (as, in this case, $H_p^jx=0$ for all $j\geq 1$).
\begin{rmq}
When $1-\gamma^2-\alpha^2\geq 1/16$, then on the boundary $\xi^2/\tau^2=(1-\gamma^2-\alpha^2)\geq 1/16$ in which case the corresponding point in the cotangent bundle is hyperbolic. The proof of Theorem \ref{thmdisp3D} for such points follows as in the case of the half-space, so we will focus on the situation  $1-\gamma^2-\alpha^2\leq 1/16$, when $|\xi/\tau|\lesssim 1/4$.
\end{rmq}

Let $\Delta$ be the Laplacian in $\mathbb{R}^3$, then the solution $u_{free}(Q,Q_0,t)$ to the free wave equation $(\partial^2_t-\Delta) u_{free}=0$ in $\mathbb{R}^3$ with $u_{free}|_{t=0}=\delta_{Q_0}$, $\partial_t u_{free}|_{t=0}=0$, where $\delta_{Q_0}$ is the Dirac distribution at $Q_0\in\mathbb{R}^3$, is given by :
\begin{equation}\label{ufree}
u_{free}(Q,Q_0,t):=\frac{1}{(2\pi)^3}\int e^{i(Q-Q_0)\xi}\cos(t|\xi|)d\xi.
\end{equation}
If $w_{in}(Q,Q_0,\tau):=\widehat{1_{t>0}u_{free}}(Q,Q_0,\tau)$ denotes its Fourier transform in time, then the following holds :
\begin{equation}\label{vfreeball}
w_{in}(Q,Q_0,\tau)=\frac{i \tau}{4\pi}\frac{e^{-i\tau|Q-Q_0|}}{|Q-Q_0|}.
\end{equation}

Consider the equation \eqref{WE} with initial data $(\delta_{Q_0},0)$, where $Q_0\in\Omega$ is an arbitrary point 
\begin{equation} \label{WEOC} 
\left\{ \begin{array}{l}
   (\partial^2_t-
 \Delta_D) u=0  \;\; \text{ in } \Omega\times \mathbb{R}, \\ 
 u|_{t=0}=\delta_{Q_0}, \; \partial_t u|_{t=0}= 0, \quad u|_{\partial\Omega}=0. 
 \end{array} \right.
 \end{equation}
Let $u(Q,Q_0,t)=\cos(t\sqrt{-\Delta_D})(\delta_{Q_0})(Q)$ denote the solution to \eqref{WEOC} : in order to prove Theorem \ref{thmdisp3D} we construct $u$ for all $t$ and then deduce global in time dispersive bounds. We may assume, without loss of generality, that $\text{dist}(Q_0,\partial\Omega)\geq \text{dist}(Q,\partial\Omega)$ : indeed, when this is not the case we can use the symmetry of the Green function to change $Q_0$ and $Q$. We may assume that, in the coordinates $(r,\theta,z)$, the source point is of the form $Q_0=(s,0,0)$, where $s-1>0$ represents the distance from $Q_0$ to the boundary. 
Let $Q$ be an arbitrary point of $\Omega$, then $Q:=(r\cos \theta,r\sin\theta, z)$.
We introduce the distance between $Q$ and $Q_0$ as follows
\begin{equation}\label{Phi}
\tilde\phi(r,\theta,z,s):=|Q-Q_0|=\sqrt{r^2-2sr\cos \theta+s^2+z^2}.
\end{equation}
In the normal coordinates $(x,y,z)$ we have $Q_0=(s-1,\frac{\pi}{2},0)$ and $Q=((1+x)\sin y, (1+x)\cos y,z)$; letting $\phi(x,y,z,s):=\tilde\phi(1+x,\frac{\pi}{2}-y,z,s)$, we have $\phi(x,y,z,s)=\sqrt{(1+x)^2-2s(1+x)\sin y+s^2+z^2}$. The coordinates $(x,y,z)$ will be particularly useful when working near a glancing point; near hyperbolic (or elliptic) points we keep the cylindrical coordinates $(r,\theta,z)$. We will switch them when necessary. 
\vskip1mm

Let $u_{free}$ be given in \eqref{ufree}. 
By finite speed of propagation, for any sufficiently small time $0<t< d(Q_0,\partial\Omega)$, the solution to \eqref{WEOC} in $\Omega$ is just $1_{t>0}u_{free}$, whose Fourier transform equals $w_{in}$. In the following, we decompose $w_{in}$ according to the initial directions of propagation as follows :
let $\psi_0(\beta)$ be a smooth function supported near $1$, equal to $1$ for $1\geq \beta \geq 1/36$, equal to $0$ for $\beta \leq 1/64$ and such that $0\leq \psi_0\leq 1$. Let also $
\psi\in C^{\infty}_0(1/4,4)$ equal to $1$ near $1$ such that $1-\psi_0(\beta)=\sum_{j\geq1}\psi(2^{2j}\beta)$. Write $w_{in}=w_0+\sum_{j\geq 1}w_{j}$,
with
\begin{equation}\label{dew1}
w_0(Q,Q_0,\tau)=\frac{\tau^2}{(2\pi)^2}\frac{i\tau}{4\pi}\int \frac{\psi_0(1-\gamma^2)}{\phi(x,\tilde y,\tilde z,s)}e^{i\tau((y-\tilde y)\alpha+(z-\tilde z)\gamma)}e^{-i\tau\phi(x,\tilde y,\tilde z,s)} d\alpha d\gamma d\tilde y d\tilde z,
\end{equation}
\begin{equation}\label{dew22j}
w_{j}(Q,Q_0,\tau)=\frac{\tau^2}{(2\pi)^2}\frac{i\tau}{4\pi}\int \frac{\psi(2^{2j}(1-\gamma^2))}{\phi(x,\tilde y,\tilde z,s)}e^{i\tau((y-\tilde y)\alpha+(z-\tilde z)\gamma)}e^{-i\tau\phi(x,\tilde y,\tilde z,s)} d\alpha d\gamma d\tilde y d\tilde z.
\end{equation}
Let $\psi_j(\beta):=\psi(2^{2j}\beta)$. We set $u^+_{free}:=1_{t>0}u_{free}=\int e^{i\tau t}w_{in}(Q,Q_0,\tau) d\tau$. Using \eqref{dew1} and \eqref{dew22j}, we decompose as follows $u_{free}=u_{free,0}+\sum_{j\geq 1} u_{free,j}$ and set $u^+_{free,j}:=1_{t>0}u_{free,j}$, where $\mathcal{F}(u^+_{free,j})=w_{j}$. \\

The paper is organized as follows : in Section \ref{sec:cylinder} we consider $h\in (0,h_0)$ for some small $h_0\in(0,1)$ and $s\geq \sqrt{2}$ and we show that, for all $u^+_{free,j}$ with $0\leq j\leq\frac 13\log_2(s/h)$, we may construct the outgoing wave in a similar way to that used in \cite{IL} in the exterior of a ball as each $w_{j}$ hits the obstacle at hyperbolic or glancing points of order exactly $2$. The assumptions on $s$ and $j$ are necessary to construct the reflected waves near glancing points and to make sure that stationary phase methods do apply. In Section \ref{sectdispcyl} we obtain dispersive bounds first for each $j\leq \frac 13\log_2(s/h)$ and show that the sum over $j$ is still bounded as expected. Both Sections \ref{sec:cylinder} and \ref{sectdispcyl} deal separately with the glancing and hyperbolic regimes, and also with the cases $\text{dist}(Q,\partial\Omega)\geq\sqrt{2}-1$ or $\text{dist}(Q,\partial\Omega)\leq\sqrt{2}-1$ as each case needs to be handled in a different way. In Section \ref{secBesHan} we consider $h\in (0,h_0)$ and either $s\leq \sqrt{2}$ or $s\geq \sqrt{2}$ and $j\gtrsim  \frac 13\log_2(s/h)$ : in these cases we cannot construct the reflected waves as before, either because the data is too close to the boundary or because the phase functions of $w_{j}$ don't oscillate anymore. We obtain an explicit parametrix in terms of Bessel and Hankel functions and proceed with the dispersive bounds. In the last Section we explain why the last parametrix still allows to obtain dispersion in the case of small frequencies.

\section{Parametrix for \eqref{WE} when $s\geq \sqrt{2}$, $h\in (0,h_0)$ and $2^{-3j}s/h\gtrsim 1$}\label{sec:cylinder}

We consider the source point to be of the form $Q_0=(s,0,0)$, where $s-1$ represents the distance from $Q_0$ to the $\partial\Omega$. In this section we consider $s\geq\sqrt{2}$. Let $h_0\in(0,1)$ be small and $h\in (0,h_0)$.

\begin{lemma} 
Let $Q_0=(s,0,0)$ with $s\geq \sqrt{2}$ and $j$ such that $2^{-3j}s/h\gtrsim 1$. Then $u_{free,j,h}(\cdot,Q_0,t)$ solves the free wave equation and $u_{free,j,h}(\cdot,Q_0,0)|_{\partial\Omega}=O(h^{\infty})$. Moreover, $u_{free,j,h}(P,Q_0,t)|_{P\in\partial\Omega}=O((h/t)^{\infty})$ for $P=(0,\cdot,z)$ with $|z|\geq 4t$. 
\end{lemma}
\begin{proof}
The first statement follows from the fact that $\Delta$ commutes with $D_z$; for $j$ as above, the second statement follows using non-stationary phase arguments for the phase $\tau(t+(z-\tilde z)\gamma+(y-\tilde y)\alpha-\phi(0,\tilde y,\tilde z,s))$ of $u_{free,j,h}$. If $|z|\geq 4t$, the phase is also non-stationary with respect to $\tau$ which allows to conclude.
\end{proof}
Our goal in this section is to construct, for each $0\leq j\leq \frac{1}{3}\log_2(s/(hM))$, the solution $u_{j}$ to the Dirichlet wave equation on $\Omega$ whose incoming part (before reflection) equals $u_{free,j}$. To do that, we first set
\begin{equation}\label{defUbar}
\underline{u}_{j}(Q,Q_0,t):=
\left\{ \begin{array}{l}
u_{j}(Q,Q_0,t), \text{ if } Q\in\Omega,\\
0, \text{ if } Q\in \overline{\Theta}.
 \end{array} \right.
\end{equation}
Then, using Duhamel formula and with $u_j^+:=1_{t>0}u_j$, $\underline{u}_j$ reads as follows
 \begin{equation}\label{Gbarform}
 \underline{u}_j|_{t>0} =u^+_{free,j}-u^{\#}_j,\quad u^{\#}_j (Q,Q_0,t): 
 =\int_{\partial\Omega} \frac{\partial_{\nu} u^+_j(P,Q_0,t-|Q-P|)}{4\pi |Q-P|}d\sigma (P).
 \end{equation}
Let $h_0\in(0,1)$ small enough and $h\in (0,h_0)$. Let $\chi\in C^{\infty}_0([\frac 12, 2])$ be a smooth cutoff equal to $1$ on $[\frac 34, \frac 32]$ and such that $0\leq \chi\leq 1$.
As we are interested in evaluating $\chi(hD_t)\underline{u}_j(Q,Q_0,t)$, let
\begin{equation}\label{udiezh}
u^+_{free,j,h}:=\chi(hD_t)u^+_{free,j}, \quad u^{\#}_{j,h}:=\chi(hD_t)u^{\#}_j(Q,Q_0,t).
\end{equation}
As the free wave flow $u_{free,j,h}$ satisfies the usual dispersive estimates, 
we are reduced to evaluating the sum over $j\leq \frac 13 \log_2(s/(M))$ of $u^{\#}_{j,h}(Q,Q_0,t)$ (or, when possible, of $\chi(hD_t)u^{+}_j:=u^{+}_{j,h}$). 
Using \eqref{Gbarform} we have
\begin{equation}\label{vhform}
u^{\#}_{j,h}(Q,Q_0,t)
=\int e^{it\tau}\chi(h\tau)
 \int_{P\in\partial\Omega}\mathcal{F}(\partial_\nu u^+_j|_{\partial\Omega})(P,Q_0,\tau)
  \frac{1}{4\pi |Q-P|}e^{-i\tau|Q-P|}
d\sigma(P) d\tau,
\end{equation}
where $\mathcal{F}(\partial_\nu u^+_j|_{\partial\Omega})(P,Q_0,\tau)$ denotes the Fourier transform in time of $\partial_\nu u^+_j|_{\partial\Omega}(P,Q_0,t)$.
\begin{dfn}
For a source point $Q_0$ as above, we define its apparent contour $\mathcal{C}_{Q_0}$ as  the set of points $P\in\partial\Omega$ such that the ray $Q_0P$ is tangent to $\partial\Omega$ : in other words, for $\tilde\phi$ defined in \eqref{Phi}, we have
\[
\mathcal{C}_{Q_0}:=\{P\in\partial\Omega \text{ with coordinates } (1,\theta,z) \text{ such that } \partial_r \tilde\phi(1,\theta,z,s)=0 \}.
\]
As $\partial_r\tilde\phi=(r-s\cos\theta)/\tilde\phi$ cancels at $r=1$ when $\cos\theta=\frac 1s$, we find $\mathcal{C}_{Q_0}:=\{ P=(1,\arccos (1/s),z), z\in \mathbb{R}\}$. In the coordinates $(x,y,z)$ we have $\mathcal{C}_{Q_0}=\{P=(0,y,z), y=\arcsin(1/s)\}$.
In the following we set $\theta_*:=\arccos(1/s)=\frac{\pi}{2}-\arcsin(1/s)$ and $y_*:=\arcsin(1/s)$.
\end{dfn}
\begin{dfn}\label{defjjsh}
Let $h\in (0,h_0)$. We define $j(s,h):=\sup\{j,2^{-3j}s/h\geq 1\}$ so that $2^{-3j(s,h)}s/h\sim 1$.
\end{dfn}
On the support of $\psi_j(1-\gamma^2)$ we have $\sqrt{1-\gamma^2}\sim 2^{-2j}$ and in this section we consider only $0\leq j\leq j(s,h)$. In the following we deal separately with the case $\frac{\alpha}{\sqrt{1-\gamma^2}}$ near $1$, when the possible glancing points have exactly second order contact with the boundary and the case $\frac{\alpha}{\sqrt{1-\gamma^2}}$ outside a small neighborhood of $1$. 

Let $\chi_0\in C^{\infty}_0([-2,2])$ and equal to $1$ on $[-\frac 32,
\frac 32]$, fix $\varepsilon>0$ small enough and set $\chi_{\varepsilon}(\cdot):=\chi_0((\cdot-1)/\varepsilon)$. We let $w_{j,gl}$ be defined by \eqref{dew1}, \eqref{dew22j} with additional cutoff $\chi_{\varepsilon}(\frac{\alpha}{\sqrt{1-\gamma^2}})$ supported for $|\frac{\alpha}{\sqrt{1-\gamma^2}}-1|\leq 2\varepsilon$ .
Define also $w_{j,he}$ as in \eqref{dew1}, \eqref{dew22j} with additional cutoff $1- \chi_{\varepsilon_1}(\frac{\alpha}{\sqrt{1-\gamma^2}})$. Then $u^+_{free,j}=u^+_{free,j,gl}+u^+_{free,j,he}$,
\[
u^+_{free,j,gl}:=\int e^{it\tau} w_{j,gl} d\tau,\quad u^+_{free,j,he}:=\int e^{it\tau} w_{j,he} d\gamma, \quad u^+_{free,j,h}=\chi(hD_t)u^+_{free,j}.
\]

\subsection{The glancing part of $u^{+}_j$ for $0\leq j\leq j(s,h)$}\label{sectgl} 
We construct $u^+_{j,gl}$, then $\partial_{\nu}u^+_{j,gl}$, in order to obtain the "glancing part" of $u^{\#}_j$ from formula \eqref{Gbarform}. For $j=0$, the following result due to Melrose and Taylor holds:
\begin{prop}\label{MT}
Microlocally near a glancing point of exactly second order contact with the boundary there exist smooth phase functions $\iota(x,y,z,\alpha,\gamma)$ and $\zeta(x,y,z,\alpha,\gamma)$ such that $\phi_{\pm}=\iota\pm(-\zeta)^{3/2}$ satisfy the eikonal equation and there exist symbols $a,\,b$ satisfying the transport equation such that, for any parameters $\alpha,\gamma$ in a conic neighborhood of a glancing direction and for $\tau>1$ large enough,
\begin{equation}
G_{\tau}(x,y,z,\alpha,\gamma):=e^{\ii\tau\iota(x,y,z,\alpha,\gamma)}\Big(aA_+(\tau^{2/3}\zeta)+b\tau^{-1/3}A'_+(\tau^{2/3}\zeta)\Big)A^{-1}_+(\tau^{2/3}\zeta_0)
\end{equation} 
satisfies 
$(\tau^2+\Delta)G_{\tau}=e^{\ii\tau\iota(x,y,z,\alpha,\gamma)}\Big(a_{\infty}A_+(\tau^{2/3}\zeta)+b_{\infty}\tau^{-1/3}A'_+(\tau^{2/3}\zeta)\Big)A^{-1}_+(\tau^{2/3}\zeta_0)$,
where the symbols verify $a_{\infty}$, $b_{\infty}\in O(\tau^{-\infty})$ and where we set $\zeta_0=\zeta|_{x=0}$. Moreover, the following properties hold
\begin{itemize}
\item $\iota$ and $\zeta$ are homogeneous of degree $0$ and $-1/3$ and satisfy $\langle d\iota,d\iota\rangle-\zeta\langle d\zeta,d\zeta\rangle=1$,
$\langle d\iota,d\zeta\rangle=0$,
where $\langle\cdot,\cdot\rangle$ is the polarization of $p$; the phase $\zeta_0$ is independent of $y,z$ so that $\zeta_0(\alpha,\gamma)$ vanishes at a glancing direction; the diffractive condition means that $\partial_x\zeta|_{x=0}<0$ near a glancing point ; 
\item the symbols $a(x,y,z,\alpha,\gamma)$ and $b(x,y,z,\alpha,\gamma)$ belong to the class $\mathcal{S}^0_{(1,0)}$  and satisfy the appropriate transport equations. Moreover $a|_{x=0}$ is elliptic at the glancing point with essential support included in a small, conic neighborhood of it, while $b|_{x=0}=0$.
\end{itemize}
\end{prop}
The functions $\iota$ and $\zeta$ of the Melrose-Taylor parametrix solve the system of equations
\begin{equation}\label{MT-para}
\left\{\begin{aligned}
&(\partial_x\iota)^2+\frac{(\partial_y\iota)^2}{(1+x)^2}+(\partial_z\iota)^2-\zeta\Big((\partial_x\zeta)^2+\frac{(\partial_y\zeta)^2}{(1+x)^2}+(\partial_z\zeta)^2\Big)=1,\\
&\partial_x\iota \partial_x\zeta+\frac{\partial_y\iota \partial_y\zeta}{(1+x)^2}+\partial_z\iota \partial_z\zeta=0.
\end{aligned}\right.
\end{equation}
The system \eqref{MT-para} admits the pair of solutions 
$\iota(y,z,\alpha,\gamma)=y\alpha+z\gamma$,   $\zeta(x,\alpha,\gamma)=\alpha^{2/3}\tilde\zeta((1+x)\sqrt{1-\gamma^2}/\alpha)$,
where for $\rho:=(1+x)\frac{\sqrt{1-\gamma^2}}{\alpha}$, $\tilde \zeta$ is the (unique) solution to
$\frac{1}{\rho^2}-\tilde{\zeta}(\rho)[\tilde{\zeta}'(\rho)]^2=1$, $\tilde\zeta(1)=0$.
\begin{lemma}\label{lemzeta}
The equation $-\tilde{\zeta}(\partial_\rho\tilde{\zeta})^2+1/\rho^2=1$, $\tilde{\zeta}(1)=0$ has a unique solution of the form
\begin{equation}\label{tildezet>}
\frac23 (-\tilde{\zeta}(\rho))^{3/2}=\int_{1}^{\rho}\frac{\sqrt{w^2-1}}{w}dw=\sqrt{\rho^2-1}-\arccos\left(\frac{1}{\rho}\right),
\end{equation}
if $\rho>1$, while for $\rho<1$ we have 
\begin{equation}\label{tildezet<}
\frac23\tilde{\zeta}(\rho)^{3/2}=\int_{\rho}^1\frac{\sqrt{1-w^2}}{w}dw=\log [(1+\sqrt{1-\rho^2})/\rho]-\sqrt{1-\rho^2}.
\end{equation}
We note that at $\rho=1$ we have $\tilde{\zeta}=0$ and $\lim_{\rho\rightarrow1} \frac{(-\tilde \zeta)(\rho)}{\rho-1}=2^{1/3}$.
\end{lemma} 
\begin{cor}\label{corMtau} Let $\tilde\psi_0(\beta)\in C^{\infty}[\frac{1}{81},2]$ be a smooth function supported near $1$ and such that $\tilde\psi_0=1$ on the support of $\psi_0$.
Consider the operator $M_{\tau}: \mathcal{E}'(\mathbb{R}^2)\rightarrow  \mathcal{D}(\mathbb{R}^3)$, where $\mathcal{E}'(\mathbb{R}^2)$ is the dual space of $C^{\infty}(\mathbb{R}^2)$, 
\begin{equation*}
M_{\tau}(f)(x,y,z):=\left(\frac{\tau}{2\pi}\right)^{2}\int G_{\tau}(x,y,z,\alpha,\gamma)\tilde\psi_0(1-\gamma^2)\hat{f}(\tau\alpha,\tau\gamma)d\alpha d\gamma.
\end{equation*}
Near the glancing region $(\tau^2+\Delta)M_{\tau}(f)\in O(\tau^{-\infty})
$ (up to the boundary) for all $f\in \mathcal{E}'(\mathbb{R}^2)$.
Moreover, the restriction to the boundary $M_{\tau}(f)|_{\partial\Omega}=:J_{\tau}(f)$ defined by
\begin{equation*}
J_{\tau}(f)(y,z)=\left(\frac{\tau}{2\pi}\right)^2\int e^{\ii\tau(\iota(y,z,\alpha,\gamma)-\tilde{y}\alpha-\tilde{z}\gamma)}a(0,y,z,\alpha,\gamma,\tau)\tilde\psi_0(1-\gamma^2)f(\tilde{y},\tilde{z})d\alpha d\gamma d\tilde{y} d\tilde{z},
\end{equation*}
has a microlocal inverse $J^{-1}_{\tau}$ as $a(x,y,z,\alpha,\gamma,\tau)$ is the elliptic symbol of Proposition \ref{MT}.
\end{cor}
We define the following operator $T_\tau: \mathcal{E}'(\mathbb{R}^2)\rightarrow  \mathcal{D}(\mathbb{R}^3)$ for $F\in \mathcal{E}'(\mathbb{R}^2)$
\begin{equation}\label{T-tau}
T_{\tau}(F)(x,y,z)=\left(\frac{\tau}{2\pi}\right)^2\tau^{1/3}\int{e^{\ii\tau(y\alpha+z\gamma+\frac{\sigma^3}{3}+\sigma\zeta(x,\alpha,\gamma))}}(a+ b\frac{\sigma}{\ii})\tilde\psi_0(1-\gamma^2)\hat{F}(\tau\alpha,\tau\gamma)d\alpha d\gamma.
\end{equation}
According to \cite[Lemma A.2]{SmSo94Duke}, $T_{\tau}$ is an elliptic FIO near a glancing point and $(\tau^2+\Delta)T_{\tau}(F) \in O_{C^{\infty}}(\tau^{-\infty})$. 
\begin{lemma}\label{lemmaF} 
Let $Q_0=(s,0,0)$ with $s\geq\sqrt{2}$, $y_*=\arcsin(1/s)$ and assume
  $\tau>1$ is large enough. Then there exists an unique function $F_{\tau}$ satisfying $w_{0,gl}(x,y,z,\tau)=T_{\tau}(F_{\tau})(x,y,z)$ for $(x,y)$ in a neighborhood of $(0,y_*)$. Moreover, $F_{\tau}$ is explicit and has the following form
\begin{equation}\label{formF}
\hat{F}(\tau\alpha,\tau\gamma)=\tau^{\frac16} e^{-\ii\tau\sqrt{1-\gamma^2}\Gamma_0(\frac{\alpha}{\sqrt{1-\gamma^2}},s)}f(\alpha,\gamma,\tau) \frac{\chi_{\varepsilon}(\frac{\alpha}{\sqrt{1-\gamma^2}})\psi_0(1-\gamma^2)}{(1-\gamma^2)^{5/12}(s^2-1)^{1/4}},
\end{equation} 
where $f(\alpha,\gamma,\tau)$ is an elliptic symbol of order $0$ 
$\psi_0(1-\tau^2)$ is the smooth cutoff from \eqref{dew1}
and $\chi_{\varepsilon}$ is the smooth cut-off introduced to define $w_{0,gl}$. For $|\tilde\alpha-1|\leq 2\varepsilon$, $\Gamma_0(\tilde\alpha,s)=y_*\tilde\alpha+\sqrt{s^2-1}+\frac{(1-\tilde\alpha)^2}{2\sqrt{s^2-1}}(1+O(1-\tilde\alpha))$. 
\end{lemma}
The proof of Lemma \ref{lemmaF} follows exactly as in \cite{IL} as $\sqrt{1-\gamma^2}\geq 1/8$ on the support of $\psi_0$. Our goal is to describe, microlocally near the glancing regime, $u^+_{j,gl,h}:=\chi(hD_t)u^+_{j,gl}(\cdot,t)$ for all $0\leq j\leq j(s,h)$. For $j=0$ :
\begin{prop}\label{struttura}
For $Q=(x,y,z)$ near the glancing region we have 
\begin{equation*}
u^+_{0,gl}(Q,Q_0,t)=\frac{1}{(2\pi)^2}\int e^{\ii t\tau}\big(w_{0,gl}(x,y,z,\tau)-M_{\tau}(J_{\tau}^{-1}(w_{0,gl}|_{\partial\Omega})(x,y,z))\big)d\tau,
\end{equation*}
where, for $F_{\tau}$ provided by Lemma \ref{lemmaF} satisfying $w_{0,gl}(\cdot,\tau)=T_{\tau}(F_{\tau})$,  $M_{\tau}(J_{\tau}^{-1}(w_{0,gl}|_{\partial\Omega}))$ reads as
\begin{equation}\label{MtauJ-1tau}
\left(\frac{\tau}{2\pi}\right)^2\int e^{\ii\tau(y\alpha+z\gamma)}\left(aA_{+}(\tau^{2/3}\zeta)+b\tau^{-1/3}A'_{+}(\tau^{2/3}\zeta)\right) \frac{A(\tau^{2/3}\zeta_0)}{A_{+}(\tau^{2/3}\zeta_0)}\tilde\psi_0(1-\gamma^2)\hat{F}_{\tau}(\tau\alpha,\tau\gamma)d\alpha d\gamma.
\end{equation}
\end{prop}
\begin{cor}
For $P=(0,y,z)\in\partial\Omega$ near $\mathcal{C}_{Q_0}$ we have
\begin{equation}\label{u_xomega2}
\begin{aligned}
\mathcal{F}(\partial_x u^+_{0,gl})(P,Q_0,\tau)&=\left(\frac{\tau}{2\pi}\right)^{2}
\tau^{2/3+1/6} \int e^{i\tau (y\alpha+z\gamma-\sqrt{1-\gamma^2}\Gamma_0(\frac{\alpha}{\sqrt{1-\gamma^2}},s))}\\
&\times  b_{\partial}f 
\frac{\chi_{\varepsilon}(\frac{\alpha}{\sqrt{1-\gamma^2}})\psi_0(1-\gamma^2)}{(s^2-1)^{1/4}}\frac{{(1-\gamma^2)^{-5/12+1/3}}}{A_+(\tau^{2/3}\zeta_0(\alpha,\gamma))}d\alpha d\gamma,
\end{aligned}
\end{equation}
where $\zeta_0(\alpha,\gamma)=\alpha^{2/3}\tilde\zeta(\sqrt{1-\gamma^2}/\alpha)$ with $\tilde\zeta$ defined in Lemma \ref{lemzeta}. For $\tilde \alpha$ near $1$, $ b_{\partial}(y,z,\tilde\alpha,\tau)$ is elliptic of order $0$ in $\tau$ with main contribution $a_0\tilde\alpha^{-1/3}(\partial_{\rho}\tilde\zeta)(\rho)|_{\rho=\frac{1}{\tilde\alpha}}$, $a_0=a|_{x=0}$.
\end{cor}
\begin{proof}
Using \eqref{MtauJ-1tau}, we can compute the normal derivatives of each of the two contributions of $u^+_{0,gl}$ and then take the difference. As such, for $P=(0,y,z)$ near the glancing region, we obtain the following
\begin{equation}\label{u_xomega}
\mathcal{F}(\partial_x u^+_{0,gl})(P,Q_0,\tau)=\left(\frac{\tau}{2\pi}\right)^{2}\int e^{\ii\tau(y\alpha+z\gamma)}\tau^{2/3}(1-\gamma^2)^{1/3}b_{\del}
\left(A'-A_{+}'\frac{A}{A_{+}}\right)(\tau^{2/3}\zeta_0)\hat{F}_{\tau}(\tau\alpha,\tau\gamma)d\alpha d\gamma,
\end{equation}
where $(1-\gamma^2)^{1/3}b_{\partial}(y,z,\alpha,\gamma,\tau)=a(0,y,z,\alpha,\gamma,\tau)(\partial_x\zeta)(0,\alpha,\gamma)+\tau^{-1}\partial_xb(0,y,z,\alpha,\gamma,\tau)$. As $a_0:=a|_{x=0}$ is elliptic and as 
$\partial_x\zeta|_{x=0}=(1-\gamma^2)^{1/3}\frac{1}{\tilde\alpha^{1/3}}(\partial_{\rho}\tilde\zeta)(\rho)|_{\rho=\frac{1}{\tilde\alpha}}$
for $\tilde\alpha=\frac{\alpha}{\sqrt{1-\gamma^2}}$ on the support of $\chi_{\varepsilon}$ then $ b_{\partial}(y,z,\tilde\alpha,\tau):=\frac{a_0}{\tilde\alpha^{1/3}}(\partial_{\rho}\tilde\zeta)(\rho)|_{\rho=\frac{1}{\tilde\alpha}}$ is elliptic, close to $1$ on the support of the symbol.
Replacing $\hat{F}_{\tau}$ by \eqref{formF} and using the Wronskian relation
 $A'(z)A_{+}(z)-A'_{+}(z)A(z)=\ii e^{-\ii \pi/3}$ allows to conclude.
\end{proof}

In the remaining of this section we show that, if $j\leq j(s,h)$, similar integral formulas hold for each $w_{j,gl}$; moreover we explicitly compute the corresponding functions $F_{j,\tau}$ to determine $\mathcal{F}(\partial_x u^+_{j,gl})(P,Q_0,\tau)$, $P\in \partial\Omega$. For $j=0$ we may follow closely the approach in \cite[Section 3.1.1]{IL} (as the glancing order contact is exactly $2$) to provide a detailed proof. Let $1\leq j\leq j(s,h)$ : we use the explicit form of $w_{j,gl}$ and write it as an oscillatory integral involving the Airy function $A(\tau^{2/3}\zeta)$ and its derivative. After the changes of variables $\alpha=\sqrt{1-\gamma^2}\tilde\alpha$ and $\tilde z=\phi(x,\tilde y,0,s) z_1$ in \eqref{dew22j}, $w_{j,gl}(Q,Q_0,\tau)$ becomes
\[
\frac{\tau^2}{(2\pi)^2}\frac{i\tau}{4\pi}\int \frac{\psi_j(1-\gamma^2)}{\sqrt{1+z_1^2}}\chi_{\varepsilon}(\tilde\alpha)\sqrt{1-\gamma^2}e^{i\tau(z\gamma+\sqrt{1-\gamma^2}(y-\tilde y)\tilde\alpha -\phi(x,\tilde y,0,s)(z_1\gamma+\sqrt{1+z_1^2}))} d\tilde\alpha d\gamma d\tilde y d\tilde z_1.
\]
The critical point w.r.t. $z_1$ satisfies $\gamma+\frac{z_1}{\sqrt{1+z_1^2}}=0$ hence $1+z_1^2=\frac{1}{1-\gamma^2}$. Set $z_1=-\sqrt{\frac{w^2}{1-\gamma^2}-1}$, then the phase is stationary w.r.t. $w$ at $w=1$ and at this point, the second order derivative of the phase equals $\tau\phi(x,\tilde y,0,s)\sqrt{1-\gamma^2}\sim 2^{-j}s/h$. As $j\leq j(s,h)$, then $2^{-j}\geq 2^{-j(s,h)}\sim (s/h)^{-1/3}$, hence $2^{-j}s/h\gtrsim (s/h)^{2/3}$ with $s\geq \sqrt{2}$. For $w$ near $1$ the stationary phase applies and the critical value of the phase depending of $z_1$ becomes $-\phi(x,\tilde y,0,s)\sqrt{1-\gamma^2}$. 
For $w\notin [1/\sqrt{2},\sqrt{2}]$ we perform integrations by parts with large parameter $\sim 2^{-j}s/h>(s/h)^{2/3}$. We obtain, modulo $O((h/2^{-j}s)^{\infty})$ contributions,
\begin{equation}\label{wjglform}
w_{j,gl}(Q,Q_0,\tau)=C\tau^{2+1/2}\int \frac{\psi_j(1-\gamma^2)\chi_{\varepsilon}(\tilde\alpha)}{\phi^{1/2}(x,\tilde y,0,s)}(1-\gamma^2)^{\frac 12+\frac 12-\frac 12-\frac 14}e^{i\tau(z\gamma+\sqrt{1-\gamma^2}((y-\tilde y)\tilde\alpha-\phi(x,\tilde y,0,s)))} d\tilde\alpha d\gamma d\tilde y.
\end{equation}
For $\tilde\alpha$ near $1$, we can perform a suitable change of variable w.r.t. $\tilde y$ such that the phase $\tilde y\tilde\alpha+\phi(x,\tilde y,0,s)$ transforms into an Airy type phase function of the form $\sigma^3/3+\sigma\tilde\zeta(\frac{1+x}{\tilde\alpha})+\Gamma_0(\tilde\alpha,s)$, where $\tilde\zeta$ is the function defined in Lemma \ref{lemzeta}. 
Let
$\varphi(x,\tilde y,\tilde\alpha, s):=\tilde y\tilde\alpha+\phi(x,\tilde y,0,s)$.
As $\partial_{\tilde y}\phi(x,\tilde y,0,s)=-\frac{s(1+x)\cos \tilde y}{\phi}$, $\partial_{\tilde y}^2\phi(x,\tilde y,0,s)=\frac{s(1+x)\sin \tilde y-(\partial_{\tilde y}\phi)^2}{\phi}$, then $\partial_{\tilde y}^2\varphi(x,\tilde y,\tilde\alpha,s)=0$ when
$\tilde y= y_*(x):=\arcsin\left(\frac{1+x}{s}\right)$ and there $\partial_x\phi(x,\tilde y,0,s)_{|y_*(x)}=0$ and $\partial_{\tilde y}\phi(x,\tilde y,0,s)_{|y_*(x)}=-(1+x)$.  For $\tilde y$ near $y_*(x)$ there are two critical points $y_{\pm}=y_{\pm}(x,\tilde\alpha)$ satisfying
\begin{equation}\label{ycritpmsF}
s(1+x)\sin(y_\pm)=\tilde\alpha^2\pm\sqrt{s^2-\tilde\alpha^2}\sqrt{(1+x)^2-\tilde\alpha^2},  \phi(x,y_{\pm},0,s)=\sqrt{s^2-\tilde\alpha^2}\mp\sqrt{(1+x)^2-\tilde\alpha^2}.
\end{equation}
\begin{lemma}\label{lemgam0} Let $\tilde y=y_*(x)+\mathtt{Y}$. 
There exists a unique change of variables $\mathtt{Y}\mapsto \sigma$ which is smooth and satisfying $\frac{d\mathtt{Y}}{d{\sigma}}\notin\{0,\infty\}$ such that, for $\tilde \zeta$ given by Lemma \ref{lemzeta}, we have
\begin{equation}\label{fctssfin}
\varphi(x,y_*(x)+\mathtt{Y},\tilde\alpha,s)=\frac{{\sigma}^3}{3}+{\sigma}\tilde\alpha^{2/3}\tilde\zeta(\frac{1+x}{\tilde\alpha})+\Gamma_0(\tilde\alpha,s),
\end{equation} 
where $\Gamma_0(\tilde\alpha,s):=\sqrt{s^2-1}+\arcsin (\frac{1}{s})\tilde\alpha+\frac{(1-\tilde\alpha)^2}{2\sqrt{s^2-1}}(1+O(1-\tilde\alpha)\big)$ for $\tilde \alpha$ near $1$.
\end{lemma}
\begin{proof}
As the phase $\varphi$ has degenerate critical points of order exactly two, it follows from \cite{CFU} that there exists a unique change of variables $\mathtt{Y}\mapsto \sigma$ which is smooth and satisfying $\frac{d\mathtt{Y}}{d{\sigma}}\notin\{0,\infty\}$ and that there exist smooth functions $\zeta^{\#}(x,\tilde\alpha,s)$ and $\Gamma(x,\tilde\alpha,s)$ such that
\begin{equation}\label{fctss}
\varphi(x,y_*(x)+\mathtt{Y},\tilde\alpha,s)=\frac{{\sigma}^3}{3}+{\sigma}\zeta^{\#}(x,\tilde\alpha,s)+\Gamma(x,\tilde\alpha,s).
\end{equation}
As the change of coordinates is regular the critical points $\mathtt{Y}_{\pm}:=y_{\pm}(x,\tilde\alpha)-y_*(x)$ of $\varphi$ must correspond to 
$\sigma_{\pm}=\pm\sqrt{-{\zeta^{\#}}(x,\tilde\alpha,s)}$.
Write $\zeta^{\#}(x,\tilde\alpha,s):=\tilde\alpha^{\frac23}\tilde{\zeta}^{\#}\big(\frac{1+x}{\tilde\alpha},\tilde\alpha,s\big)$. 
We will show that $\tilde\zeta^{\#}$ satisfies the same equation as $\tilde\zeta$ in \eqref{lemzeta}.
As the critical values of the two functions in \eqref{fctss} must coincide, we have
\begin{equation}\label{sum-diff}
\varphi(x,y_*(x)+\mathtt{Y}_{\pm},\tilde\alpha,s)=\mp\frac23{(-\zeta^{\#})}^{\frac32}(x,\tilde\alpha,s)+\Gamma(x,\tilde\alpha,s),
\end{equation}
from which we deduce 
$\frac 43 \tilde\alpha{(-\tilde\zeta^{\#})}^{\frac32}(\frac{1+x}{\tilde\alpha},x,s)=\varphi(x,y_{-},\tilde{\alpha},s)-\varphi(x,y_{+},\tilde\alpha,s)$.
Taking the derivative with respect to $x$ in the last equation yields (with $y_{\pm}=y_*(x)+\mathtt{Y}_{\pm}$)
\begin{equation}\label{eltilde}
\begin{aligned}
2(-\partial_x\tilde\zeta^{\#})(-\tilde\zeta^{\#})^{\frac12}=&\partial_x\phi(x,y_*(x)+\mathtt{Y}_{-},0,s)-\partial_x\phi(x,y_*(x)+\mathtt{Y}_{+},0,s)\\
									  &-\partial_x y_+\partial_y\varphi(x,y_+,\tilde\alpha,s)+\partial_x y_{-} \partial_y\varphi(x,y_-,\tilde\alpha,s).
\end{aligned}\end{equation}
The last two terms in the second line of \eqref{eltilde} vanish as $y_{\pm}$ are the critical points of the function $\varphi$ with respect to $y$ ; for the same reason we have that $\partial_y\phi(x,y_{\pm}(x,\tilde\alpha),0,s)=-\tilde\alpha$. As $\phi(x,y,0,s)$ satisfies the eikonal equation $(\partial_x\phi)^2(x,y,0,s)+\frac{1}{(1+x)^2}(\partial_y\phi)^2(x,y,0,s)=1$,
then $(\partial_x\phi(x,y_{\pm}(x),0,s))^2=1-\frac{\tilde\alpha^2}{(1+x)^2}$.
Moreover, $\partial_x\phi_{|y_{\pm}}=\frac{s}{\phi(x,y_{\pm},0,s)}(\tilde\rho-\sin(y_{\pm}))$ (with $\tilde\rho=\frac{1+x}{s}$) which is non positive in the ``$y_+$ case" and positive in the ``$y_{-}$ case".
Eventually we obtain, using \eqref{eltilde} and the right signs of $\partial_x\phi$, $-\tilde\zeta^{\#}[-\partial_x\tilde\zeta^{\#}]^2=1-\frac{\tilde\alpha^2}{(1+x)^2}$,
which is the same equation as in Lemma \ref{lemzeta} with $\rho=\frac{1+x}{\tilde\alpha}=(1+x)\frac{\sqrt{1-\gamma^2}}{\alpha}$. As the degenerate critical point occurs at $\sigma=0$, hence at $\zeta^{\#}=0$, we deduce by uniqueness of the solution that $\tilde\zeta^{\#}=\tilde\zeta=\tilde\zeta(\frac{1+x}{\tilde\alpha})$.\\

Next, we compute the explicit form of the function $\Gamma(x,\tilde\alpha,s)$. Taking the sum in \eqref{sum-diff} gives $\Gamma(x,\tilde\alpha,s)=\tfrac12(\varphi(x,y_+(x),\tilde\alpha,s)+\varphi(x,y_{-}(x),\tilde{\alpha},s))$ ; taking the derivative w.r.t. $x$ yields $\partial_x\Gamma(x,\tilde\alpha,s)=0$. As such, $\Gamma$ is independent of $x$ and we define $\Gamma_0(\tilde\alpha,s):=\Gamma(0,\tilde\alpha,s)$, then
\[
\Gamma_0(\tilde\alpha,s)=\frac{1}{2}\big((y_++y_-)\tilde\alpha+\phi(0,y_+,0,s)+\phi(0,y_-,0,s)\big),
\]
where $y_{\pm}=y_{\pm}(0)$. 
For small $x\geq 0$ and for $y$ in a neighborhood of $y_*=y_*(0)$, $y$ remains sufficiently close to $y_*(x)$ : shrinking the support if necessary, we may assume $|y-y_*(x)|<1/2$. For $|y_{\pm}-y_*|<1/2$ we may compute, using \eqref{ycritpmsF} with $x=0$, the first approximation of $y_{\pm}$ :  we have
\begin{equation}\label{ypm}
y_{\pm}=\arcsin \Big(\frac{\tilde\alpha^2}{s}\pm\sqrt{1-\tilde\alpha^2}\sqrt{1-\frac{\tilde\alpha^2}{s^2}}\Big), \quad y_*=\arcsin(\frac 1s).
\end{equation}
As $\Gamma_0(\tilde\alpha,s)=\frac 12(\varphi(0,y_+,\tilde\alpha,s)+\varphi(0,y_-,\tilde\alpha,s))$ and $\partial_y\varphi|_{y_{\pm}}=0$ then
$\partial_{\tilde\alpha}\Gamma_0=\frac 12(y_++y_-)+\frac 12 \sum_{\pm}\partial_{\tilde\alpha}y_{\pm}\partial_y\varphi|_{y_{\pm}}=\frac 12(y_++y_-)$. This yields $\Gamma_0(1,s)=\sqrt{s^2-1}+\arcsin \frac{1}{s}$ and $\partial_{\tilde\alpha}\Gamma_0(1,s)=\arcsin(1/s)$. We need the higher order derivatives : using \eqref{ypm}, it follows that $(y_++y_-)$ reads as an asymptotic expansion of even powers 
of $\sqrt{1-\tilde\alpha^2}$ and with main term $\arcsin (\frac{\tilde\alpha^2}{s})$. We find, with $Z_{\pm}=\frac{\tilde\alpha^2}{s}\pm\sqrt{1-\tilde\alpha^2}\sqrt{1-\frac{\tilde\alpha^2}{s^2}}$, $Z_{\pm}|_{\tilde\alpha=1}=\frac{1}{s}$,
\[
\frac 12 \partial_{\tilde\alpha}(y_++y_-)=\frac{\tilde\alpha}{s}\Big(\frac{1}{\sqrt{1-Z_+^2}}+\frac{1}{\sqrt{1-Z_-^2}}\Big)-\frac{\tilde\alpha(s^2+1-2\tilde\alpha^2)}{2s^2\sqrt{1-\tilde\alpha^2}\sqrt{1-\frac{\tilde\alpha^2}{s^2}}}\Big(\frac{1}{\sqrt{1-Z_+^2}}-\frac{1}{\sqrt{1-Z_-^2}}\Big).
\] 
As $\Big(\frac{1}{\sqrt{1-Z_+^2}}-\frac{1}{\sqrt{1-Z_-^2}}\Big)=\frac{Z_+^2-Z_-^2}{\sqrt{1-Z_+^2}\sqrt{1-Z_-^2}(\sqrt{1-Z_+^2}+\sqrt{1-Z_-^2})}$ and $Z_+^2-Z_-^2=4\frac{\tilde\alpha^2}{s}\sqrt{1-\tilde\alpha^2}\sqrt{1-\frac{\tilde\alpha^2}{s^2}}$,
\[
\frac 12 \partial_{\tilde\alpha}(y_++y_-)=\frac{\tilde\alpha}{s}\Big(\frac{1}{\sqrt{1-Z_+^2}}+\frac{1}{\sqrt{1-Z_-^2}}\Big)-\frac{2\tilde\alpha^3(s^2+1-2\tilde\alpha^2)}{s^3\sqrt{1-Z_+^2}\sqrt{1-Z_-^2}(\sqrt{1-Z_+^2}+\sqrt{1-Z_-^2})}.
\]
At $\tilde\alpha=1$ we obtain $\partial^2_{\tilde\alpha}\Gamma_0(1,s)=\frac 12 \partial_{\tilde\alpha}(y_++y_-)|_{\tilde\alpha=1}=\frac{1}{\sqrt{s^2-1}}$. In the same way we notice that all the higher order derivatives of $\Gamma_0$ come with a factor $\frac{1}{\sqrt{s^2-1}}$. The proof is achieved.
\end{proof}
After the changes of coordinates $\tilde y=y_*(x)+\mathtt{Y}$, $\mathtt{Y}\rightarrow \sigma$, $\sigma =(\tau\sqrt{1-\gamma^2})^{-1/3}\tilde \sigma $ we obtain $w_{j,gl}(Q,Q_0,\tau)$ as follows (with $\mathtt{Y}=\mathtt{Y}(\sigma)=\mathtt{Y}((\tau\sqrt{1-\gamma^2})^{-1/3}\tilde\sigma)$)
\[
\tau^{2+\frac 12-\frac 13}\int \frac{\psi_j(1-\gamma^2)(1-\gamma^2)^{\frac 14-\frac 16}\chi_{\varepsilon_1}(\tilde\alpha)}{\phi^{1/2}(x,y_*(x)+\mathtt{Y},0,s)}\frac{d\mathtt{Y}}{d\sigma} e^{i\tau(z\gamma+\sqrt{1-\gamma^2}(y\tilde\alpha-\Gamma_0(\tilde\alpha,s)))} e^{-i(\frac{\tilde\sigma^3}{3}+\tilde\sigma (\tau\sqrt{1-\gamma^2})^{2/3}\tilde\zeta(\frac{1+x}{\tilde\alpha}))}d\tilde\sigma d\tilde\alpha d\gamma.
\]
At this point we let again $\alpha=\sqrt{1-\gamma^2}\tilde\alpha$. Following \cite{CFU}, we integrate by parts in $\tilde\sigma$ and apply the Malgrange theorem to write $w_{j,gl}$ under the form $w_{j,gl}=T_{j,\tau}(F_{j,\tau})$, where the operator $T_{j,\tau}$ has the same phase as $T_{\tau}$ and symbols $a_j,b_j$ which are asymptotic expansions with small parameter $h/2^{-j}s$ and where the function $\hat{F}_{j,\tau}$ has phase $-\tau\sqrt{1-\gamma^2}\Gamma_0(\frac{\alpha}{\sqrt{1-\gamma^2}},s)$ and symbol $\frac{\tau^{\frac 12-\frac 13}}{(s^2-1)^{1/4}}(1-\gamma^2)^{\frac{1}{12}-\frac 12}\psi_j(1-\gamma^2)f_j$, where $f_j$ is an asymptotic expansion with parameter $h/2^{-j}s$. Notice that, if for $j=0$ the powers of $(1-\gamma^2)$ in play no role in \eqref{formF} or in \eqref{u_xomega2} as $\psi_0(1-\gamma^2)$ is supported in $[\frac{1}{64},2]$, for $1\leq j\leq j(s,h)$ it is essential to keep track of them.

\subsection{The "non-glancing" parts of $u^{+}_{j}$, $0\leq j\leq j(s,h)$}\label{sectransparam}
In this section we describe the form of $u^+_{j,he,h}$ whose incoming part equals $u^+_{free,j,he,h}:=\int e^{it\tau}\chi(h\tau)w_{j,he}d\tau$. We obtain as before $w_{j,he}(Q,Q_0,\tau)$ under the form \eqref{wjglform} but where $\chi_{\varepsilon}(\tilde\alpha)$ is now replaced by $(1-\chi_{\varepsilon}(\tilde\alpha))$. The phase $\tau(z\gamma+\sqrt{1-\gamma^2}((y-\tilde y)\tilde\alpha-\phi(x,\tilde y,0,s)))$ has two critical points $y_{\pm}(x,\tilde\alpha)$ satisfying \eqref{ycritpmsF} such that $|y_+(x,\tilde\alpha)-y_-(x,\tilde\alpha)|\gtrsim\varepsilon$, as $\tilde \alpha$ stays away from a fixed neighborhood of $1$ on the support of $1-\chi_{\varepsilon}$ and it is stationary with respect to $\tilde\alpha$ when $\tilde y=y$. 

The stationary phase applies with large parameter $\tau \sqrt{1-\gamma^2}\sim 2^{-j}/h$ and gives, modulo $O((h2^{j})^{\infty})$ terms,
 \begin{equation}\label{dew1he}
w_{j,he}(Q,Q_0,\tau)=\tau^{1+1/2}\int \frac{\psi_j(1-\gamma^2)(1-\chi_{\varepsilon}(\partial_y\phi(x,y,0,s)))}{\sqrt{\phi(x,y,0,s)}}(1-\gamma^2)^{\frac 14-\frac12}\tilde\sigma^{\pm}_{free,j,he}e^{i\tau(z\gamma-\sqrt{1-\gamma^2}\phi(x,y,0,s))}  d\gamma.
\end{equation}
Recall from \eqref{ycritpmsF} that $\phi(x, y_{\pm}(x,\tilde\alpha), 0,s)=\sqrt{s^2-\tilde\alpha^2}\mp \sqrt{(1+x)^2-\tilde\alpha^2}$. 
Here $\tilde\sigma^{\pm}_{free,j,he}$ are classical symbols that read as asymptotic expansion with small parameter $h2^{j}$. Let now $1-\gamma^2=2^{-2j}\varphi^2$, then $\varphi\sim 1$ on the support of $\psi_j(2^{-2j}\varphi^2)=\psi(\varphi^2)$ and $d\gamma/d\varphi\sim 2^{-2j}$. The phase $\tau(z\sqrt{1-2^{-2j}\varphi^2}-2^{-j}\varphi\phi(x,y,0,s))$ is stationary when $2^{-j}(-z)\sim \phi(x,y,0,s)$ and its second order derivative equals $\tau (-z)2^{-2j}/\sqrt{1-2^{-2j}\varphi}^3$. At the critical points $\tau (-z)2^{-2j}\sim 2^{-j}s/h\gtrsim (s/h)^{2/3}$, so the stationary phase yields, modulo $O((h/s)^{\infty})$,
 \begin{equation}\label{dew1henotint}
w_{j,he}(Q,Q_0,\tau)=\tau\tilde\psi\big(\frac{\phi(x,y,0,s)}{2^{-j}(-z)}\big)\frac{(1-\chi_{\varepsilon}(\partial_y\phi(x,y,0,s)))}{\phi(x,y,0,s)}2^{-2j+j/2+j/2}\sigma^{\pm}_{free,j,he}e^{-i\tau\phi(x,y,z,s)},
\end{equation}
where $\tilde\psi$ is a smooth cutoff supported near $1$, equal to $0$ near $0$ and such that $\tilde\psi=1$ on the support of $\psi$. The symbols $\sigma^{\pm}_{free,j,he}$ are asymptotic expansions with main contribution $\tilde\sigma^{\pm}_{free,j,he}$ and small parameter $h2^{j}/s$.

If we denote $\Sigma_{free,j}$ the factor of $e^{-i\tau\phi(x,y, z,s)}$ in \eqref{dew1henotint}, then $u^+_{free,j,he,h}=\int e^{i\tau(t-\phi(x,y,z,s))}\chi(h\tau)\Sigma_{free,j}d\tau$. After the reflection on the boundary, the solution to the wave equation with Dirichlet boundary condition reads as $\int e^{i\tau(t-\phi_R(x,y,z,s))}\chi(h\tau)\Sigma_{R,j} d\tau$, where $\phi_R$ satisfies the eikonal equation \eqref{MT-para} and the boundary condition $\phi_R|_{x=0}=\phi|_{x=0}$ and $\partial_x\phi_R|_{x=0}=-\partial_x\phi|_{x=0}$. The symbol $\Sigma_{R,j}$ is an asymptotic expansion with small parameter $(\tau2^{-j}s)^{-1}$ that reads as $\Sigma_{R,j}(\cdot,\tau)=\sum_{k}\tau^{-k}\Sigma_{R,k}$, where $\Sigma_{R,k}$ solve a system of the transport equations and $\Sigma_{R,j}|_{x=0}=\Sigma_{free,j}|_{x=0}$. 
We obtain $\partial_x u^+_{j,he,h}|_{x=0}=\int e^{i\tau(t-\phi(0,y,z,s))}\chi(h\tau)(-i)\tau \Sigma_{j} (y,z,s,\tau)d\tau$, 
where $\Sigma_j$ is a classical symbol that reads as an asymptotic expansion with small parameters $\tau^{-1}, (\tau2^j s)^{-1}$ and whose main contribution equals $2i\partial_x\phi(0,y,z,s) \Sigma_j |_{x=0}$.
\begin{rmq}\label{rmqtransv}
On the support of $1-\chi_{\varepsilon}$ we have $1-\partial_y\phi|_{x=0}\gtrsim \varepsilon$ : from the eikonal equation, we obtain the following lower bound :
$(\partial_x\phi)^2|_{x=0}=\big(1-(1+x)^{-2}(\partial_y\phi)^2-(\partial_z\phi)^2\big)\Big|_{x=0}\geq c(\varepsilon)$, where $c(\varepsilon)>0$ depends only on $\varepsilon$.
As $\partial_x\phi|_{x=0} =\frac{1-s\sin y}{\phi(0,\theta,z,s)}$, this implies $s|\sin y_*-\sin y|\geq c(\varepsilon)\phi(0,y,z,s)$, where $y_*=\arcsin(1/s)$. 
\end{rmq}
For all $0\leq j\leq j(s,h)$ we eventually find, for all $P=(0,y,z)\in \partial\Omega$,
\begin{equation}\label{partxu1he2-2j}
\partial_xu^+_{j,he,h}(P,Q_0,t)=\int e^{i\tau(t-\phi(0,y,z,s))}\chi(h\tau)\frac{\tau^2}{\psi(0,y,z,s)}2^{-j}\sigma_{j,he}(y,z,s,\tau)d \tau,
\end{equation}
where $\sigma_{2^{-2j},he}$ is an asymptotic expansion with small parameters $\tau^{-1}, (\tau2^{-3j}s)^{-1}$ supported for $s|\sin y_*-\sin y|\geq c(\varepsilon)\phi(0,y,z,s)$ and $2^{-j}(-z)\sim \phi(0,y,0,s)$.

\section{High-frequency case. Dispersive estimates when $d(Q_0,\partial\Omega)\geq \sqrt{2}-1$} \label{sectdispcyl}

\subsection{Dispersion for the glancing part when $d(Q,\partial\Omega)\geq \sqrt{2}-1$} 
Let $Q_0=(s,0,0)$, $Q=((1+x_Q)\sin y_Q,(1+x_Q)\cos y_Q,z_Q)$ in $\Omega$, and assume $s\geq r:=1+x_Q\geq \sqrt{2}$.
We prove the following :
\begin{prop}\label{propglanc}
There exists $C>0$ such that for all $t>h$, the following holds uniformly with respect to $Q,Q_0$ such that  $s\geq r\geq \sqrt{2}$ where $s=1+x_{Q_0}$, $r=1+x_Q$: $\sum_{0\leq j\leq j(s,h)} |u^{\#}_{j,gl,h}(Q,Q_0,t)|\leq \frac{C}{h^2t}$.
\end{prop}
\begin{proof}
We write the details of the proof for $j=0$ while keeping track of the factors $\sqrt{1-\gamma^{2}}$. The proof of dispersive bounds for $1\leq j\leq j(s,h)$ will follows exactly in the same way as all stationary arguments follow for such values of $j$ and we will be able to sum up all the contributions as these bounds have additional non-positive powers of $2^j$. Let $j=0$ and set $ I_{0,gl}(Q,Q_0,\tau):=\int_{P\in\partial\Omega}\frac{\mathcal{F}(\partial_x u^+_{j,gl})(P,Q_0,\tau)}{4\pi\tau |P-Q|}e^{-i\tau|P-Q|}d\sigma(P)$. Then
\begin{equation}\label{totale}
u^{\#}_{0,gl,h}(Q,Q_0,t)=\frac{1}{4\pi} \int \chi(h\tau)e^{\ii t\tau} \tau
 I_{0,gl}(Q,Q_0,\tau) d\tau.
\end{equation}
Writing $|P-Q|=\phi(x_Q,y-y_Q,z-z_Q,1)$ for a point $P=(\sin y, \cos y, z)$ on the boundary $\partial\Omega$ and replacing  \eqref{u_xomega2} in \eqref{vhform} we find, after the change of coordinates $\alpha=\sqrt{1-\gamma^2}\tilde\alpha$,
\begin{equation}\label{IQQ0}
\begin{aligned}
I_{0,gl}(Q,Q_0,\tau)=
 \int \tau^{-1+2+\frac 56}e^{\ii\tau (z\gamma+\sqrt{1-\gamma^2}(y\tilde\alpha-\Gamma_0(\tilde\alpha,s))-\phi(x_Q,y-y_Q,z-z_Q,1))}\\
\times  \frac{f(\alpha,\gamma,\tau) b_{\del}(y,z,\tilde\alpha,\tau)}{\phi(x_Q,y-y_Q,z-z_Q,1)} \tfrac{(1-\gamma^2)^{-\frac{5}{12}+\frac 13+\frac 12}\psi_0(1-\gamma^2)\chi_{\varepsilon}(\tilde\alpha)}{(s^2-1)^{1/4}A_{+}(\tau^{\frac23}\zeta_0(\alpha,\gamma))}d\tilde\alpha d\gamma  dy dz.
\end{aligned}
\end{equation}
\begin{lemma}\label{lemI1gl}
There exists a constant $C>0$ such that $|I_{0,gl}(Q,Q_0,\tau)|\leq C/t $ 
uniformly with respect to $Q,Q_0$ and $t$ such that $\sqrt{s^2-1+z_Q^2}\sim t$.
Moreover, for $\frac{t}{\sqrt{s^2-1+z_Q^2}}\notin[1/4,4]$, we have $|I_{0,gl}(Q,Q_0,\tau)|\leq \frac{C}{\sqrt{s^2-1}}$.
\end{lemma}
If $\frac{t}{\sqrt{s^2-1+z_Q^2}}\in[1/4,4]$, the estimate of Proposition \ref{propglanc} follows using \eqref{totale} and Lemma \ref{lemI1gl}. If not,
the phase of \eqref{totale} is not stationary w.r.t. $\tau$ and we proceed by integrations by parts which give at most $O(h^{\infty}/t)$.
\end{proof}
\begin{proof}(Proof of Lemma \ref{lemI1gl}) We apply the stationary phase with respect to $z$ in the integral \eqref{IQQ0}: let $r=1+x_Q$ and set $z=z_Q+\tilde{z}\sqrt{1+r^2-2r\cos(y_Q-y)}$. As $r\geq \sqrt{2}$, this is well defined and $dz/d\tilde z=\phi(x_Q,y-y_Q,0,1)$. As $\phi(x_Q,y-y_Q,z-z_Q,1)=\phi(x_Q,y-y_Q,0,1)\sqrt{1+\tilde z^2}$ the phase of $I_{0,gl}$ becomes $\tau(z_Q\gamma-\sqrt{1-\gamma^2}(-y\tilde\alpha+\Gamma_0(\tilde\alpha,s))+\phi(x_Q,y-y_Q,0,1)(\tilde{z}\gamma-\sqrt{1+\tilde{z}^2}))$ and its critical point with respect to $\tilde z$ satisfies $\tilde{z}=\frac{\gamma}{\sqrt{1-\gamma^2}}$. As, in case $j$ large, this value is large, we renormalize $\tilde z$ by taking $\tilde z=\sqrt{\frac{w^2}{1-\gamma^2}-1}$; as such, the critical point is $w=1$ and the second order derivative of the phase equals $\tau \phi(x_Q,y-y_Q,0,1)\sqrt{1-\gamma^2}$. The stationary phase in $w$ yields a factor $\tau^{-1/2}\times(1-\gamma^2)^{-\frac12-\frac14}$ and the symbol $\tau^{1+5/6}\frac{b_{\partial}(1-\gamma^2)^{5/12}}{\phi(x_Q,y-y_Q,z-z_Q,1)}$ becomes $\tau^{1+1/3}(1-\gamma^2)^{\frac{5}{12}-\frac 34}\frac{\tilde b_{\partial}(y,z_Q,\alpha,\gamma,\tau)}{\phi^{1/2}(x_Q,y-y_Q,0,1)}$, where $\tilde b_{\partial}$ has main contribution $b_{\partial}$. We obtain
\begin{equation}\label{I1glbeforespha}
\begin{aligned}
I_{0,gl}(Q,Q_0,\tau)=&\tau^{\frac43}\int e^{\ii\tau(z_Q\gamma-\sqrt{1-\gamma^2}(-y\tilde\alpha+\Gamma_0(\tilde\alpha,s)+\phi(x_Q,y-y_Q,0,1)))}\\
&(1-\gamma^2)^{-1/3}\frac{\tilde b_{\partial}(y,Q,\tilde\alpha,\gamma,\tau)}{\phi^{1/2}(x_Q,y-y_Q,0,1)}\frac{f(\alpha,\gamma,\tau)\psi_0(1-\gamma^2) \chi_{\varepsilon}(\tilde\alpha)}{(s^2-1)^{1/4}A_{+}(\tau^{\frac23}\zeta_0(\alpha,\gamma))}
d\tilde\alpha d\gamma dy.
\end{aligned}
\end{equation}
The phase $\phi(x_Q,y-y_Q,0,1)$ has two degenerate critical points of order exactly two at $y=y_Q\pm\arccos(1/r)$, where $r=1+x_Q$. Near $y_Q-\arccos(1/r)$, its first order derivative equals $-1$, hence for $y$ near this point the phase of $I_{0,gl}$ is non-stationary w.r.t. $y$ and repeated integrations by parts yield $O(\frac{\tau^{-\infty}}{\sqrt{s^2-1}})$. Let $y_c:=y_Q+\arccos(1/r)$.
Notice that, if $y\in [0,2\pi)$ is sufficiently close to $y_*$ on the support of $I_{0,gl}$ (say $|y-y_*|\leq \frac{\pi}{16}$) and is such that $|y-y_c|\geq \frac{\pi}{8}$, then $1-\tilde\alpha$ has to be bounded from below by a fixed constant there where the phase of $I_{0,gl}$ is stationary w.r.t. $y$. Taking $\varepsilon$ smaller if necessary, it follows that for such value of $y$ outside a small, fixed neighborhood of $y_c$, $\tilde\alpha$ cannot belong to the support of $\chi_{\varepsilon}(\tilde\alpha)$. We are reduced to studying the integral \eqref{I1glbeforespha} for $|y-y_c|\leq \frac{\pi}{8}<1$. Let $\epsilon_1>0$ be small enough. We study separately the cases $|y-y_c|\leq \tau^{-1/3+\epsilon_1}$ and $\tau^{-1/3+\epsilon_1}\lesssim |y-y_c|\leq\frac{\pi}{8}$ ; to do that, we introduce a smooth cut-off $\chi_0$ supported in $[-2,2]$ and equal to $1$ on $[-3/2,3/2]$ and split $I_{0,gl}=I_{0,gl}^{\chi_0}+I_{0,gl}^{1-\chi_0}$, where $I_{0,gl}^{\chi}$ has the form \eqref{I1glbeforespha} with additional cut-off $\chi((y-y_c)\tau^{1/3-\epsilon_1})$.

\subsubsection{Case $\tau^{-1/3+\epsilon_1}\leq |y-y_c|\leq \frac{\pi}{8}$ : study of $I_{1,gl}^{1-\chi_0}$}
We set $\tilde\alpha=\tilde\alpha(\beta,\tau):=1-\tau^{-2/3}\beta$ : as on the support of $\chi_{\varepsilon_1}(\tilde\alpha)$ we have $1-\tilde\alpha\lesssim \varepsilon$, it follows that $\tau^{-2/3}\beta\lesssim \varepsilon$. This choice of coordinates is motivated by the behavior of the Airy factor $A_{+}(\tau^{\frac23}\zeta_0(\alpha,\gamma))$ : as $\tau^{2/3}\zeta_0(\alpha,\gamma)=\tau^{2/3}\alpha^{2/3}\tilde\zeta(\frac{\sqrt{1-\gamma^2}}{\alpha})=\tau^{2/3}\sqrt{1-\gamma^2}^{2/3}\tilde\alpha^{2/3}\tilde\zeta(\frac{1}{\tilde\alpha})$, then
\begin{equation}\label{A+}
\tau\alpha(-\tilde\zeta)^{3/2}(\frac{1}{\tilde\alpha})
 =\sqrt{2}\sqrt{1-\gamma^2}\beta^{3/2}(1+O(\tau^{-2/3}\beta)),
\end{equation}
where we used Lemma \ref{lemzeta}. As such, for $(\sqrt{2}\sqrt{1-\gamma^2}\beta^{3/2})^{2/3}$ large enough, $A_{+}(\tau^{\frac23}\zeta_0(\alpha,\gamma))$ does oscillate, while for $(\sqrt{2}\sqrt{1-\gamma^2}\beta^{3/2})^{2/3}$ bounded it may be brought into the symbol. Write $1=\chi_0(\beta)+(1-\chi_0)(\beta)$. 
On the support of $1-\chi_0(\beta)$ the Airy factor may oscillate and the phase function of $I_{0,gl}^{1-\chi_0}$ equals $z_Q\gamma-\sqrt{1-\gamma^2}\varphi$, where we have set
\begin{equation}\label{varphi}
\varphi(y,\tilde\alpha, r):=-y\tilde \alpha+\Gamma_0(\tilde\alpha,s)+\phi(x_Q,y-y_Q,0,1)-\frac 23(-\tilde\zeta)^{3/2}(\frac{1}{\tilde\alpha}).
\end{equation}
With $\varphi$ defined in \eqref{varphi} we have 
\begin{multline}\label{I1gl1-chi025}
I_{0,gl}^{1-\chi_0}(Q,Q_0,\tau)=\tau^{\frac43-\frac23}\int e^{-\ii\tau(z_Q\gamma-\sqrt{1-\gamma^2}\varphi)}\tilde \chi_{\varepsilon}(1-\tau^{-2/3}\tilde\beta)(1-\chi_0)((y-y_c)\tau^{1/3-\epsilon_1})\\
\beta^{1/4}(1-\gamma^2)^{\frac{1}{12}-\frac 13} (\tilde b f)(y,\beta,\gamma,\tau) \times\frac{\psi_0(1-\gamma^2)}{(s^2-1)^{1/4}\phi^{1/2}(x_Q,y-y_Q,0,1)}
 dy d\gamma d\beta,
\end{multline}
where the factor $\beta^{1/4}(1-\gamma^2)^{1/12}$ comes from the Airy term $A_+^{-1}$ (using \eqref{A+}). 
\begin{lemma}\label{lemgam0r}
Let $y=y_c+\mathtt{Y}$, where $y_c=y_Q+\arccos(1/r)$. There exists a unique change of variables $\mathtt{Y}\mapsto \sigma$ which is smooth and satisfying $\frac{d\mathtt{Y}}{d{\sigma}}\notin\{0,\infty\}$ such that, for $\tilde \zeta$ given by Lemma \ref{lemzeta}, we have
\begin{equation}\label{fctssfinr}
-(y_c+\mathtt{Y})\tilde\alpha+\phi(x_Q,y-y_Q+\mathtt{Y},0,1)=\frac{{\sigma}^3}{3}+{\sigma}\tilde\alpha^{2/3}\tilde\zeta(\frac{1}{\tilde\alpha})+\tilde\Gamma(\tilde\alpha,r),
\end{equation}
and where $\tilde\Gamma(\tilde\alpha,r):=\sqrt{r^2-1}-y_c\tilde\alpha+\frac{(1-\tilde\alpha)^2}{2\sqrt{r^2-1}}(1+O(1-\tilde\alpha)\big)$.
\end{lemma}
\begin{proof} 
We proceed exactly as in the proof of Lemma \ref{lemgam0} (where now $x=0$ and $s$ is replaced by $r$). As $y_c$ is the degenerate critical point of order $2$ of $\phi$, there exist a smooth change of variable $\mathtt{Y}\rightarrow \sigma$ and smooth phase functions $\zeta^{\#}$ and $\tilde\Gamma$ such that the LHS term in \eqref{fctssfinr} reads as $\frac{{\sigma}^3}{3}+{\sigma}\zeta^{\#}(\tilde\alpha,r)+\tilde\Gamma(\tilde\alpha,r)$. Exactly as in Lemma \ref{lemgam0} we obtain that $\zeta^{\#}=\tilde\alpha^{2/3}\tilde\zeta(\frac{1}{\tilde\alpha})$. It remains to determine $\tilde\Gamma$.
The two critical points satisfy 
\[
r\cos(\arccos(1/r)+\mathtt{Y}_{\pm})=\tilde\alpha^2\mp\sqrt{r^2-\tilde\alpha^2}\sqrt{1-\tilde\alpha^2}.
\]
We have as before $\phi(x_Q,y_c+\mathtt{Y}_{\pm},0,1)=\sqrt{r^2-\tilde\alpha^2}\pm\sqrt{1-\tilde\alpha^2}$.
As $\cos(\arccos(1/r)+\mathtt{Y})=\sin(\arcsin(1/r)-\mathtt{Y})$ we use the computations from Lemma \ref{lemgam0} to determine $\arcsin(1/r)-\frac 12(\mathtt{Y}_++\mathtt{Y}_-)$. As $-y_c=-y_Q-\arccos(1/r)=-y_Q-\frac{\pi}{2}+\arcsin(1/r)$ we obtain $\tilde\Gamma(\tilde\alpha,r)=-(y_Q+\frac{\pi}{2})\tilde\alpha+\Gamma_0(\tilde\alpha,r)$ where $\Gamma_0(\tilde\alpha,r)$ is the same as in Lemma \ref{lemgam0} and compute the derivatives of this new function at $\tilde\alpha=1$ using those of $\Gamma_0$ as follows
\[
\tilde\Gamma(1,r)=\sqrt{r^2-1}-y_c, 
\tilde\Gamma '(1,r)=-y_c, 
\tilde\Gamma''(1,r)=\frac{1}{\sqrt{r^2-1}}, \tilde\Gamma^{(k)}=\frac{c_k}{\sqrt{r^2-1}}(1+O(\frac{1}{\sqrt{r^2-1}})).
\]
\end{proof}
Using the changes of variable $y\rightarrow y_c+\mathtt{Y}$, $\mathtt{Y}\rightarrow \sigma$ from Lemma \ref{lemgam0} yields   
\[
\varphi(y,\tilde\alpha,r)=\frac{{\sigma}^3}{3}+{\sigma}\tilde\alpha^{2/3}\tilde\zeta(\frac{1}{\tilde\alpha})+\tilde\Gamma(\tilde\alpha,r)+\Gamma_0(\tilde\alpha,s)-\frac 23(-\tilde\zeta)^{3/2}(\frac{1}{\tilde\alpha}).
\]
Let $y=y_c+\mathtt{Y}$, $\mathtt{Y}\rightarrow\sigma$ as in the Lemma \ref{lemgam0r} and set moreover $\sigma=\tau^{-1/3}w$ : then $ \tau^{\epsilon_1}\lesssim |w|$ on the support of the symbol (and if $|\tau^{-1/3}w|\geq \frac{\pi}{4}$ the integral defining $I_{0,gl}^{1-\chi_0}$ is $O(\tau^{-\infty})$). We apply the stationary phase in $w$ near the critical points : let $\chi$ be a smooth cut-off supported in a fixed neighborhood of $1$ and equal to $1$ near $1$ and set $\chi_{\pm}:=\chi(\pm\frac{\sqrt{2\beta}}{w})$, $\tilde\alpha=1-\tau^{-2/3}\beta$; let also $\overline\chi:=1-\chi_{+}-\chi_-$. Write 
\[
I_{0,gl}^{1-\chi_0}(Q,Q_0,\tau)=\sum_{\chi\in\{ \chi_{\pm},\overline\chi\}}I_{0,gl}^{1-\chi_0,\chi},
\] 
where $I_{0,gl}^{1-\chi_0,\chi}$ are given by \eqref{I1gl1-chi025} with additional cutoffs $\chi(\frac{\sqrt{2\beta}}{w})$.
\begin{lemma}
For $w$ in a small, fixed neighborhood of $\pm\sqrt{2\beta}$, we have
\begin{multline}\label{I1gl1-chi023}
I_{0,gl}^{1-\chi_0,\chi_{\pm}}(Q,Q_0,\tau)=\tau^{4/3-2/3-1/3}\int e^{-\ii\tau(z_Q\gamma-\sqrt{1-\gamma^2}\varphi_{\pm})} \chi_{\varepsilon}(1-\tau^{-2/3}\beta)\Sigma_{\pm}(\beta,\gamma,\tau)(1-\gamma^2)^{\frac{1}{12}-\frac13-\frac14}\\
\times\frac{(1-\chi_0)(\mathtt{Y}(\tau^{-1/3}w)\tau^{1/3-\varepsilon_1})\psi_0(1-\gamma^2)}{(s^2-1)^{1/4}\phi^{1/2}(x_Q,\arccos(1/r)+\mathtt{Y}(\tau^{-1/3}w_{\pm}),0,1)}
d\gamma d\beta ,
\end{multline}
where $\varphi_{\pm}:=\mp\frac 23(-\tilde\zeta)^{3/2}(\frac{1}{\tilde\alpha})+\tilde\Gamma(\tilde\alpha,r)+\Gamma_0(\tilde\alpha,s)-\frac 23(-\tilde\zeta)^{3/2}(\frac{1}{\tilde\alpha})|_{\tilde\alpha=1-\tau^{-2/3}\beta}$. 
Here $\Sigma_{\pm}$ are asymptotic expansions with parameter $ \tau^{-\epsilon_1}$ and main contribution $\frac{d\mathtt{Y}}{d\sigma}\frac{\beta^{1/4}}{\sqrt{|w_{\pm}|}}\chi_{\pm}(\frac{\sqrt{2\beta}}{w_{\pm}})\Sigma(y_c+\tau^{-1/3}w_{\pm},\beta,\tau)$. 
\end{lemma}
\begin{proof}
The critical points are $w_{\pm}:=\pm\tau^{1/3}\sqrt{-\tilde\zeta(\frac{1}{1-\tau^{-2/3}\beta})}$ which gives $w_{\pm}=\pm\sqrt{2\beta}(1+O(\sqrt{2\tau^{-2/3}\beta}))$.
As $|w|\geq \tau^{\epsilon_1}$ on the support of $(1-\chi_0)(\mathtt{Y}(\tau^{-1/3}w)\tau^{1/3-\varepsilon_1})$ and $\frac{w}{\sqrt{2\beta}}\sim \pm1$ on the support of the symbol, we also have $\sqrt{2\beta}\gtrsim \tau^{\epsilon_1}$. 
At $w_{\pm}$, the second order derivative of the phase 
equals $\partial^2_{w}(\tau\sqrt{1-\gamma^2}\varphi)|_{w_{\pm}}=\sqrt{1-\gamma^2}w_{\pm}(1+O(\tau^{-1/3}w_{\pm}))$. As $\sqrt{1-\gamma^2}\geq 1/8$ on the support of $\psi_0$ and $|w|\geq \tau^{\epsilon_1}$ on the support of $(1-\chi_0)(w\tau^{-\epsilon_1})$, it follows that $\sqrt{1-\gamma^2}\times w_{\pm}\gtrsim \tau^{\epsilon_1}$ and the stationary phase applies at $w_{\pm}$ with a parameter larger than $\tau^{\epsilon_1}$. The exponent of the factor $(1-\gamma^2)$ is $1/12-1/3-1/4=-1/2$. The factor $\tau^{-1/3}$ before the integral \eqref{I1gl1-chi023} comes from the change of variables $\sigma\rightarrow w$. (For $1\leq j\leq j(s,h)$, replace $\tau$ by $\tau 2^{-j}$).
\end{proof}
We now consider the integral $I_{0,gl}^{1-\chi_0,\overline\chi}(Q,Q_0,\tau)$ whose symbol is supported for $\frac{w^2}{2\beta}\notin [1/2,3/2]$. 
\begin{lemma}\label{lemgam}
The stationary phase applies in $\gamma$ with large parameter $\tau$ and yields
\begin{multline}\label{I1gl1-chi02}
I_{0,gl}^{1-\chi_0,\overline\chi}(Q,Q_0,\tau)=\tau^{\frac43-1-\frac12}\int e^{-\ii\tau\sqrt{\varphi^2+z_Q^2}} \chi_{\varepsilon}(1-\tau^{-2/3}\beta)(1-\chi_0)(|w|\tau^{-\epsilon_1})\\
\beta^{1/4} \overline\chi(\frac{\sqrt{2\beta}}{w})\Sigma(y,\beta,\tau) \Big(\frac{\varphi^2}{\varphi^2+z_Q^2}\Big)^{\frac 12}\frac{1}{\varphi^{1/2}}\times\frac{\psi_0(\frac{\varphi^2}{\varphi^2+z_Q^2})}{(s^2-1)^{1/4}\phi^{1/2}(x_Q,y-y_Q,0,1)}
d\beta dy,
\end{multline}
where $\Sigma$ is an asymptotic expansion with small parameter $\tau^{-1}$ and main contribution $ b_{\partial}f$.
\end{lemma}
\begin{proof}
The critical point satisfies : $\gamma_c=-z_Q/\sqrt{\varphi^2+z_Q^2}$ and at $\gamma_c$, the second order derivative of the phase equals $\frac{\varphi}{\sqrt{1-\gamma_c^2}^3}=\varphi\times \frac{\sqrt{\varphi^2+z_Q^2}^3}{\varphi^3}\geq \varphi$ and its critical value equals $-\sqrt{z_Q^2+\varphi^2}$. 
In order to show that the stationary phase applies we will show that $\varphi\geq \sqrt{s^2-1}$. From Lemmas \ref{lemgam0r} and \ref{lemgam0} we have $\tilde\Gamma(\tilde\alpha,r)+\Gamma_0(\tilde\alpha,s)=\sqrt{r^2-1}+\sqrt{s^2-1}+(y_*-y_c)\tilde\alpha +O(1-\tilde\alpha)$. As $y$ is close to $y_*$ on the support of the symbol $I_{1,gl}$ ($|y-y_*|\leq \frac{\pi}{16}$) and $|y-y_c|\leq \frac{\pi}{8}$ (these constants may be shrunk if necessary), then $|y_c-y_*|\leq \frac{3\pi}{16}<\frac{5}{8}$ while $\sqrt{r^2-1}\geq \sqrt{2}\sqrt{\sqrt{2}-1}\geq \frac 45$. Moreover, on the support of $\chi_{\varepsilon}(\tilde\alpha)$ we have $|1-\tilde\alpha|\lesssim \varepsilon$ so we conclude taking $\varepsilon_1$ small enough compared to $\varepsilon_0$. The stationary phase yields \eqref{I1gl1-chi02}. The exponent of $\Big(\frac{\varphi^2}{\varphi^2+z_Q^2}\Big)$ equals $\frac{1}{12}-\frac 13+\frac 34$, where the last term comes from the second order derivative.
\end{proof}
\begin{cor} 
We have $I_{0,gl}^{1-\chi_0,\overline\chi}(Q,Q_0,\tau)=O(\tau^{-\infty}/t)$. Moreover, modulo $O(\tau^{-\infty}/t)$,
\begin{multline}\label{I1gl1-chi023_}
I_{0,gl}^{1-\chi_0,\chi_{\pm}}(Q,Q_0,\tau)=\tau^{4/3-1-1/2}\int e^{-\ii\tau\sqrt{\varphi_{\pm}^2+z_Q^2}} \chi_{\varepsilon_1}(1-\tau^{-2/3}\beta)\\
\frac{\tilde\Sigma_{\pm}(\beta,\tau)}{\varphi_{\pm}^{1/2}}\times\frac{\breve{\psi}_0(\frac{\varphi_{\pm}}{\sqrt{\varphi_{\pm}^2+z_Q^2}})}{(s^2-1)^{1/4}\phi^{1/2}(x_Q,y_c-y_Q+\mathtt{Y}(\tau^{-1/3}w_{\pm}),0,1)}
d\gamma d\beta ,
\end{multline}
where $\breve{\psi}_0(\cdot)=(\cdot)^{1/2}\tilde\chi_1$ and $\tilde\Sigma_-$ is a classical symbol with main contribution $\Sigma(\beta,\gamma_c,\tau)$.
\end{cor}
\begin{proof}
Using Lemma \ref{lemgam}, $I_{0,gl}^{1-\chi_0,\overline\chi}(Q,Q_0,\tau)$ of the form \eqref{I1gl1-chi02} : with symbol supported for $w$ far from $w_{\pm}$: repeated integrations by parts yield $O(\tau^{-\infty}/t)$ where the factor $1/t$ is obtained from the symbol $(\varphi\sqrt{s^2-1})^{-1/2}\lesssim 1/\sqrt{s^2-1}\lesssim 1/t$ as $2\sqrt{s^2-1}\geq \varphi\geq \sqrt{s^2-1}$, as $t\sim \sqrt{\varphi^2+z_Q^2}\leq 8\varphi$ on the support of $\psi_0$ and $r\geq\sqrt{2}$. 
To obtain \eqref{I1gl1-chi023_} we use the proof of Lemma \ref{lemgam} to \eqref{I1gl1-chi023} for the $\pm$ signs.
\end{proof}
Using the Corollary, we obtain $I_{0,gl}^{1-\chi_0}(Q,Q_0,\tau)=\sum_{\pm}I_{0,gl}^{1-\chi_0,\chi_{\pm}}+O(\tau^{-\infty}/t)$, where $I_{0,gl}^{1-\chi_0,\chi_{\pm}}$ are given in \eqref{I1gl1-chi023_}. We are left with the integration with respect to $\beta$ in the integrals \eqref{I1gl1-chi023_} whose symbols $(1-\chi_0)(\sqrt{\beta}\tau^{-\epsilon_1})\chi_{\varepsilon_1}(1-\tau^{-2/3}\beta)$ are supported for $\beta\gtrsim \tau^{2\epsilon_1}$ and $\tau^{-2/3}\beta\lesssim \varepsilon_1$. 
As $\beta$ takes values in a large interval, we consider separately dyadic intervals where $\beta\sim 2^{2k}$ and then sum all the contributions. Let $\tilde\chi$ supported near $1$ and equal to $1$ on $[\frac 34,\frac 54]$ such that 
\begin{equation}\label{sumk}
(1-\chi_0)(\sqrt{\beta}\tau^{-\epsilon_1})\chi_{\varepsilon_1}(1-\tau^{-2/3}\beta)\sum_k\tilde\chi (\beta 2^{-2k})=(1-\chi_0)(\sqrt{\beta}\tau^{-\epsilon_1})\chi_{\varepsilon_1}(1-\tau^{-2/3}\beta).
\end{equation}
On the support of $(1-\chi_0)(\sqrt{\beta}\tau^{-\epsilon_1})\chi_{\varepsilon_1}(1-\tau^{-2/3}\beta)$ we have $\tau^{2\epsilon_1}< \beta \lesssim \varepsilon_1 \tau^{2/3}$ and for each $k$ in the previous sum, $\tilde\chi (\beta 2^{-2k})$ localize at $\beta\sim 2^{2k}$. The sum is thus taken for $\epsilon_1\log_2(\tau)\leq k< \frac 13 \log_2(\tau)$.
Recall that $\varphi|_{\pm}=\varphi|_{w_{\pm}}$ where $\varphi_-=\tilde\Gamma(\tilde\alpha,r)+\Gamma_0(\tilde\alpha,s)$ and $\varphi_+=\varphi_--\frac 43(-\tilde\zeta)^{3/2}(\frac{1}{\tilde\alpha})$.
We deal separately with the $\pm$ signs. Let $I_{0,gl}^{1-\chi_0,\chi_\pm,k}$ denote the integrals in \eqref{I1gl1-chi023_} with additional cutoff $\tilde\chi(\beta2^{-2k})$. Using \eqref{sumk} we have
\[
I_{0,gl}^{1-\chi_0,\chi_\pm}=\sum_{k=\epsilon_1\log_2\tau}^{(\log_2\tau)/3} I_{0,gl}^{1-\chi_0,\chi_\pm,k}.
\]
\begin{lemma}
There exists a constant $C=C_+(\varepsilon)$ such that 
$|I_{0,gl}^{1-\chi_0,\chi_+}|\leq \sum_{k=\epsilon_1\log_2\tau}^{(\log_2\tau)/3} |I_{0,gl}^{1-\chi_0,\chi_+,k}|\leq C_+(\varepsilon)/t$.
\end{lemma}
\begin{proof}
 At $w_+=\sqrt{2\beta}(1+O(\sqrt{2\tau^{-2/3}\beta}))$, the phase $\varphi_+$ is stationary when $\tau^{1/3}(y_c-y_*)=2\sqrt{2\beta}(1+O(\sqrt{\tau^{-2/3}\beta}))$. Let $\beta=2^{2k}\Xi$ on the support of $\tilde\chi(\beta 2^{-2k})$, with $\Xi\in [1/2,3/2]$. As
\begin{equation}
\begin{aligned}
\partial_{\Xi}(\tau\sqrt{\varphi^2_++z_Q^2})&=\frac{\varphi_+}{\sqrt{\varphi_+^2+z_Q^2}}\frac{\partial\beta}{\partial\Xi}(\tau\partial_\beta\varphi_+)\\
& =\frac{\varphi_+}{\sqrt{\varphi_+^2+z_Q^2}}  2^{2k+k}\Big(\frac{\tau^{1/3}(y_c-y_*)}{2^k}-2\sqrt{2\Xi}(1+O(\sqrt{\tau^{-2/3}\beta}))\Big),
\end{aligned}
\end{equation}
the phase is stationary for $\Xi\sim 1$ only when $\frac{\tau^{1/3}(y_c-y_*)}{2^k}\sim 2\sqrt{2}$ ; as
$\frac{\varphi_+}{\sqrt{\varphi_+^2+z_Q^2}}\geq 1/8 $ on the support of $\psi_0$ and as $2^{3k}\geq \tau^{3\epsilon_1/2}$, it follows that, for $|\frac{\tau^{1/3}(y_c-y_*)}{2^{k}}-2\sqrt{2}|\geq 4$ and $\Xi\in [1/2,3/2]$, repeated integrations by parts yield a contribution $O(\tau^{-\infty}/t)$. We deduce that there are at most a finite number of values of $k$ for which the phase may be stationary ; for such $k$ the stationary phase applies at the critical point $2\sqrt{2\Xi}\sim\frac{\tau^{1/3}(y_c-y_*)}{2^k}$ as, there, the second order derivative equals $-\frac{\varphi_+}{\sqrt{\varphi_+^2+z_Q^2}}  2^{3k}\times \frac{\sqrt{2}}{\sqrt{\Xi}}$ and $\Xi\sim 1$. For all such $k$, the stationary phase yields a factor $2^{2k}\times 2^{-3k/2}\times \sqrt{\varphi^2_++z_Q^2}^{1/2}/\varphi_+^{1/2}$, where the exponent $2k$ comes from the change of variables and the exponent $2^{-3k/2}$ from the second order derivative at $\Xi\sim 1$. As $2^{2k}\leq\tau^{2/3}$, the sum over all such $k$ yields at most $2^{k/2}\leq \tau^{1/6}$ and the exponent $1/6$ is canceled by the exponent of $\tau^{4/3-2/3-1/2-1/3}$ from $I_{0,gl}^{1-\chi_0,\chi_+}$. We conclude using that  $(\varphi_+\sqrt{s^2-1})^{-1/2}\leq C_+(\varepsilon)/t$, 
where $C_+(\varepsilon)$ depends only on $\varepsilon$.
\end{proof}

\begin{lemma}
There exists a constant $C=C_-(\varepsilon)$ such that
$\sum_{k=\epsilon_1 \log_2\tau}^{(\log_2\tau)/3} |I_{0,gl}^{1-\chi_0,\chi_-,k}(Q,Q_0,t)|\lesssim C_-(\varepsilon)/t$.
\end{lemma}
\begin{proof}
We have $\varphi_-=\tilde\Gamma(\tilde\alpha,r)+\Gamma_0(\tilde\alpha,s)$ hence, for $\tilde\alpha=1-\tau^{-2/3}\beta$, we have
\begin{equation}
\begin{aligned}
\tau\varphi_-&=\tau\Big(\sqrt{r^2-1}+\sqrt{s^2-1}-(y_c-y_*)\tilde\alpha\\
&+\frac{(1-\tilde\alpha)^2}{2}\Big(\frac{1}{\sqrt{r^2-1}}(1+O(1-\tilde\alpha))+\frac{1}{\sqrt{s^2-1}}(1+O(1-\tilde\alpha))\Big)\Big)\\
\tau\partial_{\beta}\varphi_-&=\tau^{1/3}(y_c-y_*)+\tau^{-1/3}\beta\Big(\frac{1}{\sqrt{r^2-1}}(1+O(\tau^{-2/3}\beta))+\frac{1}{\sqrt{s^2-1}}(1+O(\tau^{-2/3}\beta))\Big).
\end{aligned}
\end{equation}
At $\beta=2^{2k}\Xi$ we have $\partial_{\Xi}(\tau\sqrt{\varphi^2_-+z_Q^2})=\frac{\varphi_-}{\sqrt{\varphi_-^2+z_Q^2}}2^{2k}\partial_{\beta}(\tau\varphi_-)|_{\beta=2^{2k}\Xi}$ hence
\[
\partial_{\Xi}(\tau\sqrt{\varphi^2_-+z_Q^2})=\frac{2^{4k}\tau^{-1/3}\varphi_-}{\sqrt{\varphi_-^2+z_Q^2}}\Big(\frac{\tau^{2/3}}{2^{2k}}(y_c-y_*)+ \Xi\big(\frac{1}{\sqrt{r^2-1}}(1+O(\tau^{-2/3}2^{2k}))+\frac{1}{\sqrt{s^2-1}}(1+O(\tau^{-2/3}2^{2k}))\big)\Big),
\]
 \[
 \partial^2_{\Xi}(\tau\sqrt{\varphi^2_-+z_Q^2})|_{\partial_{\beta}\varphi_-=0}=\frac{2^{4k}\tau^{-1/3}\varphi_-}{\sqrt{\varphi_-^2+z_Q^2}}\Big(\frac{1}{\sqrt{r^2-1}}(1+O(\tau^{-2/3}2^{2k}))+\frac{1}{\sqrt{s^2-1}}(1+O(\tau^{-2/3}2^{2k}))\Big).
 \]
 As $s\geq r$ and for $\frac{\varphi_-}{\sqrt{\varphi_-^2+z_Q^2}}$ on the support of $\psi_0$ we obtain a lower bound for the second order derivative of $\frac{2^{4k}\tau^{-1/3}}{\sqrt{r^2-1}}$. From now on we can proceed as in the case of $\varphi_+$ : if $\frac{2^{4k}\tau^{-1/3}}{\sqrt{r^2-1}}\geq \tau^{\epsilon}$ for some $\epsilon>0$, we apply the stationary phase if moreover $\frac{\tau^{2/3}\sqrt{r^2-1}}{2^{2k}}(y_*-y_c)\sim 1$ : the last condition reduces the number of such $k$ to at most three values for which we find
 \begin{equation}\label{mainI1gl}
 |I_{0,gl}^{1-\chi_0,\chi_-,k}(Q,Q_0,t)|\lesssim \frac{\tau^{-1/6}}{t}\times \frac{2^{2k}}{\phi^{1/2}}(\frac{\varphi_-}{\sqrt{\varphi_-^2+z_Q^2}})^{3/4-1/2}\times \big(\frac{2^{4k}\tau^{-1/3}}{\sqrt{r^2-1}}\big)^{-1/2}\sim 1/t,
\end{equation}
 where we used that  $(\varphi\sqrt{s^2-1})^{-1/2}\lesssim 1/t$ and $\phi\geq r-1$ to obtain $\frac{(r^2-1)^{1/4}}{\phi^{1/2}}\lesssim 1$. 
 If $\frac{\tau^{2/3}\sqrt{r^2-1}}{2^{2k}}(y_*-y_c)\notin [1/4,4]$, repeated integrations by parts yield a $O(\tau^{-\infty}/t)$ contribution (and we conclude using that $2^k\lesssim \tau^{1/3}$).\\

Fix $M>4$ large enough and consider $2^{4k}\tau^{-1/3}\frac{1}{\sqrt{r^2-1}}\in [M^2,\tau^{\epsilon'}]$ for some $\epsilon'>0$. As this parameter is large, we still may apply the stationary phase but we need to verify that the remainders are sufficiently small and that we can bound their sum. There is still a finite number of $k$ for which the phases may be stationary. At the critical points $\Xi_c$, the stationary phase applies and we obtain, for all $N\geq 1$,
\begin{multline}
I_{0,gl}^{1-\chi_0,\chi_-,k}(Q,Q_0,t)
=\tau^{-1/6}e^{-i\tau\sqrt{\varphi^2_-+z_Q^2}}\frac{\tilde\Sigma_-(2^{2k}\Xi_c,\tau)}{\varphi_-^{1/2}(s^2-1)^{1/4}}
\frac{2^{2k}\times |\partial^2_{\Xi}(\tau\sqrt{\varphi^2_-+z_Q^2})|^{-1/2}}{\phi^{1/2}(x_Q,\arccos(1/r)+\tau^{-1/3}w_-,0,1)}\\
+O\Big((\frac{2^{4k}\tau^{-1/3}}{\sqrt{r^2-1}})^{-N}\tau^{-1/6}\times \frac{2^{2k}}{\sqrt{s^2-1}(r-1)^{1/4}}\Big),
\end{multline}
where the main contribution of $I_{0,gl}^{1-\chi_0,\chi_-,k}(Q,Q_0,t)$ in the first line still satisfies \eqref{mainI1gl} and where the remainder in the second line is $O\Big((\frac{2^{4k}\tau^{-1/3}}{4\sqrt{r^2-1}})^{-N}/t\Big)$. In the second line we used $\phi\geq (r-1)^{1/2}$.
The bounds for the remainders follow using $\sup |\partial^2_{\Xi}(\tau\sqrt{\varphi^2_-+z_Q^2})|\geq \frac{2^{4k}\tau^{-1/3}}{\sqrt{r^2-1}}$. Notice that, taking one derivative of the cutoff $\chi_{\varepsilon}(\tau^{-2/3}2^{2k}\Xi)=\chi(\tau^{-2/3}2^{2k}\Xi/\varepsilon)$ yields a factor $\tau^{-2/3}2^{2k}/\varepsilon$ but, as $\Xi\sim 1$ on the support of $\chi(\beta2^{-2k})=\chi(\Xi)$, on the support of $\chi(\tau^{-2/3}2^{2k}\Xi/\varepsilon)\chi(\Xi)$ we have $\tau^{-2/3}2^{2k}/\varepsilon\lesssim 1$ hence for $M>4$ sufficiently large this factor doesn't change the contribution of the remainder.
For all $k$ s.t. the phase is not stationary, integration by parts 
yields a contribution of at most
\[
\tau^{-1/6}\frac{2^{2k}}{t(r-1)^{1/2}}\times (2^{-4k}\tau^{1/3}\sqrt{r^2-1})^{N+1}=\frac 1t \times  (2^{-2k}\tau^{1/6}\sqrt{r^2-1}^{1/2})\times  (2^{-4k}\tau^{1/3}\sqrt{r^2-1})^N
\]
for all $N\geq 0$. Let $N=0$ and sum over $k$ with $2^{4k}\tau^{-1/3}\frac{1}{\sqrt{r^2-1}}\in [M^2,\tau^{\epsilon}]$, then 
\[
\frac 1t \Big(\sum_{M^2\leq 2^{4k}\tau^{-1/3}\sqrt{r^2-1}^{-1}\leq \tau^{\epsilon'}} 2^{-2k}\Big)\times \tau^{1/6}\sqrt{r^2-1}^{1/2}\lesssim 1/(Mt).
\] 
Let now $k$ such that $\tau^{\epsilon_1}\lesssim 2^k$ and $2^{4k}\tau^{-1/3}\frac{1}{\sqrt{r^2-1}}\leq M^2$ for some large, fixed $M>1$. We bound each $I_{0,gl}^{1-\chi_0,\chi_-,k}$ by $\tau^{-1/6}\frac{2^{2k}}{t(r-1)^{1/4}}\lesssim M/t$ using $2^{2k}\leq M\tau^{1/6}\sqrt{r^2-1}^{1/2}$ and conclude.
\end{proof}
\begin{rmq}\label{rmqfacsmall1-gam}
In the two previous Lemmas, the bounds for $I_{0,gl}^{1-\chi_0,\chi_{\pm}}$ come with additional factors $ (\frac{\varphi_{\pm}}{\sqrt{\varphi_{\pm}^2+z_Q^2}})^{1/4}$. This is useful to keep in mind for the case when $1-\gamma^2$ behaves like $2^{-2j}$.
\end{rmq}
For $\beta$ on the support of $\chi_0(\beta)$, using \eqref{A+}, the Airy factor can be brought in the symbol. 
The phase of $I_{0,gl}^{1-\chi_0}$ equals $\tau(z_Q\gamma-\sqrt{1-\gamma^2}\varphi_0)$, where $\varphi_0:=(-(y_c-y_*+\tau^{-1/3}w)(1-\tau^{-2/3}\beta)+\sqrt{s^2-1}+\phi(x_Q,\arccos(1/r)+\tau^{-1/3}w,0,1))\geq \sqrt{s^2-1}$. As $\sqrt{1-\gamma^2}\times w\geq \tau^{\epsilon_1}$ on the support of $\psi_0(1-\gamma^2)(1-\chi_0)(w\tau^{-\epsilon_1})$, it follows that the phase is non-stationary in $w$ as $\beta\leq 2\ll \tau^{2\epsilon_1}\lesssim w^2/2$ and we integrate by parts to obtain $O(\tau^{-\infty}/t)$. 

\subsubsection{Case $|y-y_c|\leq 2\tau^{-1/3+\epsilon_1}$ : study of $I_{0,gl}^{\chi_0}$}

Let $y=y_c+\tau^{-1/3}w$, with $|w|\leq \tau^{\epsilon_1}$. 
As $\partial_w\phi(x_Q,\arccos(1/r)+\tau^{-1/3}w,0,1)=\tau^{-1/3}\Big(1-\tau^{-2/3}w^2/2(1+O(\tau^{-1/3}w))\Big)$, the derivative w.r.t. $w$ of phase of $I_{0,gl}^{\chi_0}$ equals 
\[
\tau^{1-1/3}\sqrt{1-\gamma^2}\Big(1-\tau^{-2/3}\beta-1+\tau^{-2/3}w^2/2(1+O(\tau^{-1/3}w))\Big)=\sqrt{1-\gamma^2}(-\beta+w^2/2(1+O(\tau^{-1/3}w)),
\]
hence, for $\beta\geq \tau^{2\epsilon_1}$ we perform repeated integrations by parts to obtain a $O(\tau^{-\infty}/t)$ contribution (using that the support in $w,\beta$ is bounded). We introduce $\chi_0(\beta\tau^{-2\epsilon_1})$ into the symbol of $I_{0,gl}^{\chi_0}$ without changing its contribution modulo $O(\tau^{-\infty}/t)$ terms. If we introduce moreover a cutoff $\chi_0(\beta)$ supported for $\beta\leq 2$, the Airy factor doesn't oscillate and may be brought into the symbol : in this case the phase of $I_{0,gl}^{\chi_0}$ is given by
\[
\tau(z_Q\gamma-\sqrt{1-\gamma^2}(-(y_c-y_*+\tau^{-1/3}w)(1-\tau^{-2/3}\beta)+\sqrt{s^2-1}+\phi(x_Q,\arccos(1/r)+\tau^{-1/3}w,0,1))).
\]
Let $\varphi_0:=-(y_c-y_*+\tau^{-1/3}w)(1-\tau^{-2/3}\beta)+\sqrt{s^2-1}+\phi(x_Q,\arccos(1/r)+\tau^{-1/3}w,0,1)$, then $\varphi_0\geq \sqrt{s^2-1}$ and the stationary phase w.r.t. $\gamma$ applies exactly as before. The critical point $\gamma$ satisfies $z_Q=\frac{\gamma}{\sqrt{1-\gamma^2}}\varphi_0$. The  contribution of $I_{0,gl}^{\chi_0}(Q,Q_0,\tau)$ is of the form \eqref{I1gl1-chi02} where moreover $\beta\leq 2$ and $|w|\leq \tau^{\epsilon_1}$. We then obtain
\begin{equation}\label{Ichi0}
|I_{0,gl}^{\chi_0}(Q,Q_0,\tau)|\lesssim \frac{\tau^{1/6-1/3}}{\varphi_0^{1/2}(s^2-1)^{1/4}}\times \tau^{\epsilon_1},
\end{equation}
where the exponents $1/6-1/3$ come from \eqref{I1gl1-chi02} and the change of variable $y=y_c+\tau^{-1/3}w$, and the factor $\tau^{\epsilon_1}$ from the size of the support in $w$. Let now $\beta\in [3/2, \tau^{2\epsilon_1}]$ on the support of $(1-\chi_0(\beta))\chi_0(\beta\tau^{2\epsilon_1})$, when the Airy factor does oscillate. We also have $\varphi\geq \sqrt{s^2-1}$ and the stationary phase w.r.t $\gamma$ applies. The corresponding contribution of $I_{0,gl}^{\chi_0}(Q,Q_0,\tau)$ may be bounded as in \eqref{Ichi0} but with an additional factor $\tau^{2\epsilon_1}$ arising from the support wr.t. $\beta\leq \tau^{2\epsilon_1}$. Taking $\varepsilon_1<1/18$ allows to conclude.
\end{proof}

\subsection{ Dispersive bounds when $d(Q,\partial\Omega)\leq\sqrt{2}-1\leq d(Q_0,\partial\Omega)$}

Let $1\leq r\leq \sqrt{2}\leq s$ and let $0\leq j\leq j(s,h)$. We proceed as in \cite[Section 3.3]{IL} to obtain directly the form of the reflected wave, which may be done using the Melrose and Taylor parametrix as the observation point $Q$ is close to the boundary; formula \eqref{vhform} becomes useless since $d(Q,\partial\Omega)$ may be arbitrarily small. By Proposition \ref{struttura} we are reduced to prove $|\sum_{j=0}^{j(s,h)}\int\chi(h\tau)\tau e^{it\tau}I_j(Q_0,Q,\tau)d\tau|\leq \frac{C}{h^2t} $ for a constant independent of $Q_0$ and $Q$, where we set 
\[
I_j(\tau,Q_0,Q):=\tau\int e^{\ii\tau(y\alpha+z\gamma)}\Big(a_jA_{+}(\tau^{2/3}\zeta)+b_j\tau^{-1/3}A'_{+}(\tau^{2/3}\zeta)\Big)\frac{A(\tau^{2/3}\zeta_0)}{A_{+}(\tau^{2/3}\zeta_0)} \psi_j(1-\gamma^2)\hat{F}_{j,\tau}(\alpha,\gamma)
d\alpha d\gamma
\]
obtained as done in the last part of Section \ref{sectgl}.
\begin{lemma}
There exists a constant $C>0$, such that, for all $Q$ in a small neighborhood of $\mathcal{C}_{Q_0}$, $|y-y_*|\leq \frac{\pi}{16}$ and $t\sim dist(Q_0,\partial\Omega)+dist(Q,\partial\Omega)$ the following holds
$\big|\sum_{j=0}^{j(s,h)}I_j(Q_0,Q,\tau)\big|\leq \frac{C}{t}$.
 \end{lemma}
The Lemma follows exactly as in \cite[Lemma 3.25]{IL} for all $j\leq j(s,h)$ as the observation point $Q$ is located near a glancing point of the boundary (notice that, in the case $Q$ far from $\partial\Omega$, the geometry of the obstacle was important and the approach to obtain dispersive bounds in the case of the exterior of the cylinder was different from the one in the exterior of a ball; when $Q$ is near $\partial \Omega$ the same arguments hold in both cases so we do not reproduce the proof here. Moreover, all stationary arguments hold for $j\leq j(s,h)$).

\subsection{Dispersive bounds for the "non-glancing" part, $d(Q_0,\partial\Omega)\geq \sqrt{2}-1$}
Let $s\geq \sqrt{2}-1$ as before.

We let $u^{\#}_{j,he,h}(Q,Q_0,t):=\int_{\partial\Omega}\frac{\partial_xu^+_{j,he,h}(P,Q_0,t-|Q-P|)}{4\pi |Q-P|}d\sigma(P)$, where $\partial_x u^+_{j,he}|_{\partial\Omega}$ has been defined in \eqref{partxu1he2-2j}. 
\begin{prop}\label{prophypell}
There exists $C=C(\varepsilon)>0$ such that for all $t>h$, the following holds uniformly with respect to $Q,Q_0$ such that $s\geq r\geq \sqrt{2}$ (where $s=1+x_{Q_0}$, $r=1+x_Q$) : 
\[
\sum_{j=0}^{j(s,h)} |u^{\#}_{j,he,h}(Q,Q_0,t)|\leq \frac{C}{h^2 t}.
\]
\end{prop}
\begin{proof}
Using \eqref{partxu1he2-2j}, it follows that the phase function of $u^{\#}_{j,he,h}(Q,Q_0,t)$ is $\tau(t-\Phi)$ where $\Phi:=|Q-P|+|P-Q_0|$ and the symbol is $\frac{\tau^2}{|P-Q| |P-Q_0|}\sigma_{j,he}(y,z,s,\tau)$ with $\sigma_{j,he}$ a classical symbol of order $0$ with respect to $\tau$ supported for $P$ with coordinates $(x,y,z)_P=(0,y,z)$ such that $s(\sin y_*-\sin y)\geq c(\varepsilon)|P-Q_0|$ and $2^{-j}(-z)\sim \phi(0,y,0,s)$.
In the following it will be convenient to work with the coordinates $(r,\theta,z)$ (instead of $(x,y,z)$). Recall that we set $r=1+x$, $\theta=\frac{\pi}{2}-y$. In these coordinates, the support conditions for $\sigma_{j,he}$ become $s(\cos\theta_*-\cos \theta)\geq c(\varepsilon)|P-Q_0|$, $\theta_*=\arccos(1/s)$.
We compute the derivative of the phase $\Phi$, where
\[
\Phi:=|Q-P|+|Q_0-P|=\tilde\phi(1,\theta-\theta_Q,r_Q,z-z_Q)+\tilde\phi(1,\theta,s,z),
\]
where now $P=(\cos \theta,\sin\theta,z)\in\mathbb{R}^3$, $Q_0=(s,0,0)$ and $Q=(r_Q\cos \theta_Q,r_Q\sin \theta_Q,z_Q)$ and where $\tilde\phi$ is defined in \eqref{Phi}. Let $r=r_Q$. The critical points satisfy $\partial_{\theta}\Phi=\partial_z\Phi=0$, which is equivalent to
\begin{equation}\label{systzvarphitrans}
\left\{\begin{aligned}
&\frac{z}{\tilde\phi(1,\theta,s,z)}+\frac{z-z_Q}{\tilde\phi(1,\theta-\theta_Q,r,z-z_Q)}=0,\\
&\frac{s\sin\theta}{\tilde\phi(1,\theta,s,z)}+\frac{r\sin(\theta-\theta_Q)}{\tilde\phi(1,\theta-\theta_Q,r,z-z_Q)}=0.
\end{aligned}\right.
\end{equation}
We aim at applying the stationary phase with respect to both $\theta$ and $z$. We evaluate the second order derivatives of $\Phi$ at $\nabla_{\theta,z}\Phi=0$. The second order derivative of $\Phi$ satisfies
\begin{equation}\label{derivzztrans}
\partial^2_{z,z}\Phi|_{\partial_{z}\Phi=0}=\Big(\frac{1}{\tilde\phi(1,\theta,s,z)}+\frac{1}{\tilde\phi(1,\theta-\theta_Q,r,z-z_Q)}\Big)\Big(1-\frac{z^2}{\tilde\phi^2(1,\theta,s,z)}\Big).
\end{equation}
Next, as $\partial^2_{\theta,z}\Phi=-\Big(\frac{zs\sin\theta}{\tilde\phi^3(1,\theta,s,z)}+\frac{(z-z_Q)r\sin(\theta-\theta_Q)}{\tilde\phi^3(1,\theta-\theta_Q,r,z-z_Q)}\Big)$, we obtain, using the system \eqref{systzvarphitrans},
\begin{equation}\label{derivzvartrans}
\partial^2_{\theta,z}\Phi|_{\nabla_{\theta,z}\Phi=0}=-\Big(\frac{1}{\tilde\phi(1,\theta,s,z)}+\frac{1}{\tilde\phi(1,\theta-\theta_Q,r,z-z_Q)}\Big)\frac{zs\sin\theta}{\tilde\phi^2(1,\theta,s,z)}.
\end{equation}
Finally, we compute
\begin{equation}\label{secondderiv}
\partial^2_{\theta,\theta}\Phi=\frac{s\cos\theta}{\tilde\phi(1,\theta,s,z)}-\frac{s^2\sin^2\theta}{\tilde\phi^3(1,\theta,s,z)}+\frac{r\cos(\theta-\theta_Q)}{\tilde\phi(1,\theta-\theta_Q,r,z-z_Q)}-\frac{r^2\sin^2(\theta-\theta_Q)}{\tilde\phi^3(1,\theta-\theta_Q,r,z-z_Q)}.
\end{equation}
To evaluate $\partial^2_{\theta,\theta}\Phi|_{\nabla_{z,\theta}\Phi=0}$ we need a refined analysis of the critical points. Using both equations in \eqref{systzvarphitrans} gives $\frac{s\sin\theta}{\psi(s,\theta)}=-\frac{r\sin(\theta-\theta_Q)}{\psi(r,\theta-\theta_Q)}$, where $\psi(s,\theta)=\sqrt{1-2s\cos\theta+s^2}$, and hence $\theta\in [\theta_Q-\pi,\theta_Q]$. Taking the square  in the last equality in \eqref{systzvarphitrans}, subtracting $1$ and then using the first in \eqref{systzvarphitrans}
yields $\frac{s\cos\theta-1}{\tilde\phi(1,\theta,s,z)}=\pm \frac{(r\cos(\theta-\theta_Q)-1)}{\tilde\phi(1,\theta-\theta_Q,r,z-z_Q)}$. Depending on the sign, we separate three situations : \\

\paragraph{\bf Different signs} Consider first the case $\frac{s\cos\theta-1}{\tilde\phi(1,\theta,s,z)}=- \frac{(r\cos(\theta-\theta_Q)-1)}{\tilde\phi(1,\theta-\theta_Q,r,z-z_Q)}$ when
 \[
 \frac{s\cos\theta}{\tilde\phi(1,\theta,s,z)}+\frac{r\cos(\theta-\theta_Q)}{\tilde\phi(1,\theta-\theta_Q,r,z-z_Q)}=\frac{1}{\tilde\phi(1,\theta,s,z)}+\frac{1}{\tilde\phi(1,\theta-\theta_Q,r,z-z_Q)}.
 \] 
We find  
\begin{equation}\label{derivvarvartrans}
\partial^2_{\theta,\theta}\Phi|_{\nabla_{\theta,z}\Phi=0}=\Big(\frac{1}{\tilde\phi(1,\theta,s,z)}+\frac{1}{\phi(1,\theta-\theta_Q,r,z-z_Q)}\Big)\Big(1-\frac{s^2\sin^2\theta}{\tilde\phi^2(1,\theta,s,z)}\Big).
\end{equation}
Using \eqref{derivzztrans}, \eqref{derivvarvartrans}, \eqref{derivzvartrans}, the determinant of the Hessian matrix equals
\begin{equation}\label{detHessmattrans}
\Big(\frac{1}{\tilde\phi(1,\theta,s,z)}+\frac{1}{\tilde\phi(1,\theta-\theta_Q,r,z-z_Q)}\Big)^2\times \frac{(s\cos\theta-1)^2}{\tilde\phi^2(1,\theta,s,z)}|_{\nabla_{\theta,z}\Phi=0}
\end{equation}
and on the support of the symbol $\sigma_{1,he}$ the second factor in \eqref{detHessmattrans} takes values in $[c^2(\varepsilon_1),1]$.
The unique critical point w.r.t. $z$ reads as $z_c=z_Q\times \frac{\psi(s,\theta)}{\psi(s,\theta)+\psi(r,\theta-\theta_Q)}$.
\begin{lemma}\label{lemphasstattrans}
When $\tau\times \Big(\frac{1}{\tilde\phi(1,\theta,s,z)}+\frac{1}{\tilde\phi(1,\theta-\theta_Q,r,z-z_Q)}\Big)\geq M$ for some $M>1$ large enough, the usual stationary phase applies for $\theta,z$ near the critical points and yields $|\mathcal{F}(u^{\#}_{j,he,h})(Q,Q_0,\tau)|\lesssim \frac{\tau}{t}$
when $t\sim \tilde\phi(1,\theta_c-\theta_Q,r,z_c-z_Q)+\tilde\phi(1,\theta_c,s,z_c)$. For $z,\theta$ outside a fixed neighborhood of the critical points the previous estimate still holds.
\end{lemma}
\begin{proof}
We let $j=0$ for simplicity. When $\tau\times \Big(\frac{1}{\tilde\phi(1,\theta,s,z)}+\frac{1}{\tilde\phi(1,\theta-\theta_Q,r,z-z_Q)}\Big)\geq \tau^{\epsilon}$ for some $\epsilon>0$, the stationary phase obviously applies with large parameter $\gtrsim \tau^{\epsilon}$ : then $\mathcal{F}(u^{\#}_{0,he,h})(Q,Q_0,\tau)$ takes the form
\begin{multline}\label{estimF}
\mathcal{F}(u^{\#}_{0,he,h})(Q,Q_0,\tau)=\frac{\tau^{2-1} e^{i\tau(t-\Phi)_{z_c,\theta_c}} \tilde\sigma_{0,he}(\theta,z,s,\tau)}{\tilde\phi(1,\theta-\theta_Q,r,z-z_Q)\tilde\phi(1,\theta,s,z)} \Big(\frac{1}{\tilde\phi(1,\theta,s,z)}+\frac{1}{\tilde\phi(1,\theta-\theta_Q,r,z-z_Q)}\Big)^{-1} \Big|_{z_c,\theta_c} \\
+O\big(\frac{\tau^{-\infty}}{\tilde\phi(1,\theta_c-\theta_Q,r,z_c-z_Q)+\tilde\phi(1,\theta_c,s_c,z)}\big)
\end{multline}
for some new symbol $\tilde\sigma_{0,he}$ which reads as an asymptotic expansion with main contribution $\sigma_{j,he}$ and small parameter $\lesssim \tau^{-\epsilon'}$. As the main contribution of $\mathcal{F}(u^{\#}_{0,he,h})(Q,Q_0,\tau)$ can be bounded by
$\frac{\tau\tilde\sigma_{0,he}(\theta,z,s,\tau)}{\tilde\phi(1,\theta-\theta_Q,r,z-z_Q)+\tilde\phi(1,\theta,s,z)}$,
for $t\sim \tilde\phi(1,\theta_c-\theta_Q,r,z_c-z_Q)+\tilde\phi(1,\theta_c,s,z_c)$ this allows to conclude using the integration w.r.t. $\tau$. For $t$ that doesn't satisfy this condition we conclude by integrations by parts, finite speed of propagation and support properties of the symbol. If we replace $\tau^{\epsilon'}$ by some large constant $M$, the main contribution of $\mathcal{F}(u^{\#}_{0,he,h})(Q,Q_0,\tau)$ can be bounded in the same way, but we need to bound the remainder terms as follows
\[
\frac{\tau^2}{\tilde\phi(1,\theta-\theta_Q,r,z-z_Q)\tilde\phi(1,\theta,s,z)}\times \tau^{-1}\Big(\frac{1}{\tilde\phi(1,\theta,s,z)}+\frac{1}{\tilde\phi(1,\theta-\theta_Q,r,z-z_Q)}\Big)M^{-N} 
\]
for all $N\geq 1$, which is enough to conclude. For $j\leq j(s,h)$ we conclude in the same way.\\

Let now $z,\theta$ outside a fixed neighborhood of the critical points. If moreover $|z|\geq 2t$, the phase $\tau(t-\Phi)$ is not stationary w.r.t. $\tau$; let $|z|\leq 2t$ such that $|\frac{z}{z_c}-1|\geq  c$ for some fixed constant $c>0$.
If $\tau\frac{|z_Q|}{\tilde\phi(1,\theta-\theta_Q,r,z-z_Q)}\geq M_1$ for some large $M_1>1$, then we make repeated integrations by parts as $\tau\partial_z\Phi= 
\frac{z_Q}{\tilde\phi_2(1,\theta-\theta_Q,r,z-z_Q)}(\frac{z}{z_c}-1)$. Let $\tau\frac{|z_Q|}{\tilde\phi(1,\theta-\theta_Q,r,z-z_Q)}< M_1$. As $\tau\partial_z\Phi= \tau\Big(\frac{1}{\tilde\phi(1,\theta,s,z)}+\frac{1}{\tilde\phi(1,\theta-\theta_Q,r,z-z_Q)}\Big)z-\tau 
\frac{z_Q}{\tilde\phi(1,\theta-\theta_Q,r,z-z_Q)}$, then if $\tau \partial_z\Phi\geq M_2$ for some large constant $M_2>1$, repeated integrations by parts allow to conclude;
if, instead, $\tau\partial_z\Phi\leq M_2$ then 
\[
|z|\leq \Big(\frac{M_2}{\tau}+\frac{|z_Q|}{\tilde\phi(1,\theta-\theta_Q,r,z-z_Q)}\Big)
\frac{1}{(\frac{1}{\tilde\phi(1,\theta,s,z)}+\frac{1}{\tilde\phi(1,\theta-\theta_Q,r,z-z_Q)})}
\]
and we directly obtain, using the size of the support of the integrand $(z,\theta)$ (with $\theta$ bounded)
\[
|\mathcal{F}(u^{\#}_{0,he,h})(Q,Q_0,\tau)|\lesssim \frac{\tau^2}{\tilde\phi_1\tilde\phi_2}\frac{M_1+M_2}{\tau}\times \frac{\tilde\phi_1\tilde\phi_2}{\tilde\phi_1+\tilde\phi_2}\lesssim \frac 1t,
\]
where $\tilde\phi_1=\tilde\phi(1,\theta,s,z)$ and $\tilde\phi_2=\tilde\phi(1,\theta-\theta_Q,r,z-z_Q)$. Similar arguments hold for all $j\leq j(s,h)$.
\end{proof}
\begin{lemma}
When $\tau\times \Big(\frac{1}{\tilde\phi(1,\theta,s,z)}+\frac{1}{\tilde\phi(1,\theta-\theta_Q,r,z-z_Q)}\Big)\leq M$ estimate \eqref{estimF} still holds for $t\sim \tilde\phi(1,\theta_c-\theta_Q,r,z_c-z_Q)+\tilde\phi(1,\theta_c,s,z_c)$.
\end{lemma}
\begin{proof}
For $|z/z_c-1|\geq c$ we may proceed as in the second part of the proof of the previous lemma. Let therefore $z/z_c\in [1/4,4]$ and make the change of variables $z=z_c\Xi$. Then
\[
\tau\partial_\Xi\Phi=\tau z_c\partial_z\Phi|_{z=z_c\Xi}=\tau z_c\times \frac{z_Q}{\tilde\phi_2}(\Xi-1)=\tau z_Q^2\frac{\tilde\phi_1}{\tilde\phi_2}\frac{1}{\tilde\phi_1+\tilde\phi_2}(\Xi-1).
\] 
Using \eqref{derivzztrans}, we obtain 
$\tau\partial^2_\Xi\Phi|_{\Xi=1}=\tau z_c^2\partial^2_z\Phi|_{z=z_c\Xi}=\tau (\frac{1}{\tilde\phi_1}+\frac{1}{\tilde\phi_2}) z_c^2\frac{\psi_1^2}{\tilde\phi_1^2}$,
where, from the support properties of the symbol, the last factor is bounded from below by a fixed constant. If $\tau (\frac{1}{\tilde\phi_1}+\frac{1}{\tilde\phi_2}) z_c^2\geq M$, we apply the stationary phase near $\Xi=1$ only with respect to $\Xi$ (and not with $\theta$) as in the previous lemma and, using that $\theta$ belongs to a compact set, we find the following uniform bound 
\begin{equation}
|\mathcal{F}(u^{\#}_{0.he,h})(Q,Q_0,\tau)|\lesssim \frac{\tau^2}{\tilde\phi_1\tilde\phi_2}\times z_c\times \tau^{-1/2}z_c^{-1} (\frac{1}{\tilde\phi_1}+\frac{1}{\tilde\phi_2})^{-1/2}\lesssim \frac{\tau}{\tilde\phi_1+\tilde\phi_1}\times \Big(\tau (\frac{1}{\tilde\phi_1}+\frac{1}{\tilde\phi_2})\Big)^{1/2}
\end{equation}
and we conclude using the hypothesis $\tau\times \Big(\frac{1}{\tilde\phi(1,\theta,s,z)}+\frac{1}{\tilde\phi(1,\theta-\theta_Q,r,z-z_Q)}\Big)\leq M$. Same for all $j\leq j(s,h)$.
\end{proof}

\paragraph{\bf Same sign, $P$ in the illuminated regime of $Q_0,Q$} Consider now the situation $\frac{s\cos\theta-1}{\psi(s,\theta)}= \frac{r\cos(\theta-\theta_Q)-1}{\psi(r,\theta-\theta_Q)}$.
The formula \eqref{derivzztrans} remains unchanged and $\partial^2_{z,z}\Phi$ is strictly positive. 
Moreover, from the support condition of $\sigma_{1,he}$ we have $s\cos\theta>1$ then $r\cos(\theta-\theta_Q)>1$ and in \eqref{secondderiv} we obtain a lower bound for the sum of the first and third terms at the critical points as follows :
\[
\frac{s\cos\theta}{\tilde\phi(1,\theta,s,z)}+\frac{r\cos(\theta-\theta_Q)}{\tilde\phi(1,\theta-\theta_Q,r,z-z_Q)}\geq \frac{1}{{\tilde \phi(1,\theta,s,z)}}+\frac{1}{{\tilde\phi(1,\theta-\theta_Q,r,z-z_Q)}}
\]
and we can proceed exactly as in the previous case. 
\begin{rmq}
Notice that the positivity condition $s\cos\theta>1$ is equivalent to $\cos\theta>\cos\theta_*$, where $\theta_*=\arccos(1/s)$, which in turn assures that the point $P$ belongs to the illuminated region of $Q$ (as $\theta<\theta_*$). When both conditions hold ($\cos\theta>1/s$ and $\cos(\theta-\theta_Q)>1/r$), the point $P$ belongs to the illuminated regions from $Q_0$ and $Q$.
In fact, the line $Q_0Q$ is tangent to the boundary when $\arccos(1/s)+\arccos(1/r)=\theta_Q$ : if $P\in \partial\Omega$ is such that the cosine of the angle between $QO$ and $OP$ is larger than $1/r$, then the point $Q$ belongs to the illuminated regime of $Q_0$. As such, the previous case when $\pm (s\cos\theta-1)>0$ and $\pm(1-r\cos(\theta-\theta_Q))>0$ corresponds to points $P$ which belong to the illuminated regime of only one of the two points $Q_0$ and $Q$. In the last case $s\cos\theta-1<0$ and $1-r\cos(\theta-\theta_Q)<0$ that will be dealt with in the remaining of this section, $P$ does not belong to the illuminated regions of $Q_0,Q$. 
\end{rmq}

\paragraph{\bf Same sign, $P$ in the shadow regime of $Q_0,Q$} In this case we do not have a lower bound for the determinant of the Hessian matrix as before. Replacing $\frac{r\cos(\theta-\theta_Q)}{\tilde\phi_2}=-\frac{1}{\tilde\phi_1}+\frac{1}{\tilde\phi_2}+\frac{s\sin\theta}{\tilde\phi_1}$ in the expression \eqref{secondderiv} yields the following form for the determinant of the Hessian matrix at this critical point :
\begin{equation}\label{detHessmattrans2}
\Big(\frac{1}{\tilde\phi_1}+\frac{1}{\tilde\phi_2}\Big)\times \frac{(1-s\cos\theta)}{\tilde\phi_1}\times\Big|\Big(\frac{1}{\tilde\phi_1}+\frac{1}{\tilde\phi_2}\Big) \frac{(1-s\cos\theta)}{\tilde\phi_1}-2\frac{\psi_1^2}{\tilde\phi_1^2}\Big||_{\nabla_{\theta,z}\Phi=0}.
\end{equation}
\begin{lemma} 
At $\nabla_{\theta,z}\Phi=0$ the following holds
\[
\frac{\psi_2^2}{\tilde\phi_2^2}-\frac{1}{\tilde\phi_2} \frac{(1-r\cos(\theta-\theta_Q))}{\tilde\phi_2}=\frac{r(r-\cos(\theta-\theta_Q))}{\tilde\phi_2^2}\geq \frac{r(r-1)}{\tilde\phi_2^2}.
\]
\[
\frac{\psi_1^2}{\tilde\phi_1^2}-\frac{1}{\tilde\phi_1} \frac{(1-s\cos\theta)}{\tilde\phi_1}=\frac{s(s-\cos\theta)}{\tilde\phi_2^2}\geq \frac{s(s-1)}{\tilde\phi_1^2}.
\]
\end{lemma}
The lemma is a direct computation. Taking the sum of the terms in the left hand side and using that  $\tilde\phi_j=\psi_j\sqrt{1+\frac{z_Q^2}{(\psi_1+\psi_2)^2}} $ for $j\in\{1,2\}$ and
$ \frac{(1-s\cos\theta)}{\tilde\phi_1}= \frac{(1-r\cos(\theta-\theta_Q))}{\tilde\phi_2}$ yields, at $\nabla_{\theta,z}\Phi=0$,
\[
\Big(2\frac{\psi_1^2}{\tilde\phi_1^2}-\Big(\frac{1}{\tilde\phi_1}+\frac{1}{\tilde\phi_2}\Big) \frac{(1-s\cos\theta)}{\tilde\phi_1}\Big)
\geq  \frac{s(s-1)}{\tilde\phi_1^2}+\frac{r(r-1)}{\tilde\phi_2^2},
\]
which further induces the following lower bound for the determinant of the Hessian matrix 
\[
\Big(\frac{1}{\tilde\phi_1}+\frac{1}{\tilde\phi_2}\Big)\times \frac{(1-s\cos\theta)}{\tilde\phi_1}\times \Big( \frac{s(s-1)}{\tilde\phi_1^2}+\frac{r(r-1)}{\tilde\phi_2^2}\Big).
\]
From now on we may proceed as in the proof of the first case, applying the stationary phase when this determinant is sufficiently large and obtaining bounds using the size of the intervals of integration when the stationary phase fails to apply. Thus, we obtain the equivalent of Lemma \ref{lemphasstattrans}
\begin{lemma}
When $\tau\times \Big(\frac{1}{\tilde\phi(1,\theta,s,z)}+\frac{1}{\tilde\phi(1,\theta-\theta_Q,r,z-z_Q)}\Big)^{1/2}\Big( \frac{s(s-1)}{\tilde\phi_1^2}+\frac{r(r-1)}{\tilde\phi_2^2}\Big)^{1/2}\geq M$ for some large $M>1$, the usual stationary phase applies for $\theta,z$ near the critical points and yields
$|\mathcal{F}(u^{\#}_{0,he,h})(Q,Q_0,\tau)|\lesssim \frac{\tau}{t}$
for $t\sim \tilde\phi(1,\theta_c-\theta_Q,r,z_c-z_Q)+\tilde\phi(1,\theta_c,s,z_c)$. For $z,\theta$ outside a fixed neighborhood of the critical points the previous estimate still holds.
\end{lemma}
\begin{proof}
The main contribution of $\mathcal{F}(u^{\#}_{0,he,h})(Q,Q_0,\tau)$ after applying the stationary phase is bounded by
\[
\frac{\tau^2}{\tilde\phi_1\tilde\phi_2}\times \tau^{-1}\Big(\frac{1}{\tilde\phi_1}+\frac{1}{\tilde\phi_2}\Big)^{-1/2}\Big(\frac{s(s-1)}{\tilde\phi_1^2}+\frac{r(r-1)}{\tilde\phi_2^2}\Big)^{-1/2}\lesssim \frac{\tau}{\tilde\phi_1+\tilde\phi_2}\times \sqrt{\frac{1}{\tilde\phi_1}+\frac{1}{\tilde\phi_2}}\Big(\frac{\tilde\phi_1}{4s}+\frac{\tilde\phi_2}{4r}\Big).
\]
On the support of $\sigma_{1,he}$ we obtain the desired estimates. We let the other situations to the reader.
\end{proof}
When $1\leq j\leq j(s,h)$ the phase is the same, only the symbol comes with non positive powers of $2^{j}$ : to sum them up, notice that the phase is stationary in $\tau$ only when $t\sim |z|\sim 2^js$, hence for a finite number of $j$.
\end{proof}

\section{High-frequency case. Parametrix and dispersive estimates for $d(Q,\partial\Omega)<\sqrt{2}-1$ and $d(Q_0,\partial\Omega)< \sqrt{2}-1$, or for $d(Q,\partial\Omega)\geq \sqrt{2}-1$ and $\sqrt{1-\gamma^2}\sim 2^{-j}$ with $\tau 2^{-3j}d(Q,\partial\Omega)\lesssim 1$
}\label{secBesHan}
In this section both $Q$ and $Q_0$ are close to the boundary and $t>0$. For convenience, we will assume this time that $s\leq r\leq \sqrt{2}$.
Denote $\mathcal{R}(Q,Q_0,\tau)$ the outgoing solutions of the Helmholtz equation $(\tau^2+\Delta)w=\delta_{Q_0}$, $w_{|\partial\Omega}=0$ with $Q_0=(s,0,0)$ where we recall that, in cylindrical coordinates, $\Delta=\partial^2_r+\frac{\partial_r}{r}+\frac{\partial^2_{\theta}}{r^2}+\partial^2_z$. Then the solution of the wave equation with initial condition $(u_0,u_1)=(\delta_{Q_0},0)$ is given by
\begin{equation}
u(Q,Q_0,t)=\int_0^{\infty}e^{\ii t \tau} \mathcal{R}(Q,Q_0,\tau)\frac{d\tau}{\pi}.
\end{equation}
For a given $w_0(r,\theta,z)$, the solution to the \emph{inhomogeneous equation} $(\tau^2+\Delta)w=w_0$ reads as
\begin{equation*}
w(\tau,r,\theta,z)=\int_{\mathbb{R}}e^{\ii z\vartheta}\sum_{n\in\mathbb{Z}}e^{\ii n\theta}\int_{1}^{\infty}{G}_{n}(r,\tilde{r},\kappa(\vartheta,\tau))\tilde{r}^2\hat{w}_0(\tilde{r},n,\vartheta)d\tilde{r}d\vartheta,
\end{equation*}
where the kernel $G_n$ is symmetric w.r.t. $r,\tilde r$ and, for $r\geq \tilde r$, it is given by 
\begin{equation}\label{Gn}
\begin{aligned}
{G}_n(r,\tilde{r},\kappa)&=\frac{\pi}{2\ii}(r\tilde{r})^{-\frac12}\Big(J_n(\tilde{r}\kappa)-\frac{{J}_n(\kappa)}{H_n(\kappa)}H_n(\tilde{r}\kappa)\Big)H_n(r\kappa),\\
&=\frac{\pi}{4\ii}(r\tilde{r})^{-\frac12}\Big(\bar{H}_n(\tilde{r}\kappa)-\frac{\bar{H}_n(k)}{H_n(k)}H_n(\tilde{r}\kappa)\Big)H_n(r\kappa).
\end{aligned}
\end{equation}
Here $J_n(z)=\frac12(H_n(z)+\bar{H}_n(z))$ denotes the Bessel function and $\kappa(\vartheta,\tau):=\sqrt{\tau^2-\vartheta^2}$. As $n$ is an integer, $H_{-n}(z)=(-1)^nH_n(z)$, therefore $G_{n}=G_{-n}$.
Taking $w_0=\delta_{Q_0}$, $Q_0=(s,0,0)$ and $s\leq r$ yields $\tilde r=s$ and
\begin{equation*}
\mathcal{R}(Q,Q_0,\tau)=s^2\int_{\mathbb{R}}e^{\ii z\vartheta}\sum_{n\in\mathbb{Z}}e^{\ii n\theta}G_{|n|}(r,{s},\kappa(\vartheta,\tau))d\vartheta.
\end{equation*}
Let $\psi_0,\psi\in C^{\infty}_0$ as in Section \ref{sectWEOutSC} such that $\psi_0$ is equal to $1$ on $[1/81,1]$, and to $0$ on $[0, 1/100]$, $
\psi\in C^{\infty}_0(1/4,4)$ is equal to $1$ near $1$ is such that $1-\psi_0(\beta)=\sum_{j\geq1}\psi(2^{2j}\beta)$ and $0\leq \psi_0, \psi\leq 1$, and set
\[
\mathcal{R}_j(Q,Q_0,\tau)=s^2\int_{\mathbb{R}}e^{\ii z\vartheta}\sum_{n\in\mathbb{Z}}e^{\ii n\theta}\psi(2^{2j}(1-(\vartheta/\tau)^2))G_{|n|}(r,{s},\kappa(\vartheta,\tau))d\vartheta,
\]
for $j\geq 1$ ; for $j=0$, replace $\psi$ by $\psi_0(1-\gamma^2)$.
\begin{lemma}\label{lemI}
Fix $0<h_0<1$ small enough and let $h\leq h_0$. Let $\chi\in C^{\infty}_0(1/2,2)$ valued in $[0,1]$ and equal to $1$ on $[\frac34,\frac32]$. There exist a constant $C>0$ such that for all $1\leq s\leq r\leq \sqrt{2}$ and all $t>0$, we have 
\begin{equation}\label{estimI}
I(Q,Q_0,h):=\int_0^{\infty} e^{\ii t\tau}\chi(h\tau)\mathcal{R}(Q,Q_0,\tau) d\tau \leq\frac{C}{h^2t}.
\end{equation}
Moreover, for $j\geq j(r,h)$ with $j(r,h)$ defined in Definition \ref{defjjsh}, we have $\sum_{j\geq j(r,h)}I^j(Q,Q_0,\tau) \leq \frac{C}{h^2t}$, where
\begin{equation}\label{estimIj}
I^j(Q,Q_0,h):=\int_0^{\infty} e^{\ii t\tau}\chi(h\tau)\mathcal{R}_j(Q,Q_0,\tau) d\tau.
\end{equation}
\end{lemma}
In the remaining of this section we prove Lemma \ref{lemI}.
Let $\kappa=\kappa(\vartheta,\tau)=\sqrt{\tau^2-\vartheta^2}$ and set 
\begin{equation}\label{decompJ}
G^+_{n}(r,s,\kappa)=\frac{\pi}{4i\sqrt{rs}}\overline{H}_{n}(s\kappa)H_{n}(r\kappa), G^-_{n}(r,s,\kappa):=\frac{\pi}{4i \sqrt{rs}}\frac{\overline{H}_{n}(\kappa)}{H_{n}(\kappa)}H_{n}(s\kappa)H_{n}(r\kappa).
\end{equation}
Substitute \eqref{decompJ} in \eqref{Gn} and denote $\mathcal{R}^{\pm}$ and $I^{\pm}(Q,Q_0,\tau)$ the corresponding contributions, respectively, so that $I=I^++I^-$. Let $\chi_0 \in C^{\infty}_0(-2,2)$ valued in $[0,1]$ and equal to $1$ on $[-1,1]$ and $\chi_{\pm}{(\ell)}:=(1-\chi_0(\ell))1_{\pm\ell>0}$. Consider $n\neq 0$ and write $I^{\pm}=\sum_{*\in \{0,\pm\}}I^{\pm}_{\chi_*}$, where, for $\chi_*\in\{\chi_0,\chi_\pm\}$, $n\geq 1$ and $j\geq 0$ we define
\begin{equation}\label{chi12nj}
I^{\pm,n,j}_{\chi_*}:=\int_0^{\infty}e^{\ii t\tau}\chi(h\tau)\int e^{iz\vartheta }\chi_*\big((\frac{\sqrt{\tau^2-\vartheta^2}}{n}-1)/\varepsilon\big)s^2\psi(2^{2j}(1-(\vartheta/\tau)^2))G^{\pm}_{n}(r,s,\kappa(\vartheta,\tau))d\vartheta d\tau
\end{equation}
and set $I_{\chi_*}^{\pm}=\sum_{j\geq 0} \sum_{n\in\mathbb{N}\setminus\{0\}} (e^{in\theta}+e^{-in\theta}) I_{\chi_*}^{\pm,n,j}$ for some small $\varepsilon>0$. Then $I^j=2\sum_{*\in\{0,\pm\}}\sum_{n}\cos(n\theta)I^{\pm,n,j}_{\chi_*}$. Then $\sqrt{\tau^2-\vartheta^2}/n<1-\varepsilon$, $\sqrt{\tau^2-\vartheta^2}/n\in[1-2\varepsilon,1+2\varepsilon]$ and $\sqrt{\tau^2-\vartheta^2}/n>1+\varepsilon$ on the support of $\chi_-,\chi_0,\chi_+$.\\

In the following, we look for upper bounds for $|I^{\pm,n,j}_{\chi_*}|$ first when $s\leq r\leq \sqrt{2}$, then for $r\geq \sqrt{2}$ and $j\geq j(r,h)$ and check that the sums over $n, j$ remain bounded by $C/(h^2t)$ for some uniform constant $C>0$ independent of the parameters. We may assume that $n\geq n_0$ for some large $n_0$, as, for bounded values, the result is trivial. We start with the main part $I^{\pm}_{\chi_+}$ which corresponds to values $\rho:=(\sqrt{\tau^2-\vartheta^2})/n\geq 1+\varepsilon$. Let $\vartheta=\tau\gamma$ then $\rho=\tau\sqrt{1-\gamma^2}/n\geq 1+\varepsilon$. Let $\tilde \rho\in\{\rho, r\rho, s\rho\}$. With $\Phi_+$ given in Lemma \ref{lem:Phi+} we get from \eqref{Hnn} 
\begin{equation}\label{Hnnbetter}
H_n(n\tilde\rho)\sim_{1/n}2e^{-\frac{\ii\pi}{3}}\Big(\frac{4\tilde{\zeta}(\tilde\rho)}{1-\rho^2}\Big)^\frac{1}{4}n^{-\frac13}A_+(n^{\frac23}\tilde{\zeta}(\tilde\rho))
		\Big[\sum_{j\geq0}\Big(a_j +n^{-\frac43}\Phi_+(n^{\frac23}\tilde{\zeta}(\tilde\rho))b_j\Big)(-n^{\frac23}\tilde{\zeta}(\tilde\rho))^{-3j/2}\Big],
\end{equation}
\[
 A_+(n^{\frac23}\tilde{\zeta}(\tilde\rho)) \sim_{1/n} n^{-\frac16}\big(-\tilde{\zeta}(\tilde\rho)\big)^{-\frac14}e^{-\ii \frac23n(-\tilde{\zeta}(\tilde\rho))^{\frac32}}\Big(1+O \big((-n^{\frac23}\tilde{\zeta}(\tilde\rho))^{-1}\big)\Big), \text{ if } n^{\frac23}\tilde{\zeta}(\tilde\rho)>2.
\]
On the support of the cut-off functions in \eqref{chi12nj} for $*=+$, the symbol of $I^{\pm,n,j}_{\chi_+}$ becomes
\begin{equation}\label{Jsymb}
J^{\pm,n,j}_{\chi_+}(r,s,\rho):= n^{-2\times \frac13-2\times\frac16}\frac{s^2\chi_+((\tau\sqrt{1-\gamma^2}/n-1)/\varepsilon)}{(rs)^{\frac 12}((r\rho)^2-1)^{\frac14}((s\rho)^2-1))^{\frac14}}\Sigma_\pm(r,s,\rho,n)\tau\chi(h\tau)\psi(2^{2j}(1-\gamma^2)),
\end{equation}
where $\Sigma_{\pm}$ are asymptotic expansions with small parameter $n^{-1}$ and with main contribution obtained as a product of $a_0$ in \eqref{Hnn} and $\sigma_0$ in \eqref{eq:Phi+} hence elliptic.  
The phase functions of $I_{\chi_+}^{\pm,n,j}$, denoted $\phi^{\pm}_{n}$, read as
\begin{equation}\label{inout-phase}
\tau\phi^{\pm}_{n}:=t\tau+z\gamma \tau-n \Big(f_0(r,\rho)\mp f_0(s,\rho)\Big),\quad 
f_0(r,\rho):=\tfrac23 \big(-\tilde{\zeta}(r\rho)\big)^{\frac32}-
\tfrac23 \big(-\tilde{\zeta}(\rho)\big)^{\frac32},
\end{equation}
where we recall $\rho=\frac{\tau\sqrt{1-\gamma^2}}{n}$. 
The phases $\phi^{\pm}_{n}$ of $I^{\pm,n,j}_{\chi_+}$ are stationary when $\nabla_{\tau,\gamma}(\tau \phi^{\pm}_n)=0$, that is
\begin{equation}\label{phastattaugam}
\partial_{\tau}(\tau \phi^{\pm}_n) =t+z\gamma-\frac{n}{\tau}\big(f_1(r,\rho)\mp f_1(s,\rho)\big),\quad 
\tau\partial_{\gamma}\phi^{\pm}_n =\tau\big(z+\frac{\gamma}{\sqrt{1-\gamma^2}}\frac{(f_1(r,\rho)\mp f_1(s,\rho))}{\rho}\big),
\end{equation}
where $f_1(r,\rho):=\sqrt{(r\rho)^2-1}-\sqrt{\rho^2-1}$ and where the derivative of $f_0$ is obtained from Lemma \ref{lemzeta}.

\begin{lemma}\label{lemchi+}
There exists $C>0$ so that for all $\sqrt{2}\geq r\geq s\geq 1$ the following holds $\sum_{n\geq n_0, j\geq 0}|I^{\pm,n,j}_{\chi_+}|\leq \frac{C}{h^2 t}$. For $r\geq s$ with $r\geq \sqrt{2}$ and for $j(r,h)$ given in Definition \ref{defjjsh}
we also have $\sum_{n\geq n_0, j\geq j(r,h)}|I^{\pm,n,j}_{\chi_+}|\leq \frac{C}{h^2 t}$.
\end{lemma}
\begin{proof}
We focus on $I^{-,n,j}_{\chi_+}$. Let 
$\varphi:=2^j\sqrt{1-\gamma^2}$, then $\varphi\in (1/2,2)$
on the support of $\psi(2^{2j}(1-\gamma^2))=\psi(\varphi^2)$ and $|\gamma|\geq 1/4$ (when $j=0$ there is no need to change variables).
Let $\phi^-_{n,j}:=\phi^-_{n}|_{\gamma=\sqrt{1-2^{-2j}\varphi^2}}$ for $\varphi\sim 1$. As $r\geq s$ and $\rho\geq 1+\varepsilon$, the factor depending on $r,s,\rho$ in \eqref{I-chi+Msm+} is uniformly bounded by $1/\rho$. \\

Let first $1<s\leq r\leq \sqrt{2}$ and $t\sim |z|$. If $2^{-2j}|z|\gtrsim 1$ then $\tau|\partial_{\varphi}\phi^-_{n,j}|=\tau |\frac{\partial\gamma}{\partial\varphi}\partial_{\gamma}\phi^-_{n,j}|\sim \tau 2^{-2j}|z|\gtrsim 1/h$ : repeated integrations by parts yield $O(h^{N}(2^{-2j}|z|)^{-N})$ for all $N\geq 1$, hence for small $r$ we find
\begin{equation}\label{I-chi+Msm+}
|I^{-,n,j}_{\chi_+}(Q,Q_0,h)|\lesssim \frac{1}{h^2}\times \frac{2^{-2j}}{n}\times \frac{s^2}{(rs)^{1/2}}\frac{h^{N}(2^{-2j}|z|)^{-N}}{((r\rho )^2-1)^{1/4}((s\rho)^2-1)^{1/4}}.
\end{equation}
Take $N=1$, then $\sum_{n< 2^{-j}/h}\sum_{2^{2j}\lesssim |z|} 2^{-2j}/(n\rho) \times h(2^{2j}/|z|)\leq \frac{1}{|z|}\sum_{2^{2j}\lesssim |z|} 2^{j}\times 2^{-j}/h\leq \frac{h\log(1/h)}{|z|}$ where we used $(n\rho)^{-1}=2^j h$, $n< 2^{-j}/h$ and $j\leq \log_2(1/h)$. Same computation with $N\geq 1$ yields $\sum_{n,2^{2j}\lesssim t}|I^{-,n,j}_{\chi_+}|\lesssim \frac{O(h^{N})\log(1/h)}{h^2 t}$. For $2^{-2j}|z|\leq 1$ we have again, $|I^{-,n,j}_{\chi_+}|\lesssim \frac{1}{h^2}\frac{2^{-2j}}{n\rho}$ and  $\sum_{n< 2^{-j}/h}\sum_{2^{2j}\geq |z|} 2^{-2j}/(n\rho) \lesssim \sum_{2^{2j}\geq |z|}2^{-2j+j}h\times (2^{-j}/h)\leq \sum_{2^{2j}\geq |z|}2^{-2j}\leq 1/|z|\sim 1/t$. If $t/|z|\not\in [1/2,2]$, repeated integrations by parts in $\tau$ yield the same kind of bounds with additional factors $h^N$ for all $N\geq 1$. \\

Let now $r\geq \sqrt{2}$ and $s\leq r$ such that $\tau 2^{-3j}r\leq 1$; since the phase is stationary w.r.t. $\gamma $ when $|z|\sim 2^j r$, it follows that, if $\tau 2^{-4j}|z|\geq 4$, we may integrate by parts in $\varphi$ (in which case the remainders may be dealt with as before) to conclude. Let therefore $\tau 2^{-4j}|z|\leq 4$. 
We notice that when $t\geq 4(|z|+2^{j}r)$ the phase $\tau\phi^-_{n,j}$ is not stationary in $\tau$ :
 in this case we integrate by parts in $\tau$ and obtain an upper bound for $|I^{-,n,j}_{\chi_+}|$ of the form \eqref{I-chi+Msm+} but with $h^{N}(2^{-2j}|z|)^{-N}$ replaced by $(h/t)^N$.
For $N\geq 1$ gives $\sum_{n< 2^{-j}/h,j\geq j(r,h)}|I^{-,n,j}_{\chi_+}|\lesssim \frac{h^N}{h^2t}\sum_{n< 2^{-j}/h, j\leq \log(1/h)}\frac{2^{-2j}}{n\rho}\leq \frac{h^N\log(1/h)}{h^2t}$ (where we didn't use that $j\geq j(r,h)$).\\

Let $|z|\sim 2^j r$ and $t\leq 4(|z|+2^{j}r)\sim 4|z|$, then we have again $|I^{-,\chi_+,j}_{\chi_0,\chi_0,n}|\lesssim \frac{1}{h^2}\frac{2^{-2j}}{n\rho}$ 
and we are left to estimate the sum over $j\geq 1$ satisfying $\tau 2^{-4j}|z|,\tau 2^{-3j}r\lesssim 1$. If moreover $2^{-2j}|z|\leq 1$, we find 
\begin{equation}\label{Ichi+00nnst}
\sum_{1<2^{-j}/(nh),j\geq j(r,h)}|I^{-,n,j}_{\chi_+}|\lesssim \frac{1}{h^2}\sum_{n\leq 2^{-j}/h,2^{-2j} |z|\leq 1} \frac{2^{-2j}}{n}\times nh 2^j=\frac{1}{h^2}\sum_{2^{-2j}\leq 1/|z|}h 2^{-2j+j}\times \frac{2^{-j}}{h}\lesssim \frac{1}{h^2 t}.
\end{equation}
When $2^{-2j}|z|\gtrsim 1$ we bound from below $\frac{\tilde\tau}{h}\partial^2_{\varphi}\phi^-_{n,j}|_{\partial_{\varphi}\phi^-_{n,j}=0}\gtrsim \frac{2^{-2j}|z|}{h}$. The stationary phase yields 
\begin{equation}\label{I-chi+00n}
I^{-,n,j}_{\chi_+}=\frac{1}{h^2}\int e^{\frac ih \tilde\tau\phi^{-}_{n,j}}\chi(\tilde\tau)\frac{(\tilde\tau/h)^{-1/2}}{\sqrt{\partial^2_{\varphi}\phi^-_{n,j}}|_{\partial_{\varphi}\phi^-_{n,j}=0}}\Big( \tilde J^{-,n,j}_{\chi_+}(r,s,\frac{\tilde\tau}{2^j nh})+h2^{-j}O((2^{-2j}|z|/h)^{-\infty})\Big)
d\tilde\tau,
\end{equation}
where $\tilde J^{-,n,j}_{\chi_+}(r,s,\frac{\tau}{2^j n})$ is the symbol with main contribution $J^{-,n,j}_{\chi_+}$ introduced in \eqref{Jsymb} and where $h 2^{-j}$ comes from the factors $2^{-2j}\times\frac {1}{n}\times \frac{nh}{2^{-j}}$ of the symbol. 
In order to uniformly bound the sum of \eqref{I-chi+00n}, notice that the phase is stationary when $t\sim |z|\sim 2^jr$.
As $2^{-4j}|z|\lesssim h$, then
 $|z|^{1/2}\leq h^{1/2}2^{2j}$ and
\begin{equation}\label{I-chi+Msm}
|I^{-,n,j}_{\chi_+}|  \lesssim \frac{1}{h^2 t}\times \frac{h^{1/2}|z|^{1/2}2^{-2j+j}}{n}\times nh2^j
\leq  \frac{1}{h^2 t}\times h^{2}2^{2j},\quad  h^2\sum_{n_0\leq n, 2^{j}< 1/(h n)}2^{2j}\leq h\sum 2^j \lesssim,
\end{equation}
where we used that $ n\leq 2^{-j}/h$ on the support of $\chi_+$. Notice that the condition $j\geq j(r,h)$ was particularly useful here in order to obtain the sharp bounds in \eqref{I-chi+Msm}. 
In the same way one may deal with $I^{+,n,j}_{\chi_+}$ and obtain similar bounds. The proof of the Lemma is achieved. 
\end{proof}

Next, we turn to $I_{\chi_0}^{\pm,n,j}$ whose symbols are supported for $\frac{\tau\sqrt{1-\gamma^2}}{n}\in[1-2\varepsilon,1+2\varepsilon]$.  
For each $j\geq 1$, it will be convenient to take $\tau2^{-j}\varphi= n+n^{1/3}w$ : on the support of the symbol of $I^{\pm,n,j}_{\chi_0}$ we now have $
w n^{-2/3}\in[-2\varepsilon,+2\varepsilon]$ and $2^{-j}/h\sim n\geq 1$ as $\tau\sim 1/h$.
Write again $1=\sum_{*\in\{0,\pm\}}\chi_*(w)$ where $\chi_{\pm}(\ell)=(1-\chi_0)(\ell)1_{\pm \ell>0}$ and denote $I^{\pm,n,j}_{\chi_0,\chi_*}$ the corresponding integrals (defined as in \eqref{chi12nj} but with additional cutoffs $\chi_*(n^{2/3}(\frac{\sqrt{\tau^2-\vartheta^2}}{n}-1))$). We deal separately with the cases $w>1$, $|w|\leq 2$ and $w<-1$.
\begin{lemma}
For $1< s\leq r\leq \sqrt{2}$ we have $\sum_{n\geq n_0, j\geq 1}|I^{\pm,n,j}_{\chi_0,\chi_*}|\lesssim \frac{1}{h^2  t}$, $*\in\{0,+\}$. For $r\geq s$ with $r\geq\sqrt{2}$ and $j(r,h)$ as in Definition \ref{defjjsh} 
we have $\sum_{n\geq n_0, j\geq j(r,h)}|I^{\pm,n,j}_{\chi_0,\chi_*}|\lesssim \frac{1}{h^2 t}$, $*\in\{0,+\}$.
\end{lemma}
\begin{proof}
On the support of $\chi_{+}(w)$ we may proceed in a similar way as in Lemma \ref{lemchi+} as the same asymptotic expansions hold for the Hankel factors; as the computations are similar (modulo the change of variable w.r.t. $\tau$) we focus on $I_{\chi_0,\chi_0}^{-,n,j}$ with symbol $\chi_0(w)$.
The expansion \eqref{Hnnbetter} still holds (with the simpler form \eqref{Hnnv}): when $n^{2/3}(-\tilde\zeta(\tilde\rho))< 2$ (with $\tilde\rho\in \{\rho,r\rho, s\rho\}$), the Airy factors don't oscillate and may be brought into the symbol. 

Let first $1<s\leq r\leq \sqrt{2}$ when the last inequality holds.
The phase of $I^{-,n,j}_{\chi_0,\chi_0}$ equals $\tau(t+z\sqrt{1-2^{-2j}\varphi^2})$, and taking $\tau=\frac{2^{j}}{\varphi}(n+n^{1/3}w)$ we are reduced to obtaining uniform bounds for
\begin{equation}\label{I-chi-0sm}
\frac{s^2}{(rs)^{1/2}}\int e^{i\frac{2^{j}}{\varphi}(n+n^{1/3}w) (t+z\sqrt{1-2^{-2j}\varphi^2})}\chi(h\frac{2^{j}}{\varphi}(n+n^{1/3}w))\frac{2^{j}}{\varphi}(n+n^{1/3}w)\psi(\varphi) \chi_0(w) n^{-2/3+1/3}\frac{2^{-2j+j}}{\varphi}d\varphi dw,
\end{equation}
where the factors $2^{-2j+j}\frac{n^{1/3}}{\varphi}$ come from  $\gamma\rightarrow \varphi$, $\tau\rightarrow w$ and where $n\sim \frac{2^{-j}}{h}$, 
For $t\lesssim h^{-1/3}$, the sum of all contributions of the form \eqref{I-chi-0sm} may be bounded as follows
\begin{equation}\label{I-chi_0jsmall}
\sum_{j,n}|I^{-,n,j}_{\chi_0,\chi_0}|\leq \sum_{j,n\sim 2^{-j}/h} n^{2/3}\sim \sum_{j, h^{1/3}2^{j/3}\leq 1}(2^{-j}/h)^{5/3}\leq \frac{h^{1/3}}{h^2}\lesssim \frac{1}{h^2 t}.
\end{equation}
For $t\gtrsim h^{-1/3}$ satisfying $t\geq 2|z|$, the phase is non-stationary w.r.t. $w$ ; integrations by parts with the large parameter $2^j n^{1/3}\sim 2^{2j/3}/h^{1/3}$ yield a contribution $O((2^j n^{1/3}/|t|)^{-N})$ for all $N\geq 1$ and we conclude. For $h^{-1/3}\lesssim t\leq 4|z|$ we have $\frac{1}{|z|}\leq\frac{4}{t}$ and we apply the stationary phase in both $w,\varphi$: let $\phi^{-}_{0,n,j}:=\frac{2^{j}}{\varphi}(n+n^{1/3}w) (t+z\sqrt{1-2^{-2j}\varphi^2})$ then $\partial^2_{w}\phi^-_{0,n,j}=0$ and the determinant of the Hessian matrix equals $(\partial^2_{w,\varphi}\phi^-_{0,n,j})^2\sim (2^{-j}n^{1/3}|z|)^2$ for $\varphi\sim1 $. If $2^{-j}n^{1/3}|z|\geq h^{-\epsilon}$ for some small $\epsilon>0$, we find, for small $r,s$,
\[
\sum_{j, n\geq n_0}|I^{-,n,j}_{\chi_0,\chi_0}|\leq \sum_{j,n\sim 2^{-j}/h}\frac{n^{2/3}}{2^{-j}n^{1/3}|z|}\lesssim \sum_{j, 2^j<1/h}\frac{2^j}{|z|}\times (2^{-j}/h)^{4/3}(1+O(h^{\infty}))\lesssim \frac{h^{2/3}}{h^2 t}.
\]
 If $2^{-j}n^{1/3}|z|\leq 2h^{-\epsilon}$ then we bound the sum of $|I^{-,n,j}_{\chi_0,\chi_0}|$ as in \eqref{I-chi_0jsmall} by $\sum_{j,n\sim2^{-j}/h}n^{2/3}$ and use that $2^{-j/3}/h^{1/3}\sim n^{1/3}\leq 2^{j+1}h^{-\epsilon}/|z|$ which gives $\sum_{j,n\sim2^{-j}/h}n^{2/3}\leq \frac{1}{h^2|z|}h^{2/3-\epsilon}\sum_{j}2^{j+1-4j/3}\lesssim \frac{h^{2/3-\epsilon}}{h^2t}$.\\

Let now $r\geq \sqrt{2}$ such that $n^{2/3}(-\tilde\zeta(r\rho))>1$. If moreover $n^{2/3}(-\tilde\zeta(s\rho))>1$, then both Airy factors $A(n^{2/3}\tilde\zeta(r\rho))$, $A(n^{2/3}\tilde\zeta(s\rho))$ do oscillate and we may proceed as with $\chi_+(w)$ (the only differences with the case $\chi_+$ are the absence of the phase functions of $\overline{H}_n(n\rho)/H_n(n\rho)$, which means replacing $f_1(r,\rho)$ by $\sqrt{(r\rho)^2-1}$, and also the fact that the factor depending on $r,s,\rho$ in \eqref{I-chi+Msm} may not be bounded but at most $n^{1/3}$). Consider $n^{2/3}(-\tilde\zeta(s\rho))\leq 2$, then $|H_n(ns\rho)|\sim n^{-1/3}$ and the symbol of $I^{-,n,j}_{\chi_0,\chi_0}$ becomes
\[
J^{-,n,j}_{\chi_0,\chi_0}(r,s,\rho):= n^{-2\times \frac13-\frac16}\frac{s^2\Sigma_0(r,s,\rho,n)}{(rs)^{1/2}((r\rho)^2-1)^{\frac14}}\psi(\varphi)\chi_0(w)\chi(h\frac{2^{j}}{\varphi}(n+n^{1/3}w))\frac{2^{j}}{\varphi}(n+n^{1/3}w)2^{-2j}\frac{2^jn^{1/3}}{\varphi}
\]
where the elliptic symbol $\Sigma_0$ is an asymptotic expansion with small parameter $n^{-1}$ and with main contribution obtained as a product of $a_0$ in \eqref{Hnn}, $\sigma_0$ in \eqref{eq:Phi+} and $H_n(ns\rho)n^{-1/3}\times \frac{\overline{H}_n(n\rho)}{H_n(n\rho)}$.
The factor $((r\rho)^2-1)^{-1/4}$ is always bounded by $n^{1/6}$. The factors $2^{j}(n+n^{1/3}w)\times 2^{-2j}\times 2^{j}n^{1/3}/\varphi$ occur from the changes of variables $\vartheta\rightarrow \tau\gamma$, $\gamma\rightarrow \varphi$, $\tau\rightarrow w$. If $t\lesssim h^{-1/3}$ we conclude as in \eqref{I-chi_0jsmall}. 
Let $t\geq h^{-1/3}$. The phase $\phi^-_{0,n,j}:=\tau(t+z\sqrt{1-2^{-2j}\varphi^2})-\frac 23 n(-\tilde\zeta(r\rho))^{3/2}$
is not stationary 
for $z$ such that $\tau 2^{-j}|z|\sim 2^{j}n\times 2^{-2j}|z|\sim 2^{-j} n|z|\geq h^{-\epsilon}$ for some small $\epsilon>0$ and we perform repeated integrations by parts to conclude. If $2^{-j} n|z|\leq 2h^{-\epsilon}$, then for $|t|\geq 4|z|$ we integrate by parts, while for $|t|\leq 4|z|$ we use $2^{-j}n|z|\sim 2^{-2j}|z|/h\leq 2h^{-\epsilon}$, $2^{-j}/h\geq 1$ to obtain
\[
\sum_{j\geq 1,n\sim 2^{-j}/h}|I^{-,n,j}_{\chi_0,\chi_0}|\leq \sum_{n\sim 2^{-j}/h, 2^{-2j}\leq h^{1-\epsilon}/|z|} n^{2/3}\leq\frac{h^{1/3}}{h^2}\sum_{h^2\leq 2^{-2j}\leq 2h^{1-\epsilon}/|z|} 2^{-2j+j/3}<\frac{h^{1-\epsilon}}{h^2 t}.
\]
Let $r\geq \sqrt{2}$ and $j\geq j(r,h)$: for $s$ such that $n^{2/3}(-\tilde\zeta(s\rho))\leq 2$ we conclude as before (with an additional factor $1/r$ in the symbol). For $n^{2/3}(-\tilde\zeta(s\rho))\geq 1$, the situation is similar to the one of $\chi_+$ dealt with before.
\end{proof}
\begin{lemma}
For $1< s\leq r\leq \sqrt{2}$ we have $\sum_{n\geq n_0, j\geq 1}|\sum_{\pm}I^{\pm,n,j}_{\chi_0,\chi_-}|\lesssim \frac{1}{h^2 t}$. For $r\geq s$ with $r\geq \sqrt{2}$ and $j(r,h)$ given in Definition \ref{defjjsh}, we also have $\sum_{n\geq n_0, j\geq j(r,h)}|\sum_{\pm}I^{\pm,n,j}_{\chi_0,\chi_-}|\lesssim \frac{1}{h^2 t}$.
\end{lemma}
\begin{proof}
Recall that $1-\gamma^2=2^{-2j}\varphi^2$, with $\varphi\sim 1$ on the support of $\psi$, and $\rho=\tau\sqrt{1-\gamma^2}/n=1+n^{-2/3}w$: as $w<-1$ on the support the symbol of $I^{\pm,n,j}_{\chi_0,\chi_-}$ then $\rho\in [1-\varepsilon,1-n^{-2/3}]$. It will be convenient to use the representation of $G_n$ in terms of Bessel functions $J_n$ instead of $H_n$, hence the first line in \eqref{Gn}. We estimate
\begin{equation}\label{hen2}
\begin{aligned}
&\frac{s^2}{(rs)^{1/2}}\sum_{n\geq 1}e^{\ii n\theta}\sum_{j\geq 1}\int e^{\ii\frac{2^j(n+n^{1/3}w)}{\varphi}(t-z\sqrt{1-2^{-2j}\varphi^2})}\chi\Big(\frac{2^{j}h(n+n^{1/3}w)}{\varphi}\Big)\frac{J_{n}(n(1+n^{-\frac23}w))}{{H}_{n}(n(1+n^{-\frac23}w))}\\
&\times \psi(\varphi) \chi_-(w){H}_{n}(nr(1+n^{-\frac23}w)){H}_{n}(ns(1+n^{-\frac23}w))\frac{2^j}{\varphi} (n+n^{1/3}w)2^{-2j+j} n^{\frac13}dw d \varphi.
\end{aligned}
\end{equation}
The Bessel function $J_n(n\rho)$ is given by \eqref{Jnn}. The factor $J_n/H_n$ corresponds to the quotient $\frac{A}{A_+}(n^{\frac23}\tilde{\zeta}({\rho}))=e^{-2i\pi/3}+e^{-\frac43|n|\tilde{\zeta}(\rho)^{\frac32}}$ (see Lemma \ref{lem:Phi+}). On the support of the cut-offs of $I^{-,n,j}_{\chi_0,\chi_-}$, its symbol has the form
\begin{equation}\label{HnHn}
J^{-,n,j}_{\chi_0,\chi_-}: = \frac{s^2}{(rs)^{1/2}}n^{-2/3}
\Big(\frac{4\tilde\zeta(r\rho))}{1-(r\rho)^2}\Big)^{1/4}\Big(\frac{4\tilde\zeta(s\rho))}{1-(s\rho)^2}\Big)^{1/4}A_+(n^{2/3}\tilde\zeta(r\rho))A_+(n^{2/3}\tilde\zeta(s\rho))
e^{-\frac43n\tilde{\zeta}(\rho)^{\frac32}}\Sigma_-,
\end{equation}
for some symbol $\Sigma_-$ of order $0$.
In the case of $I^{+,n,j}_{\chi_0,\chi_-}$ one should replace $A_+(n^{2/3}\tilde\zeta(s\rho))$ by $A(n^{2/3}\tilde\zeta(s\rho))$ and remove the exponential decreasing factor.
When $n^{2/3}|\tilde\zeta(r\rho)|, n^{2/3}|\tilde\zeta(s\rho)|<2$ we can proceed exactly as for $I^{-,n,j}_{\chi_0,\chi_0}$ with $r-1,s-1\lesssim n^{-2/3}$ small. 
Assume $n^{2/3}\tilde\zeta(s\rho)\geq n^{2/3}\tilde\zeta(s\rho)\geq 1$ with $\rho\leq 1-\frac{1}{n^{2/3}}<1$, then
\begin{equation}\label{HnHn1}
J^{-,n,j}_{\chi_0,\chi_-}:=\frac{s^2}{(rs)^{1/2}}\frac{n^{-\frac13-\frac16}}{(1-(r\rho)^2)^{\frac14}}\frac{n^{-\frac13-\frac16}}{(1-(s\rho)^2)^{\frac14}}e^{-\frac 43n\tilde\zeta(\rho)^{\frac 32}+\frac23n\tilde{\zeta}(s\rho)^{\frac32}+\frac23n\tilde{\zeta}(r\rho)^{\frac32}}\Sigma_-.
\end{equation}
As we are assuming $\tilde{\zeta}(r\rho),\tilde\zeta(s\rho)>0$, we have, using Lemma \ref{lemzeta}, $\frac 23 \tilde\zeta(s\rho)^{3/2}-\frac 23 \tilde\zeta(\rho)^{3/2}=-\int_{\rho}^{s\rho}\frac{\sqrt{1-w^2}}{w}dw\leq 0$. The phase function of $I^{\pm,n,j}_{\chi_0,\chi_-}$ is $\tau(t+z\gamma)$ and the factor $\frac{s^2}{(rs)^{1/2}(1-(s\rho)^2)^{\frac14}(1-(r\rho)^2)^{\frac14}}$ is at most $n^{1/3}$ when $1-s\rho \sim 1-r\rho\sim n^{-2/3}$, while for $r,s\geq 2$ this term is uniformly bounded by $1$. From now one can proceed as in the case of $I^{-,n,j}_{\chi_0,\chi_0}$ as on the support of $\chi(h\tau)$ we still have $n\sim 2^{-j}/h$, the phase is stationary for $t\sim |z|$ and for $2^{-j} n |z|\geq h^{-\epsilon}$ we integrate by parts, while for $2^{-j}n|z|\leq 2h^{-\epsilon}$ we conclude as done previously. 
\end{proof}
\begin{lemma}
For $1< s\leq r\leq \sqrt{2}$ we have $\sum_{n\geq n_0, j\geq 1}|\sum_{\pm}I^{\pm,n,j}_{\chi_-}|\lesssim \frac{1}{h^2 t}$. For $r\geq s$ with $r\geq \sqrt{2}$ and $j(r,h)$ as in Definition \ref{defjjsh}, we also have $\sum_{n\geq n_0, j\geq j(r,h)}|\sum_{\pm}I^{\pm,n,j}_{\chi_-}|\lesssim \frac{1}{h^2 t}$.
\end{lemma}
\begin{proof}
On the support of $I_{\chi_-}^{-,n,j}$ we have $\rho=\frac{\tau\sqrt{1-\gamma^2}}{n}\leq 1-\varepsilon$.
The symbol of $I_{\chi_-}^{-,n,j}$ has also the form \eqref{HnHn}. For small $r,s\leq\sqrt{2}$ and $\varepsilon>2(\sqrt{2}-1)$, we write $1-r\rho=1-r+r(1-\rho)$ to deduce that, if $\rho\leq 1-\varepsilon$, then the symbol \eqref{HnHn} takes the form \eqref{HnHn1} where the factor $\frac{s^2}{(rs)^{1/2}(1-(s\rho)^2)^{\frac14}(1-(r\rho)^2)^{\frac14}}$ is uniformly bounded by a constant depending only on $\varepsilon$ and we conclude as before. When $r,s$ are large (and $\tau 2^{-2j}|z|\leq M$, $\tau 2^{-j} r\leq M$  for large $M>1$), we separate the possible situations : the only new one is the case $n^{2/3}(1-r\rho), n^{2/3}(1-s\rho)\geq 1$ and $r$ such that $r<1/\rho\leq 1/(1-\varepsilon)$, in which case $\frac{s^2}{(rs)^{1/2}(1-(s\rho)^2)^{\frac14}(1-(r\rho)^2)^{\frac14}}\leq r n^{1/3}\leq \frac{n^{1/3}}{(1-\varepsilon)}$. In this case we have additional decay from the exponential factors and conclude as before.
\end{proof}

\section{Small frequency case}
Let $\tau\leq 1/h_0$ for some fixed $h_0>0$, small enough. We use again the parametrix in terms of Bessel functions introduced in Section \ref{secBesHan} and keep the same notations. We split $I=I^++I^-$, and for $n\geq 1$ large enough, $I^{\pm}=\sum_{*\in \{0,\pm\}}I^{\pm}_{\chi_*}$, with $I^{\pm}_{\chi_*}$ introduced as a sum of $I^{\pm,n,j}_{\chi_*}$ given in \eqref{chi12nj} where $\chi(h\tau)$ is replaced by $\tilde\chi(\tau)$ supported for $\tau\leq 2/h_0$. Take $n_0=4/h_0$.
We aim at proving that $|\sum _{\pm}I^{\pm}_{\chi_*}|\lesssim C(h_0)/t$.
\begin{itemize}
\item On the support of $I^{\pm}_{\chi_*}$, $*\in \{+,0\}$, and for $n\geq n_0$ we have $n_0\leq n\leq \frac{\tau\sqrt{1-\gamma^2}}{1-\epsilon}<\frac{4}{h_0}=n_0$.
\item On the support of $I^{\pm}_{\chi_*}$, $*\in \{+,0\}$, and for $1\leq n\leq n_0$  as $\sqrt{1-\gamma^2}\sim 2^{-j}$ and $\tau\leq 2/h_0$, only a finite number of $j$ such that $2^j\leq 1/(h_0(1-\varepsilon))$ may contribute. For each $j,n$ on this finite set, the symbols of $I^{\pm}_{\chi_*}$ are bounded and their phase may oscillate only for large $t$ or large $|z|$. If $t$ is bounded then if $r$ or $|z|$ are larger than $\max\{4 t, M\}$ for some $M>1$ large enough, integrations by parts allow to conclude (using that the sum is finite); if $|z|,r\leq 4t$ each integral is bounded and we obtain $|I^{\pm}_{\chi_*}|\lesssim C(h_0)$. If $t$ be sufficiently large, then if $t/(|z|+2^j r)\notin [1/8,8]$, integrations by parts yields a contribution $O(1/t^N)$ for each pair $(j,n)$ on the support of $I^{\pm,n,j}_{\chi_*}$. If $t/(|z|+2^j r)\in [1/8,8]$, we separate the cases $2^{-2j}|z|\geq M$ for some large $M$, when we apply the stationary phase in $\varphi=2^{j}\sqrt{1-\gamma^2}$ and we conclude as in \eqref{I-chi+Msm} or $2^{-2j}|z|\leq M$, when we bound directly as in \eqref{Ichi+00nnst}.

\item On the support of $I^{\pm}_{\chi_-}$ we have $n\geq \tau\sqrt{1-\gamma^2}/(1-\varepsilon)$, hence the sum over $n$ is unbounded but as $n\gg \tau\sqrt{1-\gamma^2}$ we may use \eqref{larg_ord} and conclude. 
\end{itemize}

\section{Appendix}
 
 \subsection{Airy functions} \label{sectAA+}
For $w\in\C$, the Airy function is defined as follows : $A(w)=\frac{1}{2\pi}\int_{\R}e^{\ii(s^3/3+sw)}ds$. Let  $A_{\pm}(w):=A(e^{\mp2\ii\pi/3}w)$, then $A_-(w)=\bar{A}_+(\bar{w})$ and $A(w)=e^{\ii\pi/3}A_+(w)+e^{-\ii\pi/3}A_{-}(w)$. Moreover, $A_{\pm}(w),A'_{\pm}(w)$ are not zero for any $w\in\R$, while all the zeros of $A(w)$ and $A'(w)$ are real and non positive. 
 
We say that $f(w)$ admits an asymptotic expansion for $w\rightarrow 0$ if there exists $(c_j)_{j\in\N}$ such that for any $j\geq 0$ we have $\lim_{w\rightarrow 0}w^{-j-1}(f(w)-\sum_0^j c_iw^i)=c_{j+1}$. We  write $f(w)\sim_{w}\sum_j c_j w^j$.
\begin{lemma}\label{lem:Phi+}
Let $\Sigma(w):=(A_+(w)A_{-}(w))^{1/2}$, then $\Sigma(z)=|A_+(w)|=|A_{-}(w)|$ is real, monotonic increasing in $w$ and nowhere vanishing. We let $\mu(w):=\frac{1}{2\ii}\log(\frac{A_-(w)}{A_+(w)})$ for $w<-1/4$.
Then $A_{\pm}(w)=\Sigma(w)e^{\mp\ii\mu(w)}$. For $w<-1$, the following asymptotic expansions hold
\begin{equation}\label{eq:Phi+}
\Sigma(w)\sim_{\frac1w}(-w)^{-\frac14}\sum_{j\geq 0}\sigma_j(-w)^{-\frac{3j}{2}},\,\, \mu(w)\sim_{\frac1w}\frac{2}{3}(-w)^{\frac32}\sum_{j\geq 0}e_j(-w)^{-\frac{3j}{2}},\,\,\sigma_0=\frac{1}{2\sqrt{\pi}}, \,\,e_0=1.
\end{equation}
The Airy quotient $\Phi_{+}(w)=\frac{A_+'(w)}{A_+(w)}=\frac{\Sigma'(w)}{\Sigma(w)}-\ii\mu'(w)$ satisfies everywhere $\Phi_{+}'(w)=w-\Phi_+^2(w)$. In particular $\Phi'_+(w)$ is bounded on $(-\infty,-1)$ and 
$\Phi_+(w)\sim_{\frac1w}(-w)^{\frac12}\sum_{j\geq0}d_j(-w)^{-\frac{3j}{2}}$, $d_0=1$, for $(-w)>1$ large.

For $w> 1$, the functions $A_{\pm}(w)$ grow exponentially $A_{\pm}(w)=\Sigma_{\pm}(w)e^{\frac23w^{3/2}}$, where $\Sigma_{\pm}$ are classical symbols of order $-1/4$ and 
we have
$\frac{A_{-}(w)}{A_{+}(w)}+e^{2\ii\pi/3}=O(w^{-\infty})$ when $w\rightarrow \infty$ and $\frac{A_{-}(w)}{A_{+}(w)}\sim_{\frac1w} e^{2\ii\mu(w)}$ when $w\rightarrow -\infty$.
Moreover the Airy function $A(w)$ decays exponentially for $w> 1$, $A(w)\sim_{\frac1w} |w|^{-\frac14}e^{-\frac23 w^{3/2}}$.
\end{lemma}

\subsection{Bessel and Hankel functions}
The Hankel function $H_{\nu}(z)$ is a solution to the Bessel's equation $w^2H_\nu''(w)+wH_\nu'(w)+(w^2-\nu^2)=0$. The couple $\{H_\nu(w),\bar{H}_\nu(w)\}$ is a fundamental system of solutions for the Bessel equation. The real and imaginary part of $H_\nu(w)$, denoted $J_\nu(w)$ and $Y_\nu(w)$ respectively, are the usual Bessel function of the first and second type. The Hankel function of order $\nu$ is defined
by (\cite[(9.1.25)]{AS})
\begin{equation}\label{Hankdef}
H_{\nu}(w)=\int_{-\infty}^{+\infty-i\pi} e^{w\sinh t-\nu t}dt.
\end{equation}
For large positive order $\nu$ and $w=\nu\rho$, the Hankel functions have the following expansions that hold \emph{uniformly} with respect to $\rho$ in the sector $|\arg(\rho)|<\pi-\epsilon$, where $\epsilon>0$ is an arbitrary number \cite[(9.3.37)]{AS}:
\begin{equation}\label{Hnn}
H_\nu(\nu\rho)= 2e^{-\frac{\ii\pi}{3}}\Big(\tfrac{-4\tilde{\zeta}(\rho)}{\rho^2-1}\Big)^{\frac14}\Big(\nu^{-\frac13}A_+(\nu^{\frac23}\tilde{\zeta}(\rho))\big(\sum_{j\geq0}a_j(\tilde\zeta)\nu^{-2j}\big)+\nu^{-\frac53}A_+'(\nu^{\frac{2}{3}}\tilde{\zeta}(\rho))\big(\sum_{j\geq0}b_j(\tilde\zeta) \nu^{-2j}\big)\Big),
\end{equation}
\begin{equation}\label{Jnn}
J_\nu(\nu\rho)= 2e^{-\frac{\ii\pi}{3}}\Big(\tfrac{-4\tilde{\zeta}(\rho)}{\rho^2-1}\Big)^{\frac14}\Big(\nu^{-\frac13}A(\nu^{\frac23}\tilde{\zeta}(\rho))\big(\sum_{j\geq0}a_j(\tilde\zeta) \nu^{-2j}\big)+\nu^{-\frac53}A'(\nu^{\frac{2}{3}}\tilde{\zeta}(\rho))\big(\sum_{j\geq0}b_j(\tilde\zeta) \nu^{-2j}\big)\Big).
\end{equation}
Here $a_j(\tilde\zeta)$, $b_j(\tilde\zeta)$ are given in \cite[(9.3.40)]{AS} and $\tilde{\zeta}(\rho)$ is provided in Lemma \ref{lemzeta} (see \cite[(9.3.38),(9.3.39)]{AS}).
When $\rho= 1+\nu^{-2/3}v$, $v=O(1)$, $w=\nu\rho=\nu+\nu^{1/3}v$, these formulas reduce to (see \cite[(9.3.23),(9.3.24)]{AS}) 
\begin{gather}\label{Hnnv}
H_\nu(\nu+\nu^{1/3}v)= \frac{2^{1/3}}{\nu^{\frac13}}A_+(-2^{1/3}v)\big(1+\sum_{j\geq 1}\tilde a_j(v)\nu^{-2j/3}\big)+\frac{2^{2/3}}{\nu}A_+'(-2^{\frac{1}{3}}v)\big(\sum_{j\geq0}\tilde b_j(v) n^{-2j/3}\big),\\
%
J_\nu(\nu+\nu^{1/3}v)= \frac{2^{1/3}}{\nu^{\frac13}}A(-2^{1/3}v)\big(1+\sum_{j\geq 1}\tilde a_j(v)\nu^{-2j/3}\big)+\frac{2^{2/3}}{\nu}A'(-2^{\frac{1}{3}}v)\big(\sum_{j\geq0}\tilde b_j(v) n^{-2j/3}\big).
\end{gather}
where $\tilde a_j$, $\tilde b_j$ are polynomials in $v$ given in \cite[(9.3.25),(9.3.26)]{AS}.
\begin{rmq}
The formulas \eqref{Hnn}, \eqref{Jnn} are among the deepest and most important results in the theory of Bessel functions. In order to prove \eqref{Jnn} starting from \eqref{Hankdef} one may chose a suitable contour that yields $J_{\nu}(\nu\rho)=(2\pi)^{-1}\int e^{i\nu\phi(\rho,t)}dt$ with $\phi(\rho,t)=\rho\sin t-t$; for $\nu$ large enough and for $\rho>1$, the critical point $t(\rho):=\arccos(1/\rho))$ is real and the critical value equals $\phi(\rho,t(\rho))=\sqrt{\rho^2-1}-\arccos(1/\rho)=\frac 23(-\tilde\zeta)^{3/2}$, where $\tilde\zeta(\rho)$ is defined as in \eqref{tildezet>}. As the phase function of $A(\nu^{2/3}\tilde\zeta)$ equals $\nu(s^3+s\tilde\zeta)$ and has critical points $s^2=-\tilde\zeta$ and critical values $\pm \frac 23(-\tilde\zeta)^{3/2}$, one obtains \eqref{Jnn} by stationary phase (see \cite{ol74} for details).
\end{rmq}

When the order is much larger than the argument $n\gg w$, \eqref{Hnn}, \eqref{Jnn} reduce to (see \cite[(9.3.1)]{AS})\begin{equation}\label{larg_ord}
J_n(w)=\sqrt{\frac{1}{2\pi n}}\Big(\frac{ew}{2 n }\Big)^{n}\Big(1+O(\frac{|w|}{n})\Big),\quad Y_{n}(w)=-\sqrt{\frac{1}{2\pi n }}\Big(\frac{ew}{2 n }\Big)^{-n}\Big(1+O(\frac{|w|}{n})\Big), \quad n\gg 1.
\end{equation}
As we consider cylindrical coordinates we deal only with $\nu=n\in \ZZZ$ : in view of the well-known relations $H_{-n}(w)=(-1)^nH_n(w)$ (see \cite[(9.1.6)]{AS}), we may consider only non negative values of $n$ in our discussion. 

%


\bigskip
\begin{thebibliography}{12}

\bibitem{AS}
M.Abramowitz and I.A.Stegun. 
\newblock 
{\it Handbook of mathematical functions, with formulas, graphs, and mathematical tables.}. 
\newblock Edited by Milton Abramowitz and Irene A. Stegun. 
\newblock Dover publications Inc., New York, 1966.

\bibitem{CFU}
C.Chester, B.Friedman, F.Ursell.
\newblock 
{\it An extension of the method of steepest descents}. 
\newblock Proc. Cambridge Philos. Soc., 53:599-611,1957

\bibitem{give85}
J.Ginibre and G.Velo
\newblock
The global Cauchy problem for the nonlinear Schrödinger equation revisited
\newblock Ann. Inst. H. Poincaré Anal. Non Linéaire, 2(4):309-327, 1985

\bibitem{GV85}
J.~Ginibre and G.~Velo.
\newblock The global {C}auchy problem for the nonlinear {K}lein-{G}ordon
  equation.
\newblock {\em Math. Z.}, 189(4):487--505, 1985.

\bibitem{give95}
J.Ginibre and G.Velo.
\newblock
Generalized Strichartz inequalities for the wave equation
\newblock
Partial differential operators and mathematical physics (Holzhau, 1994), vol. 78 of Open Theory Adv. Appl. p 153-160, Birkhäuser, Basel, 1995


\bibitem{Hor3}
L. H\"ormander.
\newblock  The Analysis of Linear Partial Differential Operators III. 
\newblock Springer, 1994


\bibitem{IL}
O.Ivanovici and G.Lebeau.
\newblock  Dispersion for the wave and Schr\"odinger equations outside strictly convex obstacles and counterexamples. 
\newblock preprint https://arxiv.org/pdf/2012.08366.pdf


\bibitem{ls95}
H.Lindblad and Ch. Sogge.
\newblock
On existence and scattering with minimal regularity for semi-linear wave equation
\newblock
J.Funct. Anal., 130(2):357-426, 1995

\bibitem{lev90}
L.Kapitanski.
\newblock
Some generalizations of the Strichartz-Brenner inequality
\newblock
Algebra i Analiz, 1(3):127-159, 1989

\bibitem{keta98}
M.Keel and T.Tao.
\newblock
Endpoint Strichartz estimates
\newblock
Amer. J. Math., 120(5):955-980, 1998


\bibitem{Meas}
L. Meas.
\newblock{\it Dispersive estimates for the wave equation inside cylindrical convex domains }
\newblock{Comptes Rendus Math\'ematique}, vol.355, issue 2, 161--165 (2017) 
%
%
%
%
%

\bibitem{meta87}
R. Melrose, M.Taylor.
\newblock Boundary problems for the wave equations with grazing and gliding
  rays, 1987.


\bibitem{ol74}
F.Olver.
\newblock Asymptotics and Special functions. Academic Press, New York, 1974

\bibitem{sm98}
H.~F. Smith.
\newblock A parametrix construction for wave equations with {$C\sp {1,1}$}
  coefficients.
\newblock {\em Ann. Inst. Fourier (Grenoble)}, 48(3):797--835, 1998.

\bibitem{SmSo94Duke}
H.~F. Smith and Ch.~D. Sogge.
\newblock {$L^p$} regularity for the wave equation with strictly convex
  obstacles.
\newblock {\em Duke Math. J.}, 73(1):97--153, 1994.

\bibitem{smso95}
H.~F.Smith and Ch.~D.Sogge.
\newblock On the critical semilinear wave equation outside strictly convex
  obstacles.
\newblock {\em J. Amer. Math. Soc.}, 8(4):879--916, 1995.

\bibitem{stri77}
R.Strichartz.
\newblock Restrictions of Fourier transforms to quadratic surfaces and decay of solutions of wave equation
\newblock Duke Math. J. 44(3):705-714, 1977

\bibitem{tat02}
D.Tataru.
\newblock
Strichartz estimates for second order hyperbolic operators with non-smooth coefficients III
\newblock 
J.Amer.Math.Soc., 15(2):419-442 (electronic), 2002


\bibitem{Taylor2} M. Taylor.
\newblock{\it Partial Differential Equations II.}
\newblock Springer, 1996.


\bibitem{zw90}
M. Zworski.
\newblock High frequency scattering by a convex obstacle.
\newblock {\em Duke Math. J.}, 61(2):545--634, 1990.


\end{thebibliography}
\end{document}